\documentclass[11pt, twoside, open=right, pointlessnumbers, a4paper]{scrartcl}

\usepackage{amsmath}
\usepackage{titlesec}
\usepackage{wrapfig}
\usepackage{graphicx}
\usepackage{color}
\usepackage{xcolor, framed}
\usepackage{soul}
\usepackage[latin1]{inputenc}

\usepackage{multicol}
\usepackage{needspace}
\let\emph\undefined\let\itshape\undefined
\newcommand{\emph}[1]{\textsl{#1}}\let\itshape\slshape
\usepackage{tocloft}
\usepackage{amsmath}
\usepackage{amssymb}
\usepackage{amsfonts}
\usepackage{upgreek}
\usepackage{amsthm}
\usepackage{mathtools}
\usepackage{mathrsfs}
\usepackage{pifont}
\DeclareMathAlphabet{\mathpzc}{OT1}{pzc}{m}{it}
\usepackage{bbm}
\usepackage{bbding}
\usepackage{setspace}
\usepackage{mathdots}
\usepackage{relsize}
\usepackage{tocloft}

\numberwithin{equation}{section}
\usepackage{mathtools}
\mathtoolsset{showonlyrefs}
\usepackage{enumitem} 

\newenvironment{myenumerate}{\begin{enumerate}[topsep=2pt,parsep=2pt,partopsep=2pt,itemsep=0pt,label={\normalfont(\alph*)}]\itemsep0pt}{\end{enumerate}}
\newenvironment{xenumerate}{\begin{enumerate}[topsep=2pt,parsep=2pt,partopsep=2pt,itemsep=0pt,label={\normalfont(\arabic*)}]\itemsep0pt}{\end{enumerate}}
\newenvironment{pnum}{\begin{enumerate}[topsep=2pt,parsep=2pt,partopsep=2pt,itemsep=0pt,label={(\roman{*})}]}{\end{enumerate}}

\newtheoremstyle{style1}% name of the style to be used
  {13pt}% measure of space to leave above the theorem. E.g.: 3pt
  {13pt}% measure of space to leave below the theorem. E.g.: 3pt
  {}% name of font to use in the body of the theorem
  {}% measure of space to indent
  {\normalfont\bfseries}% name of head font
  {.}% punctuation between head and body
  {.5em}% space after theorem head; " " = normal interword space
  {}
\theoremstyle{style1}

\newtheorem{definition}{Definition}[section]
\newtheorem{example}[definition]{Example}
\newtheorem{remark}[definition]{Remark}
\newtheorem{preremarks}[definition]{Remarks}

\newenvironment{remarks}[1][]{\begin{preremarks} \begin{myenumerate} }{  \end{myenumerate} \end{preremarks} }

\newtheoremstyle{style2}% name of the style to be used
  {13pt}% measure of space to leave above the theorem. E.g.: 3pt
  {13pt}% measure of space to leave below the theorem. E.g.: 3pt
  {\slshape}% name of font to use in the body of the theorem
  {}% measure of space to indent
  {\normalfont\bfseries}% name of head font
  {.}% punctuation between head and body
  {.5em}% space after theorem head; " " = normal interword space
  {}

\theoremstyle{style2}

\newtheorem{lemma}[definition]{Lemma}
\newtheorem{theorem}[definition]{Theorem}
\newtheorem{proposition}[definition]{Proposition}
\newtheorem{corollary}[definition]{Corollary}

\makeatletter
\deffootnote{1em}{1em}{%
\textsuperscript{\thefootnotemark\ }}
\usepackage{scrpage2}
\addtokomafont{sectioning}{\rmfamily}
\setcapindent{2em}
\usepackage{hhline}
\usepackage{colortbl}
\usepackage{tabularx}
\usepackage{longtable}
\usepackage{pgf}
\usepackage{tikz}
\usepackage{tikz-cd}
\usepackage{fancybox}
\usetikzlibrary{arrows,shapes,trees,matrix}
\usetikzlibrary{decorations.markings}
\usetikzlibrary{knots}
\usepackage{pictexwd}
\usepackage{hyperref}

\newcommand{\myforall}{\quad \text{for all}\quad }
\newcommand{\overeqn}[2]{\stackrel{\text{\tiny\eqref{#1}}}{#2}}
\newcommand{\Sym}{\operatorname{Sym}}
\newcommand{\HSym}{\operatorname{HSym}}
\newcommand{\id}{\operatorname{id}}

\newcommand{\Hom}{\operatorname{Hom}}
\newcommand{\cat}[1]{\mathcal{#1}}
\newcommand{\catf}[1]{{\normalfont\textbf{#1}}}

\newcommand{\FinGrpd}{\catf{FinGrpd}}

\newcommand{\Vect}{\catf{Vect}}
\newcommand{\FinVect}{\catf{FinVect}}

\newcommand{\TRepGrpd}{\catf{2VecBunGrpd}}
\newcommand{\RepGrpdC}{\catf{VecBun}_{\mathbb{C}}\catf{Grpd}}

\newcommand{\Cob}{\catf{Cob}}

\newcommand{\Kan}{\catf{Kan}}

\newcommand{\Mod}{\catf{Mod}}

\newcommand{\TwoVect}{\catf{2Vect}}

\newcommand{\PBun}{\catf{PBun}}
\newcommand{\Par}{\operatorname{Par}}
\newcommand{\dir}[2]{to [out=#1, in=#2]}

\newcommand{\im}{\operatorname{i}}
\newcommand{\img}{\operatorname{im}}

\newcommand{\sphere}{\mathbb{S}}
\newcommand{\Com}{\operatorname{Com}}
\newcommand{\Aut}{\operatorname{Aut}}

\newcommand{\torus}{\mathbb{T}}
\let\Phi\undefined\DeclareMathSymbol{\Phi}{\mathalpha}{letters}{"08}
\let\Psi\undefined\DeclareMathSymbol{\Psi}{\mathalpha}{letters}{"09}
\let\Sigma\undefined\DeclareMathSymbol{\Sigma}{\mathalpha}{letters}{"06}
\let\Xi\undefined\DeclareMathSymbol{\Xi}{\mathalpha}{letters}{"04}
\let\Pi\undefined\DeclareMathSymbol{\Pi}{\mathalpha}{letters}{"05}
\let\Gamma\undefined\DeclareMathSymbol{\Gamma}{\mathalpha}{letters}{"00}
\let\Omega\undefined\DeclareMathSymbol{\Omega}{\mathalpha}{letters}{"0A}
\let\Lambda\undefined\DeclareMathSymbol{\Lambda}{\mathalpha}{letters}{"03}

\newcommand{\spaceplease}{\needspace{5\baselineskip}}

\let\to=\longrightarrow
\let\mapsto=\longmapsto

\newcommand{\Bimod}{\catf{Bimod}}

\newcommand{\Grpd}{\catf{Grpd}}

\usepackage[margin=2.3cm]{geometry}

\usepackage{epstopdf}

\automark{section}
\lohead{}
\rehead{}
\ihead{}
\ohead{}
\cftsetindents{section}{1.5em}{2.8em}
\cftsetindents{subsection}{4.3em}{3.2em}
\ofoot{}

\begin{document}\thispagestyle{empty}
	\setlength{\parskip}{0pt}
	\setlength{\abovedisplayskip}{7pt}
	\setlength{\belowdisplayskip}{7pt}
	\setlength{\abovedisplayshortskip}{0pt}
	\setlength{\belowdisplayshortskip}{0pt}

	\begin{flushright}
		\textsf{ZMP-HH/18-3}\\
		\textsf{Hamburger Beiträge zur Mathematik Nr. 723}\\
		\textsf{February 2018}
	\end{flushright}
	
	\vspace*{1cm}
	
	\begin{center}
		\Large \textbf{Extended Homotopy Quantum Field Theories} \\[0.5ex] \textbf{and their Orbifoldization} \normalsize \\[4ex] Christoph Schweigert and Lukas Woike
	\end{center}

	\begin{center}
		\emph{Fachbereich Mathematik,   Universität Hamburg}\\
		\emph{Bereich Algebra und Zahlentheorie}\\
		\emph{Bundesstra{\ss}e 55,   D -- 20 146  Hamburg}
	\end{center}

	\begin{abstract}
		\noindent \textbf{Abstract.} 
		Based on a detailed definition of \emph{extended} homotopy quantum field theories
		we develop a field-theoretic orbifold construction for these theories when the target space is the classifying space of a finite group $G$, i.e.\ for $G$-equivariant topological field theories.
		More precisely, we use a recently developed bicategorical version of the parallel section
		functor to  associate to an extended equivariant topological field theory an ordinary
	extended topological field theory. 
		One main motivation is the 3-2-1-dimensional case where our orbifold construction allows us to describe the orbifoldization of equivariant modular categories by a geometric construction.
		As an important ingredient of this result, we prove that a 3-2-1-dimensional $G$-equivariant topological field theory yields a $G$-multimodular category by evaluation on the circle. % -- hoping that this will lead to a classification of 3-2-1-dimensional $G$-equivariant topological field theories. 
		The orbifold construction is a special case of a pushforward operation along an arbitrary morphism of finite groups and provides a valuable tool for the construction of extended homotopy quantum field theories. 
	\end{abstract}\thispagestyle{empty}
		\thispagestyle{empty}
	\vspace*{1cm}
	\tableofcontents

\section{Introduction and summary}
Homotopy quantum field theories, as introduced in \cite{turaevhqftold} and further developed in the monograph \cite{turaevhqft}, are topological field theories defined on bordisms equipped with maps to a fixed target space. In the most investigated special case, this target is chosen to be aspherical, i.e.\ the classifying space of a (finite) group. Homotopy quantum field theories with such a choice of target space are called \emph{equivariant topological field theories}.

It is an interesting question whether equivariant topological field theories allow for an \emph{orbi\-foldization}, i.e.\ a construction which assigns to a given equivariant topological field theory a non-equivariant topological field theory, the \emph{orbifold theory}. Such an orbifold construction should be understood as a sum over twisted sectors combined with a computation of the invariants of the theory in the appropriate sense, see \cite{dvvvorbifold} for this perspective on orbifoldization, and e.g.\ \cite{fks, bantay98, bantay02, cuietal, evansgannon}
for the study of orbifold theories, in particular permutation orbifolds. 
There exists an alternative point of view which aims at associating
to a topological field theory with defects an orbifold theory
\cite{crsorbifold}.
The relation between these two types of constructions and a
description of generalized orbifolds \cite{ffrs} in terms of topological field theories
is beyond the scope
of this paper.

The field-theoretic orbifold construction presented in this paper provides insight into the relation of equivariant and non-equivariant field theories, but also has applications on a purely algebraic level: Topological field theories, both equivariant and non-equivariant ones, produce by evaluation on certain manifolds algebraic structures of independent interest for which sometimes orbifoldization procedures are known, e.g.\ for crossed Frobenius algebras  \cite{kaufmannorb,schweigertwoikeofk} or equivariant categories  \cite{kirrilovg04,centerofgradedfusioncategories}. An orbifold construction on the level of field theories provides a profitable and unifying geometric access to these concepts. In order to describe the orbifoldization of equivariant categories, it is necessary to consider \emph{extended} field theories.

Thus, in this paper we give an orbifold construction for extended equivariant topological field theories, where the specification \emph{extended} refers to the bicategorical nature of the theory. Also we focus on oriented theories although the construction does not depend on orientability. For a given finite group $G$, the construction takes as an input an extended $G$-equivariant topological field theory, i.e.\ a symmetric monoidal functor $Z : G\text{-}\Cob(n,n-1,n-2) \to \TwoVect$ from the symmetric monoidal bicategory $G\text{-}\Cob(n,n-1,n-2)$ of $n$-dimensional bordisms equipped with a map into $BG$ to the symmetric monoidal bicategory $\TwoVect$ of 2-vector spaces. The output of our construction is the orbifold theory $Z/G : \Cob(n,n-1,n-2) \to \TwoVect$, a non-equivariant topological field theory. Our orbifold construction $Z \mapsto Z/G$ lifts previous work \cite{schweigertwoikeofk} to the bicategorical setting. We proceed as follows: 

\begin{xenumerate}
	\item First we produce from the equivariant theory $Z$ a symmetric monoidal functor $\widehat{Z}:\Cob(n,n-1,n-2) \to \TRepGrpd$ from the cobordism category to the symmetric monoidal bicategory $\TRepGrpd$ built in \cite{swpar} from 2-vector bundles over essentially finite groupoids and (higher) spans of groupoids, see also \cite{haugseng} for related concepts. Hence, this step changes the coefficients of the theory from $\TwoVect$ to the more complicated coefficients $\TRepGrpd$ which, in exchange, now contain information about the equivariance. This step will be referred to as \emph{change to equivariant coefficients} and will be explained in Section~\ref{equivariantcoeffsec}. It produces examples for extended topological field theories with non-trivial coefficients (i.e.\ with a target category different from $\TwoVect$). 
	
	\item To produce topological field theories valued in $\TwoVect$, we need the symmetric monoidal parallel section functor
	\begin{align} \Par : \TRepGrpd \to \TwoVect\end{align} whose construction was the main result of \cite{swpar}. It takes (homotopy) invariants of 2-vector bundles and sends (higher) spans of groupoids to certain pull-push maps combined with (higher) intertwiners. To some extent, it makes the idea of the `Sum functor' in \cite{fhlt} precise. By restriction to the endomorphisms of the respective monoidal units one obtains the functor developed in \cite{trova}.

\end{xenumerate}
Now we can define the orbifold theory as the concatenation
\begin{align}
\frac{Z}{G}: \Cob(n) \stackrel{\widehat{Z}}{\to} \RepGrpdC \xrightarrow{\Par} \Vect_\mathbb{C}\ . 
\end{align}
The construction is functorial in $Z$, so the orbifoldization takes the form of a functor
\begin{align} \frac{?}{G} : \HSym(G\text{-}\Cob(n,n-1,n-2), \TwoVect) &\to \Sym(\Cob(n,n-1,n-2), \TwoVect)  \label{eqnofkfunctor} \\ Z &\mapsto \frac{Z}{G}\end{align}
from the 2-groupoid $\HSym(G\text{-}\Cob(n,n-1,n-2), \TwoVect)$ of extended $G$-equivariant topological field theories to the 2-groupoid  $\Sym(\Cob(n,n-1,n-2), \TwoVect)$ of extended topological field theories. An explicit description of the orbifold construction is given in Proposition~\ref{satzorbifoldconcrete}. In Section~\ref{secpushofk}, finally, we generalize the orbifold construction to a pushforward operation 
	\begin{align}
	\lambda_* : \HSym (G\text{-}\Cob(n,n-1,n-2),\TwoVect) \to \HSym ({H}\text{-}\Cob(n,n-1,n-2),\TwoVect)
	\end{align}
for equivariant topological field theories along a morphism $\lambda : G \to H$ of finite groups.

The main motivation for the generalization of our orbifold construction 
to \emph{extended} topological field theories
comes from the 3-2-1-dimensional case. Sections~\ref{equivmonstructure}-\ref{sectwist} concentrate on the category $\cat{C}^Z$ (more precisely: 2-vector space) obtained as the value of a 3-2-1-dimensional $G$-equivariant topological field theory on the circle. These sections can be read independently of the sections involving the orbifold construction. We prove that by evaluation of the field theory on the cylinder this category comes with the structure of a 2-vector bundle over the loop groupoid $G//G$ of $G$ and is, by evaluation on the pair of pants, endowed with an equivariant monoidal structure (Proposition~\ref{satzequivariantmonoidalstructure}). Moreover, $\cat{C}^Z$ has duals (Proposition~\ref{satzczduals}) and comes with a $G$-braiding (Proposition~\ref{satzequivbraiding}). The Dehn twist yields a $G$-twist (Proposition~\ref{satzgribbonkat}) which can be interpreted as a homotopical relaxation of the self-invariance of twisted sectors known from $G$-crossed Frobenius algebras (Remark~\ref{homrelaxbmk}). In Proposition~\ref{satzdescriptorbtensorstructure} we explicitly compute how the $G$-ribbon structure of $\cat{C}^Z$ behaves under the geometric orbifold construction, and in Theorem~\ref{thmorbifoldtheorymodular} we prove our main result on the relation of this orbifold structure  to the one obtained via the purely algebraic orbi\-foldization procedure in terms of orbifold categories \cite{kirrilovg04,centerofgradedfusioncategories}. \\[2ex]
\noindent {\textbf{Theorem~\ref{thmorbifoldtheorymodular}.} \itshape
	The square 	\begin{center}
	\begin{tikzpicture}[scale=2, implies/.style={double,double equal sign distance,-implies},
	dot/.style={shape=circle,fill=black,minimum size=2pt,
		inner sep=0pt,outer sep=2pt},]
	\node (A1) at (0,1) {$\substack{\text{3-2-1-dimensional } G\text{-equivariant} \\ \text{topological field theories} } $};
	\node (A2) at (4,1) {$\substack{\text{complex finitely semisimple} \\ G\text{-ribbon categories} } $};
	\node (B1) at (0,0) {$\substack{\text{3-2-1-dimensional} \\ \text{topological field theories} } $};
	\node (B2) at (4,0) {$\substack{\text{complex finitely semisimple} \\ \text{ribbon categories} } $};
	\path[->,font=\small]
	(A1) edge node[above]{$\substack{\text{evaluation} \\  \text{on the circle}}$} (A2)
	(A1) edge node[left]{$\substack{ \text{orbifoldization $?/G$  , see \eqref{eqnofkfunctor}} \\  \text{(geometric orbifoldization)}}$} (B1)
	(A2) edge node[right]{$\substack{ \text{orbifold category \cite{kirrilovg04}} \\  \text{(algebraic orbifoldization)}}$}  (B2)
	(B1) edge node[below]{$\substack{\text{evaluation} \\  \text{on the circle}}$} (B2);
	\end{tikzpicture}
	\end{center} commutes up to natural isomorphism.} \\[2ex] We make the following statements about the modularity of the categories appearing on the right hand side: We show that the category $\cat{C}^Z$ obtained from a 3-2-1-dimensional $G$-equivariant topological field theory via evaluation on the circle is $G$-modular if its monoidal unit is simple and thereby lift one of the main results of \cite{BDSPV153D} to the equivariant case (Theorem~\ref{thmgmodcatoncircle}, \ref{thmgmodcatoncirclea}). The proof makes explicit use of the interplay between the geometric and algebraic orbi\-foldization. In case the monoidal unit of $\cat{C}^Z$ is not simple, we prove that $\cat{C}^Z$ is $G$-multimodular (Theorem~\ref{thmgmodcatoncircle}, \ref{thmgmodcatoncircleb}), see Definition~\ref{defmultimodular} for notion of $G$-multimodularity. 
This provides a functor from 3-2-1-dimensional $G$-equivariant topological field theories to $G$-multimodular categories and hence a first step towards the classification of 3-2-1-dimensional $G$-equivariant topological field theories (Remark~\ref{remclass}). 

As a further application, our construction provides a uniform geometric formulation for the following two instances of orbifoldization:

\begin{itemize}
	\item In combination with the cover functor \cite{barmeierschweigert},
	 our orbifold construction yields permutation orbifolds \cite{fks,bantay98,bantay02,evansgannon}, see Example~\ref{expermutationorbifolds}.
	
	\item The orbifoldization of extended cohomological homotopy quantum field theories leads to the twisted Drinfeld doubles of a finite group from \cite{drp90}, as discussed in \cite{muellerwoike}.
	
\end{itemize}
Our construction ensures the existence of these orbifold theories as extended topological field theories and makes them explicitly computable. For example, we provide a formula for the number of simple objects of the orbifold theory (Theorem~\ref{thmnumberofsimpleobjectsorbifold}), which as a byproduct yields restrictions for manifold invariants coming from homotopy quantum field theories (Corollary~\ref{kororbifoldtheorymodular}).

\subparagraph*{Acknowledgments.}
We would like to thank Lukas Müller for his constant interest in this project and his numerous suggestions on a draft version of this article.
We are also grateful to Michael Müger and Alexis Virelizier for helpful discussions and to Marco Benini and Vincent Koppen for comments on the manuscript. 

CS is partially supported by the Collaborative Research Centre 676 ``Particles,
Strings and the Early Universe -- the Structure of Matter and Space-Time"
and by the RTG 1670 ``Mathematics inspired by String theory and Quantum
Field Theory''.
LW is supported by the RTG 1670 ``Mathematics inspired by String theory and Quantum
Field Theory''.

\subparagraph*{Conventions.}
All vector spaces or higher analogues thereof encountered in this article will be over the field of complex numbers. Therefore, we suppress the field in the notation and write $\Vect$ instead of $\Vect_\mathbb{C}$. Still, all constructions would also work over an algebraically closed field of characteristic zero.

We will refer to weak 2-functors between bicategories just as functors unless we want to stress the categorical level.

\section{The definition of extended equivariant topological field theories}
In this first section we develop a higher categorical version of the notion of a homotopy quantum field theory in \cite{turaevhqft}. By specializing to aspherical targets we obtain extended equivariant topological field theories. In the 3-2-1-dimensional case, equivariant topological field theories have also been defined in \cite{maiernikolausschweigerteq} using the language of principal fiber bundles. The present generalization to arbitrary dimension and target space seems to be new.

\subsection{Extended homotopy quantum field theories}
The definition of an extended homotopy quantum field theory requires a suitable symmetric monoidal bordism bicategory $T\text{-}\Cob(n,n-1,n-2)$ for an arbitrary target space $T$. It will generalize the symmetric monoidal bordism bicategory $\Cob(n,n-1,n-2)$ used as the domain of extended topological field theories, see e.g.\ \cite{schommerpries}, in the sense that all manifolds involved are additionally equipped with continuous maps to $T$.

For the definition of $T\text{-}\Cob(n,n-1,n-2)$ we need not only manifolds and manifolds with boundary, but also manifolds with corners whose definition we briefly recall, see also \cite[Section~3.1.1]{schommerpries}: An \emph{$n$-dimensional manifold with corners of codimension 2} is a second countable Hausdorff space $M$ together with a maximal atlas of charts of the form
\begin{align}
M\supseteq U  \stackrel{\varphi}{\to} V\subset \mathbb{R}^{n-2}\times (\mathbb{R}_{\geq 0})^2\ .
\end{align}
Given $x\in M$ we define the \emph{index} of $x$ to be the number of coordinates of $(\text{pr}_{(\mathbb{R}_{\geq 0})^2}\circ \varphi)(x)$ equal to 0 for some chart $\varphi$ (and hence for all charts).
The \emph{corners} are points of index 2.
A \emph{connected face} of $M$ is the closure of a maximal connected subset of points of index 1. A \emph{face} is the disjoint union of connected faces. A \emph{manifold with faces} is a manifold with corners such that every point of index 2 belongs to exactly two different connected faces.

Finally, an \emph{$n$-dimensional $\langle 2 \rangle $-manifold} is an $n$-dimensional manifold $M$ with faces together with a decomposition $\partial M = \partial_0 M \cup \partial_1 M$ of its topological boundary into faces such that $\partial_0 M \cap \partial_1 M$ is the set of corners of $M$. We call $\partial_0 M$ the 0-boundary of $M$ and $\partial_1 M$ the 1-boundary of $M$.

\begin{definition}[Bordism bicategory for arbitrary target space]\label{defmscbordcattarget}
	Let $n\ge 2$. For a non-empty topological space $T$, referred to as the \emph{target space}, the \emph{bicategory $T\text{-}\Cob(n,n-1,n-2)$ of bordisms with maps to $T$} is defined in the following way:
	\begin{xenumerate}
		\item[(0)] Objects, i.e.\ 0-cells, are pairs $(S,\xi)$, where $S$ is an $(n-2)$-dimensional oriented closed manifold and $\xi: S \to T$ is a map (by a map between topological spaces we always mean a continuous map).

		\item[(1)]
		A 1-morphism or 1-cell $(\Sigma,\varphi) : (S_0,\xi_0) \to (S_1,\xi_1)$ is an oriented compact collared bordism $(\Sigma,\chi_-,\chi_+) : S_0 \to S_1$, i.e.\ a compact oriented $(n-1)$-dimensional manifold $\Sigma$ with boundary together with orientation preserving diffeomorphisms $\chi_-: S_0\times [0,1) \to \Sigma_-$ and $\chi_+:  S_1\times (-1,0] \to \Sigma_+$, where $\Sigma_- \cup \Sigma_+$ is a collar of $\partial \Sigma$,
		and a continuous map $\varphi : \Sigma \to T$ such that the diagram 
		\begin{equation}
		\begin{tikzcd}
		\, & \Sigma \arrow{dd}{\varphi}  & \\
		S_0\times\{0\}  \arrow{ru}{\chi_-} \arrow[swap]{dr}{\xi_0} & & S_1\times \{0\}  \arrow[swap]{lu}{\chi_+} \arrow{dl}{\xi_1} \\
		& T  &
		\end{tikzcd}
		\end{equation}
		commutes. We do not assume any compatibility on the collars.  Composition of 1-morphisms is by gluing of bordisms along collars and maps, respectively. Note that the collars are necessary to define the composition. Identities are given by cylinders decorated with the homotopy which is constant along the cylinder axis.

		\item[(2)]
		A 2-morphism or 2-cell $(\Sigma,\varphi) \Longrightarrow (\Sigma',\varphi')$ between 1-morphisms $(S_0,\xi_0) \to (S_1,\xi_1)$ is an equivalence class of pairs $(M,\psi)$, where $M : \Sigma \to \Sigma'$ is an $n$-dimensional collared compact oriented bordism with corners and  $\psi : M \to T$ is a map.
		Here an $n$-dimensional collared compact oriented bordism with corners is a $\langle 2 \rangle$-manifold $M$ together with
		\begin{itemize}
			\item
			a decomposition of its 0-boundary $\partial_0 M = \partial_0 M_- \cup \partial_0 M_+$ and corresponding orientation preserving diffeomorphisms $\delta_- : \Sigma \times [0,1) \to M_-$ and $\delta_+ : \Sigma' \times (-1,0] \to M_+$ onto collars of this decomposition,
			
			\item
			a decomposition of its 1-boundary $\partial_1 M = \partial_1 M_- \cup \partial_1 M_+$ and corresponding orientation preserving diffeomorphisms $\alpha_- : S_0 \times [0,1)\times [0,1] \to M_-$ and $\alpha_+ : S_1 \times (-1,0]\times [0,1] \to M_+$ onto collars of this decomposition
			such that there is an $ \varepsilon >0$ and commutative triangles
			\begin{equation}
			\label{Condition Collars 1}
			\begin{tikzcd}
			S_0 \times [0,1)\times [0,\varepsilon) \ar{r}{\alpha_-}  \ar[swap]{rd}{\chi_-\times \id}& M & S_1 \times (-1,0]\times [0,\varepsilon) \ar[swap]{l}{\alpha_+} \ar{ld}{\chi_+\times \id} \\
			& \Sigma \times [0,\varepsilon) \ar[swap]{u}{\delta_-}&
			\end{tikzcd}
			\end{equation}
			and
			
			\begin{equation}
			\label{Condition Collars 2}
			\begin{tikzcd}
			\,          & \Sigma' \times (-\varepsilon,0] \ar{d}{\delta_+}   & \\
			S_0 \times [0,1)\times (1-\varepsilon,1] \ar{r}{\alpha_-} \ar{ru}{\chi'_-\times \id-1}& M & S_1 \times (-1,0]\times (1-\varepsilon,1] \ar[swap]{l}{\alpha_+} \ar[swap]{lu}{\chi'_+\times \id-1}
			\end{tikzcd}.
			\end{equation}
			Furthermore, we require the diagram
			\[
			\begin{tikzcd}
			\,        & M \ar{dd}{\psi} & \\
			S_0\times [0,1] \sqcup \Sigma  \ar{ru}{\alpha_- \sqcup \delta_-} \ar[swap]{dr}{\xi_0\circ \text{pr}_{S_0} \sqcup \varphi} & & S_1\times [0,1] \sqcup \Sigma' \ar[swap]{lu}{\alpha_+ \sqcup \delta_+} \ar{dl}{\xi_1\circ \text{pr}_{S_1} \sqcup \varphi'} \\
			& T &
			\end{tikzcd}
			\]
			to commute. Note again that we do not assume any compatibility on the collars.
		\end{itemize}

		Two such pairs $(M,\psi)$ and $(\widetilde M,\widetilde \psi)$ are defined to be equivalent if there is an orientation-preserving diffeomorphism $\Phi : M \to M$ making the diagram
		\[
		\begin{tikzcd}
		\,        & M \ar{dd}{\Phi} & \\
		\Sigma \times [0,1)  \ar{ru}{\delta_-} \ar[swap]{dr}{\widetilde \delta_-} & & \Sigma' \times (-1,0] \ar[swap]{lu}{\delta_+} \ar{dl}{\widetilde \delta_+} \\
		& \widetilde M &
		\end{tikzcd}
		\]
		and a similar diagram for the collars of the 1-boundary commute such that additionally $\psi = \widetilde \psi \circ \Phi$.

	\end{xenumerate}
	To define the vertical composition of 2-morphisms, we fix once and for all a diffeomorphism $[0,2]\to [0,1]$ which is the identity on a neighborhood of $0$ and near 2 given by $x\mapsto x-1$. Now the vertical composition is given by gluing using the collars of 0-boundaries. Furthermore, we can use the diffeomorphism fixed above to rescale the ingoing and outgoing 1-collars. As for 1-morphisms, there is no problem in gluing maps to $T$.
	
	Horizontal composition of 2-morphisms is defined by gluing manifolds and maps along 1-bounda\-ries. The new 0-collars can be constructed from the old ones by restricting them to $[0,\varepsilon)$ such that condition \eqref{Condition Collars 1} and \eqref{Condition Collars 2} ensure that we can glue them along the boundary and then rescaling the interval keeping a neighborhood of 0 fixed.

	Disjoint union endows the bicategory $T\text{-}\Cob(n,n-1,n-2)$ with the structure of a symmetric monoidal bicategory with duals.
\end{definition}

\begin{remarks} \label{bmkextcob}
	
	\item Following standard conventions, we will denote the composition of 1-mor\-phisms and 2-mor\-phisms from right to left by using the concatenation symbol $\circ$. Whenever we draw pictures of bordisms, however, composition has to be read from left to right. \label{bmkextcobcomp}
	
	\item To maintain readability, we will often suppress the collars in the notation.
	
	\item Consider a 1-morphism $(\Sigma,\varphi) : (S_0,\xi_0) \to (S_1,\xi_1)$, a compact collared bordism $\Sigma' : S_0 \to S_1$ and a diffeomorphism $\Phi : \Sigma \to \Sigma'$ preserving orientation and the collars. This data gives rise to an invertible 2-morphism $(M,\psi) : (\Sigma,\varphi) \to (\Sigma',\Phi_* \varphi := \varphi \circ \Phi^{-1})$ as follows: As the underlying compact collared bordism with corners $M$ we take the result of gluing $\Sigma \times [0,1]$ and $\Sigma' \times [0,1]$ via $\Phi$. Moreover, $\psi : M \to T$ is the map that $\varphi$ and $\Phi_* \varphi$ give rise to;
	for details on this mapping cylinder construction see \cite[Appendix~A.2]{muellerszabo}. \label{bmkextcobdiffeo}
	
\end{remarks}

 Having defined our bordism bicategory with target $T$ we are now ready to generalize the definition of a homotopy quantum field theory in \cite{turaevhqft}.

\begin{definition}[Extended homotopy quantum field theory]
	An \emph{$n$-dimensional extended homotopy quantum field theory with target space $T$ taking values in a symmetric monoidal bicategory $\cat{S}$} is a symmetric monoidal functor $Z: T\text{-}\Cob(n,n-1,n-2) \to \cat{S}$ satisfying the \emph{homotopy invariance property}: For two 2-morphisms $(M,\psi), (M,\psi') : (\Sigma_a,\varphi_a) \Longrightarrow (\Sigma_b,\varphi_b)$ between the 1-morphisms $(\Sigma_a,\varphi_a) , (\Sigma_b,\varphi_b) : (S_0,\xi_0) \to (S_1,\xi_1)$  with $\psi \simeq \psi'$ relative $\partial M$ we have the equality 
	\begin{align}\begin{array}{cc}
	\begin{tikzpicture}[scale=1.3, implies/.style={double,double equal sign distance,-implies},
	dot/.style={shape=circle,fill=black,minimum size=2pt,
		inner sep=0pt,outer sep=2pt},]
	\draw[line width=0.5pt,->]
	(0,2) node {\footnotesize $Z(S_0,\xi_0)$}
	(4,2) node {\footnotesize $Z(S_1,\xi_1)$}
	(5,2) node {$=$}
	(2,4.2) node {\footnotesize$Z(\Sigma_a,\varphi_a)$}
	(2,-0.2) node {\footnotesize $Z(\Sigma_b,\varphi_b)$}
	(0,2.6)   \dir{90}{90} (4,2.6) ;
	\draw[line width=0.5pt,->]
	(0,1.4)   \dir{270}{270} (4,1.5)
	;
	\draw (2,3.2) edge[implies] node[left] {\footnotesize $Z(M,\psi)$} (2,0.8);
	\end{tikzpicture}  & \begin{tikzpicture}[scale=1.3, implies/.style={double,double equal sign distance,-implies},
	dot/.style={shape=circle,fill=black,minimum size=2pt,
		inner sep=0pt,outer sep=2pt},]
	\draw[line width=0.5pt,->]
	(0,2) node {\footnotesize$Z(S_0,\xi_0)$}
	(4,2) node {\footnotesize$Z(S_1,\xi_1)$}
	(2,4.2) node {\footnotesize$Z(\Sigma_a,\varphi_a)$}
	(2,-0.2) node {\footnotesize$Z(\Sigma_b,\varphi_b)$}
	(0,2.6)   \dir{90}{90} (4,2.6) ;
	\draw[line width=0.5pt,->]
	(0,1.4)   \dir{270}{270} (4,1.5)
	;
	\draw (2,3.2) edge[implies] node[left] {\footnotesize$Z(M,\psi')$} (2,0.8);
	\end{tikzpicture}
	\end{array}\end{align}
	of 2-morphisms. We denote by $\HSym(T\text{-}\Cob(n,n-1,n-2),\cat{S})$ the bicategory of $n$-dimensional extended homotopy quantum field theories, i.e.\ the bicategory of homotopy invariant symmetric monoidal functors $T\text{-}\Cob(n,n-1,n-2) \to \cat{S}$. 
\end{definition}

\begin{remarks}\label{bmkexthqft}

\item 	This definition contains the appropriate bicategorical version of the homotopy invariance property in \cite{turaevhqft}. It is made in such a way that we recover the usual homotopy invariance property if we pass from extended homotopy quantum field theories to non-extended ones by restriction to the endomorphisms of the monoidal unit in both domain and codomain.

\item As for non-extended homotopy quantum field theories, the homotopy invariance can be built in by decorating the top-dimensional bordisms with relative homotopy classes of maps rather than actual maps. 
For technical reasons, however, we work with the above definition which requires homotopy invariance as an additional property just as in \cite{turaevhqft}.

\item The symmetric monoidal bicategory $\cat{S}$ which is the codomain of $Z$ will be referred to as the \emph{coefficients} or \emph{coefficient category of $Z$}.

\item The bicategory $\HSym(T\text{-}\Cob(n,n-1,n-2),\cat{S})$ is in fact a 2-groupoid. 

\item Let $Z$ be an $n$-dimensional extended homotopy quantum field theory, $\Sigma : S_0 \to S_1$ a 1-morphism in $\Cob(n,n-1,n-2)$ and $\varphi$ and $\varphi'$ two maps $\Sigma \to T$. Then for any homotopy $\varphi \stackrel{h}{\simeq} \varphi'$  relative $\partial \Sigma$ we obtain a 2-isomorphism $Z(h) : Z(\Sigma,\varphi) \Longrightarrow Z(\Sigma,\varphi')$ by evaluation of $Z$ on $\Sigma \times [0,1]$ equipped with $h$. Note that $Z(h)$ only depends on the equivalence class of the homotopy $h$.  \label{bmkexthqft4}

\end{remarks}

 For two topological spaces $X$ and $Y$ we denote by $Y^X$ the space of maps $X \to Y$ equipped with the compact-open topology. Depending on what is convenient, we can see $X$ and $Y$ and $Y^X$ also as Kan complexes. For any space or Kan complex $Z$ we denote by $\Pi(Z)=\Pi_1(Z)$ and $\Pi_2(Z)$ the fundamental groupoid and the fundamental 2-groupoid, respectively, and also set $\Pi_j(X,Y) := \Pi_j(Y^X)$ for $j=1,2$. We obtain as a straightforward generalization of \cite[Proposition~2.8]{schweigertwoikeofk}:

\begin{proposition}[]\label{satzexttftrep}
	For any extended homotopy quantum field theory $Z: T\text{-}\Cob(n,n-1,n-2) \to \cat{S}$ and any closed oriented $(n-2)$-dimensional manifold $S$ we naturally obtain a representation
	\begin{align} \widehat{Z} (S) := Z(S,?) : \Pi_2(S,T) = \Pi_2\left(T^S\right) &\to \cat{S}\ , \end{align} i.e.\ a 2-functor $\Pi_2\left(T^S\right) \to \cat{S}$ sending $\xi : S \to T$ to $Z(S,\xi)$. Definition on homotopies and homotopies of homotopies is by evaluation on the cylinder $S\times [0,1]$ over $S$ and the cylinder $S\times [0,1]^2$ over the cylinder over $S$.  
\end{proposition}

\subsection{Aspherical targets: Extended equivariant topological field theories}
Specifying for the target space the classifying space of a (finite) group leads to equivariant topological field theories, see also \cite{turaevhqft} for non-extended case. We can now provide the following analogue in the extended case:

\begin{definition}[Extended equivariant topological field theory]
	For a group $G$ let us set \begin{align}G\text{-}\Cob(n,n-1,n-2) := BG\text{-}\Cob(n,n-1,n-2)\end{align} for the classifying space $BG$ of $G$. An \emph{$n$-dimensional extended $G$-equivariant topological field theories with values in a symmetric bicategory $\cat{S}$} is a homotopy quantum field theory $Z : G\text{-}\Cob(n,n-1,n-2) \to \cat{S}$
	with target space $BG$ and values in $\cat{S}$.
\end{definition}

\begin{remarks}
	\item A $G$-equivariant topological field theory assigns data to manifolds decorated with maps to $BG$. Homotopy classes of such maps correspond to isomorphism classes of principal $G$-bundles, and in \cite[Remark~2.3 (d)]{schweigertwoikeofk} it is explained that this identification extends to groupoids, see also Lemma~\ref{lemmamappingspacebundles} below.
	\item A class of examples of extended $G$-equivariant topological field theories is constructed in \cite{maiernikolausschweigerteq}. 
\end{remarks}

In the sequel, it will be crucial to know the following basic fact about mapping spaces with aspherical target space, i.e.\ classifying space of a group (or more generally a groupoid):

\begin{lemma}[]\label{lemmamappingspacebundles}
	Let $\Gamma$ be a groupoid. For any space $X$ the mapping space $B\Gamma^X$ is equivalent to the nerve of the functor groupoid $[\Pi (X),\Gamma]$.
	In particular, for every (discrete) group $G$ and every manifold $M$ (with boundary or corners) the space $BG^M$ is equivalent to the nerve $B\PBun_G(M)$ of the groupoid $\PBun_G(M)$ of $G$-bundles over $M$.
\end{lemma}

\begin{proof}
	We can see $X$ as a Kan complex. Since the fundamental groupoid functor $\Pi : \Kan \to \Grpd$ from the category of Kan complexes to the category of groupoids is left-adjoint to the nerve functor $B: \Grpd \to \Kan$, we find for any Kan complex $X$
	\begin{align}
	\Hom_\Kan (Y, B\Gamma^X) &\cong \Hom_\Kan (Y \times X,B\Gamma) \\&\cong \Hom_\Grpd (\Pi(Y\times X),\Gamma)\\& \cong \Hom_\Grpd(\Pi(Y) \times \Pi(X),\Gamma) \\&\cong \Hom_\Grpd(\Pi(Y) , [\Pi(X),\Gamma]) \\ &\cong \Hom_\Kan (Y, B [\Pi(X),\Gamma])\ .
	\end{align} The Yoneda Lemma implies that $B\Gamma^X$ is isomorphic to the nerve $B [\Pi(X),\Gamma]$ of the groupoid of functors from $\Pi(X)$ to $\Gamma$.    
	The additional statement involving the groupoid of bundles now follows from the holonomy description of bundles, i.e.\ the fact that for any manifold $M$ (with boundary or corners) the groupoid $\PBun_G(M)$ is equivalent to $[\Pi(M),\star //G]$.  
\end{proof}

\begin{remark}
This result says that for an extended $G$-equivariant topological field theory $Z : G\text{-}\Cob(n,n-1,n-2) \to \cat{S}$ the representation $Z(S,?) : \Pi_2\left(S,BG\right) \to \cat{S}$ from Proposition~\ref{satzexttftrep} can and will be treated as a representation of the groupoid $\Pi(S,BG)$
or rather as a \emph{2-vector bundle over} $\Pi(S,BG)$ in the sense of \cite[Definition~2.6]{swpar}. This will turn out to be a tremendous simplification.\label{bmksimpaspherical}
\end{remark}

\begin{example}[The cover functor]\label{extotalspacefunctor}
	For a finite group $G$ there is a canonical symmetric monoidal functor
	\begin{align}
	\operatorname{Cov} : G\text{-}\Cob(n,n-1,n-2) \to \Cob(n,n-1,n-2),
	\end{align} the so-called \emph{cover functor}, which is studied in \cite{barmeierschweigert} and defined as follows: For a closed oriented $(n-2)$-dimensional manifold $S$ with a map $\xi : S \to BG$ we take the pullback bundle $\xi ^* EG \to S$. This $G$-bundle is a covering map and by \cite[Proposition~4.40 and 15.35]{lee} the total space $\xi^* EG$ inherits the structure of a closed oriented manifold of dimension $n-2$. The assignment $\operatorname{Cov}(S,\xi) := \xi^* EG$ extends to a symmetric monoidal functor. 
	If we are given an extended topological field theory $Z: \Cob(n,n-1,n-2) \to \TwoVect$, its pullback $\operatorname{Cov}^* Z$ along the cover functor is a $G$-equivariant topological field theory. This provides an important class of examples of $G$-equivariant field theories. In Example~\ref{expermutationorbifolds} we will use the cover functor to formalize the idea of the permutation orbifolds appearing in \cite{fks, bantay98,bantay02}. 
	\end{example}

\section{The orbifold theory of an extended equivariant topological field theory}
In this section we set up the orbifold construction for extended equivariant topological field theories. We only consider equivariant topological field theories with coefficients $\TwoVect$, the symmetric monoidal bicategory of 2-vector spaces, see e.g.\ \cite{mortonvec} or \cite[Example 2.8]{swpar}.
Its objects are 2-vector spaces, i.e.\ additive $\mathbb{C}$-linear finitely semisimple categories, its 1-morphisms are linear functors and its 2-morphisms are natural transformations.

\subsection{Change to equivariant coefficients\label{equivariantcoeffsec}}
Given a $G$-equivariant topological field theory $Z: G\text{-}\Cob(n,n-1,n-2) \to \TwoVect$ we will produce an ordinary topological field theory $\widehat{Z} : \Cob(n,n-1,n-2) \to \TRepGrpd$ which has coefficients in the symmetric monoidal bicategory $\TRepGrpd$ which is defined in detail in \cite[Section~4.2]{swpar}, see also \cite{haugseng} for related concepts, and whose definition we now recall briefly:

\begin{itemize}

\item[(0)] Objects are pairs $(\Gamma,\varrho)$, where $\Gamma$ is an essentially finite groupoid and $\varrho$ a 2-vector bundle over $\Gamma$, i.e.\ a representation $\varrho : \Gamma \to \TwoVect$. 

 \item[(1)] A 1-morphism $(\Gamma_0,\varrho_0) \to (\Gamma_1,\varrho_1)$ is a span
 		\begin{align} \Gamma_0 \stackrel{r_0}{\longleftarrow} \Lambda \stackrel{r_1}{\to} \Gamma_1 \end{align} of essentially finite groupoids and an intertwiner
 		\begin{align} \lambda : r_0^* \varrho_0 \to r_1^* \varrho_1\ .\end{align} 

\item[(2)] A 2-morphism between the 1-morphisms $(\Gamma_0,\varrho_0) \stackrel{r_0}{\longleftarrow} (\Lambda,\lambda) \stackrel{r_1}{\to} (\Gamma_1,\varrho_1)$ and $(\Gamma_0,\varrho_0) \stackrel{r_0'}{\longleftarrow} (\Lambda',\lambda') \stackrel{r_1'}{\to} (\Gamma_1,\varrho_1)$ is an equivalence of class  of 

\begin{itemize}
			\item a span of spans, i.e.\ a diagram
			\begin{center}
				\begin{tikzpicture}[scale=2,     implies/.style={double,double equal sign distance,-implies},
				dot/.style={shape=circle,fill=black,minimum size=2pt,
					inner sep=0pt,outer sep=2pt},]
				\node (A1) at (0,0) {$\Gamma_0$};
				\node (A2) at (2,1) {$\Lambda$};
				\node (A3) at (4,0) {$\Gamma_1$};
				\node (B2) at (2,-1) {$\Lambda'$};
				\node (C) at (2,0) {$\Omega$};
				\node (B1) at (1,-0.5) {$$};
				\node (B3) at (2,-1) {$$};
				%\node (B2) at (1,0) {$B\Omega$};
				\path[->,font=\scriptsize]
				(A2) edge node[above]{$r_0$} (A1)
				(A2) edge node[above]{$r_1$} (A3)
				(B2) edge node[below]{$r_0'$} (A1)
				(B2) edge node[below]{$r_1'$} (A3)
				%(A4) edge node[above]{$i$} (A5)
				(C) edge node[right]{$t'$} (B2)
				(C) edge node[right]{$t$} (A2);
				\draw (1.5,0) edge[implies] node[above] {\scriptsize$\alpha_0$} (0.5,0);
				\draw (2.5,0) edge[implies] node[above] {\scriptsize$\alpha_1$} (3.5,0);
				\end{tikzpicture}
			\end{center} in essentially finite groupoids commutative up to the indicated natural isomorphisms

						\item together with a natural morphism
						\begin{center}
							\begin{tikzpicture}[scale=1.5,     implies/.style={double,double equal sign distance,-implies},
							dot/.style={shape=circle,fill=black,minimum size=2pt,
								inner sep=0pt,outer sep=2pt},]
							\node (A1) at (0,1) {$(r_0t)^* \varrho_0=t^* r_0^* \varrho_0$};
							\node (A2) at (4,1) {$t^* r_1^* \varrho_1 = (r_1t)^* \varrho_1$};
							%\node (A5) at (4,1) {$E$};
							\node (B1) at (0,0) {$(r_0't')^* \varrho_0 = {t'}^* {r_0'}^* \varrho_0$};
							\node (B2) at (4,0) {${t'}^* {r_1'}^* \varrho_1 = (r_1't')^* \varrho_1$};
							\path[->,font=\scriptsize]
							(A1) edge node[above]{$t^* \lambda$} (A2)
							(A1) edge node[left]{$\varrho_0({\alpha_0})$} (B1)
							%(A4) edge node[above]{$i$} (A5)
							(B1) edge node[above]{${t'}^* \lambda'$} (B2)
							(A2) edge node[right]{$\varrho_1({\alpha_1})$} (B2);
							\draw (A2) edge[implies] node[above] {\scriptsize$\omega$} (B1);
							\end{tikzpicture}
						\end{center}

\end{itemize} of intertwiners.
\end{itemize}
We will refer to these coefficients as \emph{equivariant coefficients}. 

The following result generalizes the constructions in \cite[Section~3.3]{schweigertwoikeofk}:

\begin{theorem}[]\label{thmhatconstruction}
	For any finite group $G$ the assignment $Z\mapsto \widehat{Z}$ from Proposition~\ref{satzexttftrep} naturally extends to a functor\small 
	\begin{align}\label{changetoequivcoeffeqn}
		\widehat{?} : \HSym (G\text{-}\Cob(n,n-1,n-2),\TwoVect) \to \Sym(\Cob(n,n-1,n-2), \TRepGrpd)
	\end{align} \normalsize
	assigning to an extended $G$-equivariant topological field theory
	an extended topological field theory 
	with values in $\TRepGrpd$. We call \eqref{changetoequivcoeffeqn} the change to equivariant coefficients. 
\end{theorem}

\begin{proof}
	In the first step, we specify all the data needed to define $\widehat{Z}$ for an extended $G$-equivariant topological field theory $Z: G\text{-}\Cob(n,n-1,n-2) \to \TwoVect$:
	\begin{xenumerate}
		
		\item[(0)] To an object $S$ in $\Cob(n,n-1,n-2)$ we assign the 2-vector bundle $\widehat{Z}(S) : \Pi( S,BG) \to \TwoVect$ from Proposition~\ref{satzexttftrep} taking into account Remark~\ref{bmksimpaspherical}.
		
		\item[(1)] To a 1-morphism $\Sigma : S_0 \to S_1$ in $\Cob(n,n-1,n-2)$ we assign the span
		\begin{align}
			\Pi(S_0,BG) \stackrel{r_0}{\longleftarrow} \Pi(\Sigma,BG) \stackrel{r_1}{\to} \Pi(S_1,BG)
		\end{align} ($r_0$ and $r_1$ are the obvious restriction functors) and the intertwiner
		\begin{align}
			Z(\Sigma,?) : r_0^* \widehat{Z}(S_0) \to r_1^* \widehat{Z}(S_1)\label{eqnintertwinerZ} 
		\end{align} consisting of the map $Z(\Sigma,\varphi) : Z(S_0,\varphi|_{S_0}) \to  Z(S_1,\varphi|_{S_1})$ for each map $\varphi : \Sigma \to BG$ and natural isomorphisms
		\begin{align}\label{eqnatisohat}
		\begin{array}{c}	\begin{tikzpicture}[scale=2, implies/.style={double,double equal sign distance,-implies},
			dot/.style={shape=circle,fill=black,minimum size=2pt,
				inner sep=0pt,outer sep=2pt},]
			\node (A1) at (0,1) {$Z(S_0,\varphi|_{S_0})$};
			\node (A2) at (3,1) {$Z(S_1,\varphi|_{S_1})$};
			\node (B1) at (0,0) {$Z(S_0,\varphi'|_{S_0})$};
			\node (B2) at (3,0) {$Z(S_1,\varphi'|_{S_1})$};
			\path[->,font=\scriptsize]
			(A1) edge node[above]{$Z(\Sigma,\varphi)$} (A2)
			(A1) edge node[left]{$Z(S_0 \times I,h|_{S_0})$} (B1)
			(A2) edge node[right]{$Z(S_1 \times I,h|_{S_1})$} (B2)
			(B1) edge node[below]{$Z(\Sigma,\varphi')$} (B2);
			\draw (A2) edge[implies] node[above] {\scriptsize$\cong$} (B1);
			\end{tikzpicture}\end{array}
		\end{align} for every equivalence class $\varphi \stackrel{h}{\simeq}\varphi'$ of homotopies between maps $\varphi,\varphi' : \Sigma \to BG$.  
		These isomorphisms \eqref{eqnatisohat} are obtained as follows:
		First we will define an invertible 2-morphism\small
				\begin{align} ((S_1 \times I) \circ \Sigma \circ (S_0 \times I),  h|_{S_1} \cup \varphi \cup \id_{\varphi_0|_{S_0}}  )   \stackrel{\widehat{h}}{\Longrightarrow}   ((S_1 \times I) \circ \Sigma \circ (S_0 \times I), \id_{\varphi'|_{S_1}} \cup \varphi' \cup h|_{S_0}  ) \label{eqnhhat}
				\end{align} 
				\normalsize
				in $G\text{-}\Cob(n,n-1,n-2)$ and use it together with the functoriality of $Z$ to obtain the isomorphisms \eqref{eqnatisohat} as 
	\begin{align} Z(S_1 \times I,h|_{S_1}) \circ Z(\Sigma,\varphi) &\stackrel{\phantom{Z(\widehat{h})}}{\cong} Z(S_1 \times I,h|_{S_1}) \circ Z(\Sigma,\varphi) \circ Z(S_0 \times I,\id_{\varphi_0|_{S_0}})\\
			&\stackrel{\phantom{Z(\widehat{h})}}{\cong} Z((S_1 \times I) \circ \Sigma \circ (S_0 \times I),  h|_{S_1}  \cup \varphi \cup \id_{\varphi_0|_{S_0}} )\\
			&\stackrel{Z(\widehat{h})}{\cong} Z((S_1 \times I) \circ \Sigma \circ (S_0 \times I), \id_{\varphi'|_{S_1}} \cup \varphi' \cup h|_{S_0}  )\\
				&\stackrel{\phantom{Z(\widehat{h})}}{\cong}  Z(S_1 \times I,\id_{\varphi'|_{S_1}} ) \circ Z(\Sigma,\varphi') \circ Z(S_0 \times I,h|_{S_0})\\
					&\stackrel{\phantom{Z(\widehat{h})}}{\cong}  Z(S_1 \times I,\id_{\varphi'|_{S_1}} ) \circ Z(\Sigma,\varphi')\ .
		\end{align} 
		The needed 2-isomorphism \eqref{eqnhhat} will be obtained as a homotopy 
		\begin{align}
	h|_{S_1} \cup \varphi \cup \id_{\varphi_0|_{S_0}} \stackrel{\widehat{h}}{\simeq}	\id_{\varphi'|_{S_1}} \cup \varphi' \cup h|_{S_0}  : (S_1 \times I) \circ \Sigma \circ (S_0 \times I) \to BG
		\end{align} relative boundary, see also Remark~\ref{bmkexthqft}, \ref{bmkexthqft4} for this strategy. 
		For the definition of this homotopy we note that $h$ gives rise to homotopies
		\begin{align}
			\varphi|_{S_0} &\stackrel{h}{\simeq} h_t|_{S_0}\ ,\\
			h_t|_{S_1} &\stackrel{h}{\simeq} \varphi'|_{S_1}
		\end{align} for all $t\in I$, which by abuse of notation we just denote by $h$ again. Now the map $\widehat{h}_t : (S_0 \times I) \circ \Sigma \circ (S_1 \times I) \to BG$ is defined by gluing together $h_t$ and these two auxiliary homotopies as indicated in the picture
		\begin{center}
			\vspace*{0.5cm}
			\includegraphics[width=0.7\textwidth]{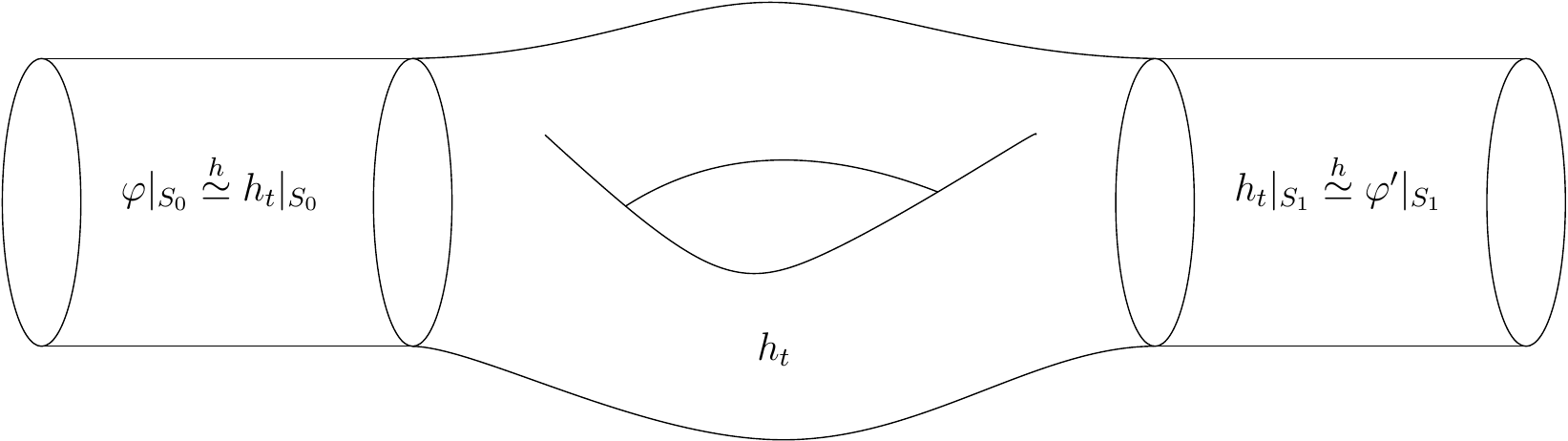},
			\vspace*{0.5cm}
		\end{center} in which we see $\Sigma$ with the cylinders over $S_0$ and $S_1$, respectively, glued to it. 
	A direct computation shows that the isomorphisms \eqref{eqnatisohat} are coherent.
		
		\item[(2)] To a 2-morphism $M: \Sigma_a \Longrightarrow \Sigma_b$ between 1-morphisms $\Sigma_a , \Sigma_b : S_0 \to S_1$ the functor $\widehat{Z}$ assigns the strict span of spans \begin{center}
			\begin{tikzpicture}[scale=2,     implies/.style={double,double equal sign distance,-implies},
			dot/.style={shape=circle,fill=black,minimum size=2pt,
				inner sep=0pt,outer sep=2pt},]
			\node (A1) at (0,0) {$\Pi(S_0,BG)$};
			\node (A2) at (2,1) {$\Pi(\Sigma,BG)$};
			\node (A3) at (4,0) {$\Pi(S_1,BG)$};
			\node (B2) at (2,-1) {$\Pi(\Sigma',BG)$};
			\node (C) at (2,0) {$\Pi(M,BG)$};
			\node (B1) at (1,-0.5) {$$};
			\node (B3) at (2,-1) {$$};
			\path[->,font=\scriptsize]
			(A2) edge node[above]{$r_0$} (A1)
			(A2) edge node[above]{$r_1$} (A3)
			(B2) edge node[below]{$r_0'$} (A1)
			(B2) edge node[below]{$r_1'$} (A3)
			(C) edge node[right]{$t'$} (B2)
			(C) edge node[right]{$t$} (A2);
			\end{tikzpicture}
		\end{center} coming from restriction of maps together with the map
		\begin{align} Z(M,?) : t^* Z(\Sigma,?) \to {t'}^* Z(\Sigma',?) \end{align} of intertwiners coming from evaluation of $Z$ on maps $M \to BG$. For this to be really a map of intertwiners we need to verify the condition given in \cite[Remark~2.7,~3.]{swpar}. Combining this with \cite[Remark~2.7,~2.]{swpar} we see that we need to prove that for every equivalence class $\psi \stackrel{h}{\simeq} \psi' :M \to BG$ of homotopies the 2-cell \small
		\begin{align}\begin{array}{c}
			\begin{tikzpicture}[scale=1.5, implies/.style={double,double equal sign distance,-implies},
			dot/.style={shape=circle,fill=black,minimum size=2pt,
				inner sep=0pt,outer sep=2pt},]
			\node (A1) at (0,2) {$Z(S_0,\psi'|_{S_0})$};
			\node (A2) at (3,2) {$Z(S_1,\psi'|_{S_1})$};
			\node (B1) at (0,0) {$Z(S_0,\psi|_{S_0})$};
			\node (B2) at (3,0) {$Z(S_1,\psi|_{S_1})$};
			\node (b1) at (0.2,-0.1) {};
			\node (b2) at (2.8,-0.1) {};
			\node (d1) at (0.2,0.1) {};
			\node (d2) at (2.8,0.1) {};
			\node (C1) at (0,-2) {$Z(S_0,\psi'|_{S_0})$};
			\node (C2) at (3,-2) {$Z(S_1,\psi'|_{S_1})$};
			\node (I1) at (1.5,1.7) {$\cong$};
			\node (I2) at (1.5,-1.7) {$\cong$};
			\node (I3) at (-1,0) {$\cong$};
			\node (I4) at (4,0) {$\cong$};
			\node (I5) at (-2.5,0) {$\id_{Z(S_0,\psi'|_{S_0})}$};
			\node (I6) at (5.5,0) {$\id_{Z(S_1,\psi'|_{S_1})}$};
			\node (M1) at (1.5,1.1) {\scriptsize$Z(\Sigma_a,\psi|_{\Sigma_a})$};
			\node (M2) at (1.5,-1.1) {\scriptsize$Z(\Sigma_b,\psi|_{\Sigma_b})$};
			\path[->,font=\scriptsize]
			(A1) edge node[above]{$Z(\Sigma_a,\psi'|_{\Sigma_a})$} (A2)
			(C1) edge node[below]{$Z(\Sigma_b,\psi'|_{\Sigma_b})$} (C2)
			(B1) edge node[left]{$Z(S_0 \times I,h|_{S_0})$} (A1)
			(B2) edge node[right]{$Z(S_1 \times I,h|_{S_1})$} (A2)
			(B1) edge node[left]{$Z(S_0 \times I,h|_{S_0})$} (C1)
			(B2) edge node[right]{$Z(S_1 \times I,h|_{S_1})$} (C2);
			\draw (1.5,0.8) edge[implies] node[right] {\scriptsize$Z(M,\psi)$} (1.5,-0.8);
			\draw[line width=0.5pt,->]
			(b1)   \dir{270}{270} (b2) 
			;
			\draw[line width=0.5pt,->]
			(d1)   \dir{90}{90} (d2) 
			;
			\draw[line width=0.5pt,->]
			(A1)   \dir{180}{180} (C1) 
			;
			\draw[line width=0.5pt,->]
			(A2)   \dir{0}{0} (C2) 
			;
			\end{tikzpicture}\end{array}\label{eqnhaton2}
		\end{align}\normalsize
		in $\TwoVect$, in which the 2-cells occupying the two squares in the middle block come from the definition of $\widehat{?}$ on 1-morphisms, is equal to $Z(M,\psi' ) : Z(\Sigma_a,\psi'|_{\Sigma_a}) \to  Z(\Sigma_b,\psi'|_{\Sigma_b})$. Indeed, this follows from homotopy invariance because \eqref{eqnhaton2} can be described by evaluation of $Z$ on a map on $M$ homotopic to $\psi'$ relative $\partial M$. 
		
	\end{xenumerate}
	In the next step, one needs to prove that $\widehat{Z}$ is a symmetric monoidal functor. The proof is a generalization of the proof of \cite[Theorem~3.9]{schweigertwoikeofk} and relies on the gluing property of the stack $\Pi(?,BG)$, see \cite[Section~3.3]{schweigertwoikeofk} for a review, and the fact that $Z$ is symmetric monoidal. Let us give a few more details: For two 1-morphisms $\Sigma : S_0 \to S_1$ and $\Sigma' : S_1 \to S_2$ in $\Cob(n,n-1,n-2)$ consider the diagram  
	\begin{center}
		\begin{tikzpicture}[scale=1.3,  implies/.style={double,double equal sign distance,-implies},
		dot/.style={shape=circle,fill=black,minimum size=2pt,
			inner sep=0pt,outer sep=2pt},]
		\node (O) at (4,4.5) {$\Pi(\Sigma'\circ \Sigma,BG)$};
		\node (A1) at (0,0) {$\Pi(S_0,BG)$};
		\node (A2) at (2,1) {$\Pi(\Sigma,BG)$};
		\node (A3) at (4,0) {$\Pi(S_1,BG)$};
		\node (C) at (4,0.2) {};
		\node (A4) at (6,1) {$\Pi(\Sigma',BG)$};
		\node (A5) at (8,0) {$\Pi(S_2,BG)$};
		\node (B2) at (4,2) {$\Pi(\Sigma,BG)\times_{\Pi(S_1,BG)} \Pi(\Sigma',BG)$};
		\path[->,font=\scriptsize]
		(O) edge node[left]{$R$} (B2)
		(O) edge node[left]{$s_0$} (A1)
		(O) edge node[right]{$s_2$} (A5)
		(A2) edge node[above]{$r_0$} (A1)
		(A2) edge node[above]{$r_1$} (A3)
		(A4) edge node[above]{$r_1'$} (A3)
		(A4) edge node[above]{$r_2'$} (A5)
		(B2) edge node[left]{$p\ $} (A2)
		(B2) edge node[right]{$\ \ \  p'$} (A4);
		\draw (A2) edge[implies] node[above] {\scriptsize$\eta$} (A4);
		\end{tikzpicture},
		\end{center}
	$r_0,r_1,r_1',r_2, s_0$ and $s_2$ are the restriction functors, the inner square is the homotopy pullback and $R$ also comes from restriction. The gluing property of $\Pi(?,BG)$ says that $R$ is an equivalence, which exhibits $\Pi(\Sigma'\circ \Sigma)$ as another model for the homotopy pullback (for this model the pullback square commutes strictly). 
	Now by \cite[Remark~4.3,~1.]{swpar} the composition $\widehat{Z}(\Sigma')\circ \widehat{Z}(\Sigma)$ is canonically 2-isomorphic to the 1-morphism in $\TRepGrpd$ with span part
	\begin{align}
	\Pi(S_0,BG) \xleftarrow{s_0} \Pi(\Sigma \circ \Sigma',BG) \xrightarrow{s_2} \Pi(S_2,BG)
	\end{align} and intertwiner $s_0^* \widehat{Z}(S_0) \to s_2^* \widehat{Z}(S_2)$ whose evaluation on $\varphi : \Sigma '\circ \Sigma \to BG$ is given by
	\begin{align}
	Z(\Sigma',\varphi|_{\Sigma'}) \circ Z(\Sigma,\varphi|_{\Sigma}) \cong Z(\Sigma'\circ \Sigma,\varphi)\ ,
	\end{align} where this last isomorphism is part of the data of $Z$. This gives us the needed isomorphism $\widehat{Z}(\Sigma')\circ \widehat{Z}(\Sigma) \cong \widehat{Z}(\Sigma'\circ \Sigma)$. 
	
	The proof of the strict preservation of vertical composition of 2-morphisms and the preservation of the horizontal composition of 2-morphisms up to the 2-isomorphisms for the composition of 1-morphisms just constructed proceeds in an analogous way using the gluing property of $\Pi(?,BG)$. 
	
	The symmetric monoidal structure comes from the additivity of $\Pi(?,BG)$ and the monoidal structure of $Z$. 
	
	Finally, we observe that $\widehat{Z}$ is functorial in $Z$.      
	
\end{proof}

\subsection{Definition and explicit description of the orbifold construction}
As outlined in the introduction, the orbifold construction for equivariant topological field theories
 is obtained by first changing to equivariant coefficients using Theorem~\ref{thmhatconstruction} and then applying the parallel section functor 
\begin{align} \Par : \TRepGrpd\to \TwoVect \end{align} from \cite[Theorem~4.9]{swpar} which extends taking parallel sections of 2-vector bundles \cite[Definition~2.10]{swpar} to a symmetric monoidal functor by means of pull-push constructions.  We are now ready to state the following central definition: 

\begin{definition}[Orbifold construction for extended $G$-equivariant topological field theories]\label{defofkext}
	Let $G$ be a finite group. Then the \emph{orbifold construction for extended $G$-equivariant topological field theories} is the functor
	\begin{align}
		\frac{?}{G} : \HSym (G\text{-}\Cob(n,n-1,n-2),\TwoVect) \to  \Sym(\Cob(n,n-1,n-2), \TwoVect)
	\end{align} from the 2-groupoid of extended $G$-equivariant topological field theories to the 2-groupoid of extended topological field theories defined as the concatenation
	\begin{center}	\begin{tikzpicture}[scale=2,     implies/.style={double,double equal sign distance,-implies},
		dot/.style={shape=circle,fill=black,minimum size=2pt,
			inner sep=0pt,outer sep=2pt},]
		\node (A1) at (0,1) {$\HSym (G\text{-}\Cob(n,n-1,n-2),\TwoVect)$};
		\node (A2) at (4,1) {$\Sym(\Cob(n,n-1,n-2), \TRepGrpd)$};
		\node (A3) at (4,0) {$\Sym(\Cob(n,n-1,n-2), \TwoVect)\ . $};
		\path[->,font=\scriptsize]
		(A1) edge node[above]{$\widehat{?}$} (A2)
		(A2) edge node[right]{$\Par _* =\Par \circ ? $} (A3)
		(A1) edge node[below]{$?/G$} (A3);
		\end{tikzpicture}
	\end{center}
	For an extended $G$-equivariant topological field theory $Z$, we call the extended topological field theory $Z/G$ the \emph{orbifold theory of $Z$}. 
\end{definition}

 From the prescriptions for the change of coefficients and the definition of the parallel section functor from \cite{swpar},
  it is straightforward to deduce the following explicit description of the orbifold construction which we are going to need in Section~\ref{3-2-1case}: 

\begin{proposition}[]\label{satzorbifoldconcrete}
	For any finite group $G$ and an extended $G$-equivariant topological field theory $Z: G\text{-}\Cob(n,n-1,n-2) \to \TwoVect$ the orbifold theory $Z/G: \Cob(n,n-1,n-2) \to \TwoVect$ admits the following description:
	\begin{myenumerate}
		\item To an object $S$ in $\Cob(n,n-1,n-2)$ the orbifold theory assigns the 2-vector space \begin{align} \frac{Z}{G}(S) = \Par \widehat{Z}(S)\end{align} of parallel sections of the 2-vector bundle $\widehat{Z}(S)=Z(S,?)$ over the groupoid $\Pi(S,BG)$, see Proposition~\ref{satzexttftrep}. \label{satzorbifoldconcretea} 
		
		\item To a 1-morphism $\Sigma : S _0 \to S_1$ in $\Cob(n,n-1,n-2)$ the orbifold theory assigns the 2-linear map (i.e.\ a linear functor)\label{satzorbifoldconcreteb} 
		\begin{align}
			\frac{Z}{G}(\Sigma) \ : \ \frac{Z}{G}(S_0) = \Par \widehat{Z}(S_0) \to \frac{Z}{G}(S_0) = \Par \widehat{Z}(S_0)
		\end{align} given by
		\begin{align}
			\left( \frac{Z}{G} (\Sigma) s \right) (\xi_1) = \lim_{(\varphi,h)\in r_1^{-1}[\xi_1]} Z(S_1 \times [0,1],h)Z(\Sigma,\varphi) s (\varphi|_{S_0}) \myforall &s \in \Par \widehat{Z}(S_0),\\ \quad &\xi_1 : S_1 \to BG, 
		\end{align} where $r_1 : \Pi(\Sigma,BG) \to \Pi(S_1,BG)$ is the restriction functor. 
		
		\item For a 2-morphism $M : \Sigma_a \Longrightarrow \Sigma_b$ between 1-morphisms $\Sigma_a,\Sigma_b : S_0 \to S_1$ in $\Cob(n,n-1,n-2)$ the value of the 2-morphism\label{satzorbifoldconcretec}
		\begin{align} \frac{Z}{G}(M)\ :\ \frac{Z}{G}(\Sigma_a) \to \frac{Z}{G}(\Sigma_b) \end{align} on $s \in \Par \widehat{Z}(S_0)$ and $\xi_1 : S_1 \to BG$ is given by the commutativity of the square\small
		\begin{center}
			\begin{tikzpicture}[scale=2, implies/.style={double,double equal sign distance,-implies},
			dot/.style={shape=circle,fill=black,minimum size=2pt,
				inner sep=0pt,outer sep=2pt},]
			\node (A1) at (0,1) {${\displaystyle \left(\frac{Z}{G}(\Sigma_a)s\right)(\xi_1)}$};
			\node (A2) at (4,1) {${\displaystyle \lim_{ \substack{ (\varphi_b, h_b, \psi, g) \\ \in (r_1^b)^{-1}[\xi_1] \times_{\Pi(\Sigma_b,BG)} \Pi(M,BG)}  }   Z(S_1 \times [0,1],h_b* g|_{S_1}) Z(\Sigma_a,\psi|_{\Sigma_a}) s(\psi|_{S_0})  }$};
			\node (B1) at (0,0) {${\displaystyle \left(\frac{Z}{G}(\Sigma_b)s\right)(\xi_1)}$};
			\node (B2) at (4,0) {${\displaystyle \lim_{ \substack{ (\varphi_b, h_b, \psi, g) \\ \in (r_1^b)^{-1}[\xi_1] \times_{\Pi(\Sigma_b,BG)} \Pi(M,BG)}  }   Z(S_1 \times [0,1],h_b) Z(\Sigma_b,\psi|_{\Sigma_b}) s(\psi|_{S_0})  }   \ ,$};
			\path[->,font=\scriptsize]
			(A1) edge node[above]{pull} (A2)
			(A1) edge node[left]{${\displaystyle \frac{Z}{G}(M)}$} (B1)
			(A2) edge node[right]{$Z(M,?)$} (B2)
			(B2) edge node[below]{push} (B1);
			\end{tikzpicture}
		\end{center}\normalsize
		where \begin{itemize}
			\item the pull map uses the pullback of limits along the functor $(r_1^b)^{-1}[\xi_1] \times_{\Pi(\Sigma_b,BG)} \Pi(M,BG) \to (r_1^a)^{-1}[\xi_1]$ defined using the universal property of the homotopy pullbacks involved,
			
			\item the map labelled with $Z(M,?)$ uses the vertex-wise transformation coming from the transformation $Z(\Sigma_a,\psi|_{\Sigma_a}) \xrightarrow{Z(M,\psi)} Z(\Sigma_b,\psi|_{\Sigma_b})$, the isomorphism \begin{align} Z(S_1 \times [0,1],g|_{S_1}) Z(\Sigma_b,\psi|_{\Sigma_b}) \cong Z(\Sigma_b,\psi|_{\Sigma_b}) Z(S_0 \times [0,1],g|_{S_0})\end{align} and the fact that $s$ is parallel,
			
			\item and the push map uses the pushforward of limits in 2-vector spaces, see \cite[Section~2.1]{swpar}, along the functor $(r_1^b)^{-1}[\xi_1] \times_{\Pi(\Sigma_b,BG)} \Pi(M,BG) \to (r_1^a)^{-1}[\xi_1]$ defined using the universal property of the homotopy pullbacks involved. 
			
		\end{itemize}
	\end{myenumerate}
\end{proposition}

 The orbifold construction for extended equivariant topological field theories generalizes previous work in \cite{schweigertwoikeofk}. Indeed, it can be compared with the orbifoldization in the non-extended case if we take into account that an extended field theory can be restricted to the endomorphisms of the empty set to obtain a non-extended field theory. Recalling how we generalized the change of coefficients in Theorem~\ref{thmhatconstruction} and the parallel section functor, see \cite{swpar} and in particular Proposition~4.8 therein, we obtain the following statement: 

\spaceplease 
\begin{proposition}[]
For any finite group $G$ the square
	\begin{center}
\begin{tikzpicture}[scale=2, implies/.style={double,double equal sign distance,-implies},
dot/.style={shape=circle,fill=black,minimum size=2pt,
	inner sep=0pt,outer sep=2pt},]
\node (A1) at (0,1) {$\HSym(G\text{-}\Cob(n,n-1,n-2), \TwoVect) $};
\node (A2) at (4,1) {$\Sym(\Cob(n,n-1,n-2), \TwoVect) $};
\node (B1) at (0,0) {$\HSym(G\text{-}\Cob(n,n-1), \Vect)$};
\node (B2) at (4,0) {$\Sym(\Cob(n,n-1), \Vect)$};
\path[->,font=\scriptsize]
(A1) edge node[above]{$?/G$} (A2)
(A1) edge node[left]{restriction} (B1)
(A2) edge node[right]{restriction} (B2)
(B1) edge node[below]{$?/G$} (B2);
\end{tikzpicture}
\end{center} commutes up to natural isomorphism of functors. The upper horizontal functor is the bicategorical orbifold construction in the extended case from Definition~\ref{defofkext}, the lower horizontal functor is the categorical orbifold construction in the non-extended case from \cite[Definition~4.3]{schweigertwoikeofk}. 
\end{proposition}

\subsection{Generalization of the orbifold construction to a pushforward along a group morphism\label{secpushofk}}
In \cite{schweigertwoikeofk} we generalized the non-extended orbifold construction to a push operation along an arbitrary morphism $\lambda : G \to H$ of finite groups in the sense that the orbifold construction corresponds to the pushforward along the morphism $G \to 1$ to the trivial group. This is also possible for the extended orbifold construction as we will sketch now: First denote by
\begin{align} \lambda_* : \Pi(?,BG) \to \Pi(?,BH) \end{align} the stack morphism induced by $\lambda$.
For an extended $G$-equivariant topological field theory $Z : G\text{-}\Cob(n,n-1,n-2) \to \TwoVect$ we would like to define a symmetric monoidal functor $\widehat{Z}^\lambda : H\text{-}\Cob(n,n-1,n-2) \to \TRepGrpd$.
To an object $(S,\xi)$ in $H\text{-}\Cob(n,n-1,n-2)$, i.e.\ an $(n-2)$-dimensional closed oriented manifold together with a map $\xi : S \to BH$ it assigns the pullback $q^* Z(S,?)$ of the 2-vector bundle $Z(S,? ) : \Pi(S,BG) \to \TwoVect$ along the functor $q: \lambda_*^{-1}[\xi] \to \Pi(S,BG)$ featuring in the defining square of the homotopy fiber
\begin{center}
		\begin{tikzpicture}[scale=1.5, implies/.style={double,double equal sign distance,-implies},
		dot/.style={shape=circle,fill=black,minimum size=2pt,
			inner sep=0pt,outer sep=2pt},]
		\node (A1) at (0,1) {$\lambda_*^{-1}[\xi]$};
		\node (A2) at (2,1) {$\Pi(S,BG)$};
		\node (B1) at (0,0) {$\star$};
		\node (B2) at (2,0) {$\Pi (S,BH)$};
		\path[->,font=\scriptsize]
		(A1) edge node[above]{$q$} (A2)
		(A1) edge node[left]{} (B1)
		(A2) edge node[right]{$\lambda_*$} (B2)
		(B1) edge node[below]{$\xi$} (B2);
		\end{tikzpicture}
	\end{center} of $\lambda_* : \Pi(S,BG) \to \Pi(S,BH)$ over $\xi$. On 1-morphisms and 2-morphisms one straightforwardly generalizes the assignments made in \cite[Section~4]{schweigertwoikeofk} to obtain $\widehat{Z}^\lambda$. The construction is obviously functorial in $Z$, so we obtain the following result:
	
	\begin{proposition}[]
	For any morphism $\lambda : G \to H$ of finite groups the assignment $Z\mapsto \widehat{Z}^\lambda$ extends to a functor
	\begin{align}
	\widehat{?}^\lambda : \HSym ({G}\text{-}\Cob(n,n-1,n-2),\TwoVect) \to \HSym ({H}\text{-}\Cob(n,n-1,n-2),\TRepGrpd).
	\end{align}
	\end{proposition}
	
	\begin{definition}[]\label{defpushforward}
	For a morphism $\lambda : G \to H$ of finite groups we define the \emph{pushforward of $G$-equivariant topological field theories}
	\begin{align}
	\lambda_* : \HSym (G\text{-}\Cob(n,n-1,n-2),\TwoVect) \to \HSym ({H}\text{-}\Cob(n,n-1,n-2),\TwoVect)\label{eqnpushfunctorgrp}
	\end{align}
	 along $\lambda$ as the concatenation
	\begin{center}	\begin{tikzpicture}[scale=2,     implies/.style={double,double equal sign distance,-implies},
			dot/.style={shape=circle,fill=black,minimum size=2pt,
				inner sep=0pt,outer sep=2pt},]
			\node (A1) at (0,1) {$\HSym (G\text{-}\Cob(n,n-1,n-2),\TwoVect)$};
			\node (A2) at (4,1) {$\HSym ({H}\text{-}\Cob(n,n-1,n-2),\TRepGrpd)$};
			\node (A3) at (4,0) {$\HSym ({H}\text{-}\Cob(n,n-1,n-2),\TwoVect)$};
			\path[->,font=\scriptsize]
			(A1) edge node[above]{$\widehat{?}^\lambda$} (A2)
			(A2) edge node[right]{$\Par _*$} (A3)
			(A1) edge node[below]{$\lambda_*$} (A3);
			\end{tikzpicture}.
		\end{center}
		\end{definition}
		This generalizes the pushforward construction given in \cite{schweigertwoikeofk}. 
		Again, the orbifold construction can be identified with the pushforward along the group morphism $G \to 1$ to the trivial group.
		
	The results in \cite[Section~4.3]{schweigertwoikeofk} generalize to the present context of extended field theories although the details are involved and will not be pursued further in this article:	For composable morphisms $\lambda : G \to H$ and $\mu : H \to J$ of finite groups we  obtain the composition law $(\mu \circ \lambda)_* \cong \mu_* \circ \lambda_*$, where $\cong$ denotes a canonical coherent equivalence of functors between 2-groupoids. Then by sending a finite group $G$ to the 2-groupoid of extended $G$-equivariant topological field theories and a morphism of finite groups to the corresponding push functor \eqref{eqnpushfunctorgrp} we  obtain a 3-functor
		\begin{align}
		\FinGrpd \to 2\text{-}\Grpd \label{pushgrpdext}
		\end{align} from the category of finite groups (seen as a tricategory with only identity 2-cells and 3-cells) to the tricategory of 2-groupoids. 
		
		As in \cite{schweigertwoikeofk} we can deduce by means of \eqref{pushgrpdext} that the orbifold construction is essentially surjective, i.e.\ any extended topological field theory arises as an orbifold theory of a $G$-equivariant theory for any given finite group $G$. 

In \cite{muellerwoike} we use the pushforward for the construction of examples of extended homotopy quantum field theories.

\section{The 3-2-1-dimensional case and the orbifold construction for modular categories\label{3-2-1case}}
The main focus of this section lies on equivariant topological field theories and their orbifoldization in the 3-2-1-dimensional case because there the situation allows for an interesting algebraic description in terms of equivariant modular categories and their orbifoldization via taking homotopy fixed points, i.e.\ via orbifold categories, see e.g.\ \cite{kirrilovg04}. Before addressing this orbifoldization and the relation to our geometrically motivated construction we have to establish a few facts on 3-2-1-dimensional equivariant topological field theories which are interesting in their own right. 

Let $Z : G\text{-}\Cob(3,2,1) \to \TwoVect$ be an extended $G$-equivariant topological field theory for a finite group $G$. Then 
by the construction from Proposition~\ref{satzexttftrep} and Remark~\ref{bmksimpaspherical} we obtain a 2-vector bundle over the groupoid of $G$-bundles over $\sphere^1$ by
 evaluation of $Z$ on the circle $\sphere^1$ equipped with $G$-bundles over the circle. 
 
 The groupoid $\PBun_G(\sphere^1)$ of $G$-bundles over $\sphere^1$ is non-canonically equivalent to the action groupoid $G//G$. More precisely, the equivalence chooses a basepoint and orientation of $\sphere^1$ and assigns to a given bundle the holonomy of the based loop surrounding $\sphere^1$ once in the positive direction. So whenever a bundle is characterized by a group element, we actually mean the holonomy with respect the loop determined by the basepoint and the orientation. To illustrate this issue, consider the bent cylinder (as a bordism $\sphere^1 \coprod \sphere^1 \to \emptyset$)
 \begin{center}
 	\vspace*{0.5cm}
 	\includegraphics[width=0.13\textwidth]{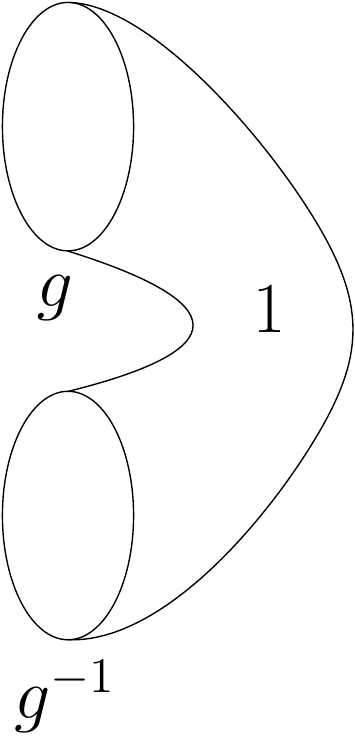}\label{bentcylinderfig}
 	\vspace*{0.5cm}
 \end{center} with the identity homotopy on it. On all circle-shaped slices of the cylinder we have the same bundle. But since the upper and the lower copy of the circle carry different orientations we obtain holonomy values which are inverse to each other. 
 
 By means of a fixed equivalence $\PBun_G(\sphere^1)\cong G//G$, we obtain from $Z$ a 2-vector bundle
\begin{align} G//G \to \TwoVect \end{align} sending an object $g\in G$ in $G//G$ to a 2-vector space $\cat{C}_g^Z=Z(\sphere^1,g)$ and a morphism $h : g \to hgh^{-1}$ in $G//G$ to a 2-linear equivalence
\begin{align} \phi_h : \cat{C}_g^Z=Z(\sphere^1,g) \to \cat{C}_{hgh^{-1}}^Z=Z(\sphere^1,hgh^{-1}).\label{eqnactionofGbyequiv}\end{align} 
We use the notation $h.X := \phi_h X$. 
By construction, this 2-linear equivalence arises by evaluation of $Z$ on the cylinder with ingoing holonomy $g$ and outgoing holonomy $hgh^{-1}$. The two bundles characterized by these holonomies are isomorphic by a gauge transformation $h$. Technically, we have to understand $h$ as a homotopy of the classifying maps for the bundles characterized by the holonomies $g$ and $hgh^{-1}$. This homotopy is put on the cylinder such that we can evaluate $Z$ on it. Depending on what is convenient we will switch between the pictorial representations
	\begin{center}
\vspace*{0.5cm}
\includegraphics[width=0.45\textwidth]{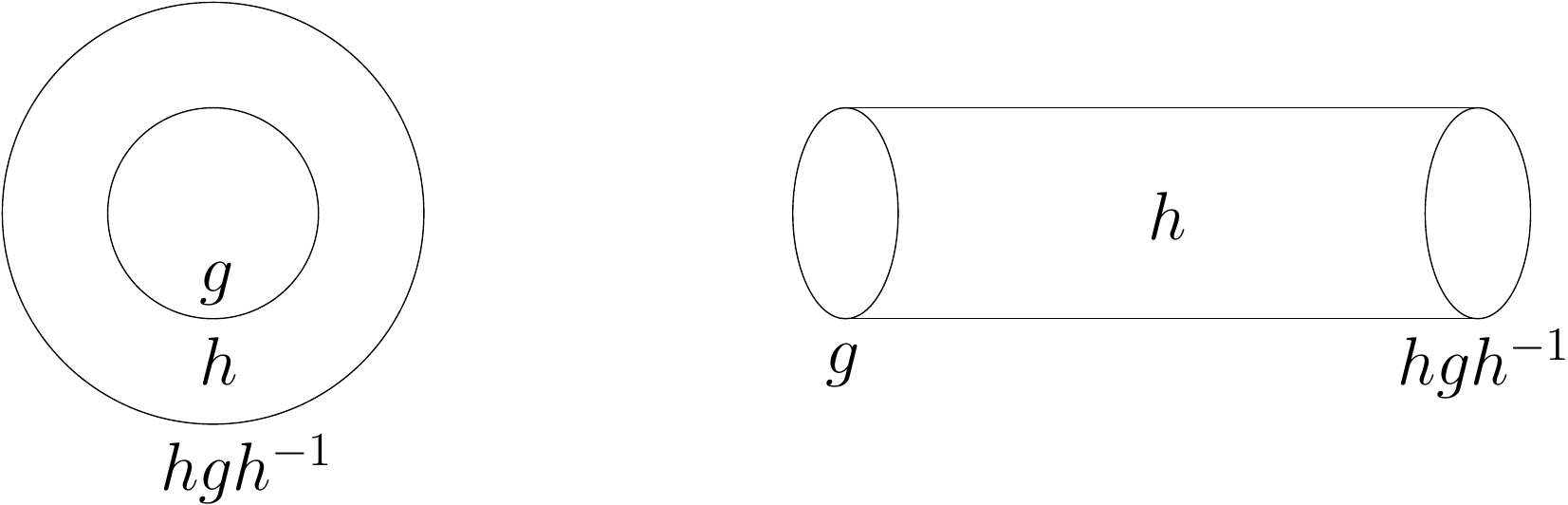}
\vspace*{0.5cm}
\end{center} for the corresponding 1-morphism in $G\text{-}\Cob(3,2,1)$.

Now 
\begin{itemize}
\item the category
\begin{align} \cat{C}^Z := \bigoplus_{g\in G} \cat{C}_g^Z \end{align} 
\item together with the equivalences $\phi_h : \cat{C}^Z \to \cat{C}^Z$ obtained from the equivalences \eqref{eqnactionofGbyequiv} 
\item and the coherence data of our 2-vector bundle consisting of natural isomorphisms $\alpha_{g,h} : \phi_g \circ \phi_h \cong \phi_{gh}$ and $\phi_1 \cong \id_{\cat{C}^Z}$
\end{itemize} form a \emph{$G$-equivariant category} in the terminology of \cite{kirrilovg04}. Using a physics inspired terminology we call $\cat{C}_g$ the \emph{twisted sector} for the group element $g\in G$. The sector $\cat{C}_1^Z$ of the neutral element $1 \in G$ is called the \emph{neutral sector}.

Just like in the non-equivariant case treated in \cite{BDSPV153D}, the category obtained by evaluation on the circle carries a lot of interesting structure, which arises from the geometric framework provided by the topological field theory. For an equivariant version of Dijkgraaf-Witten theory, this analysis was carried out in \cite{maiernikolausschweigerteq}.
The structure present in the general equivariant case will be investigated in the next sections and can be summarized as follows:

\begin{theorem}\label{summarythm}
The evaluation of a 3-2-1-dimensional $G$-equivariant topological field theory on the circle is naturally a $G$-ribbon category.
\end{theorem}

Below we will establish this structure step by step and in particular recall the necessary equivariant analogues of the monoidal product (Section~\ref{equivmonstructure}), the duality  (Section~\ref{secduality}), the braiding (Section~\ref{secbraiding}) and the twist (Section~\ref{sectwist}). 
The structure given in Theorem~\ref{summarythm} will be complemented by a non-degeneracy property leading to equivariant modularity that is strongly linked to non-equivariant modularity via the orbifold construction (Section~\ref{secmod}). 

\spaceplease
\subsection{Equivariant monoidal structure\label{equivmonstructure}}
The pair of pants, appropriately decorated with bundles, will give an equivariant monoidal structure on the category obtained by evaluation of an equivariant 3-2-1-dimensional topological field theory on the circle.

\begin{definition}[Equivariant monoidal category, after $\text{\cite[VI.2.1]{turaevhqft}}$]
Let $G$ be a finite group. A \emph{$G$-equivariant monoidal category} is a $G$-equivariant category $\cat{C} = \bigoplus_{g\in G}\cat{C}_g$, i.e.\ a 2-vector bundle over $G//G$ taking values in the 2-category of categories, together with monoidal structure on $\cat{C}$ and the structure of a monoidal functor on each of the equivalences $\phi_g : \cat{C} \to \cat{C}$ such that
\begin{xenumerate}
  \item for $X \in \cat{C}_g$ and $Y \in \cat{C}_h$ we have $X\otimes Y \in \cat{C}_{gh}$,
  \item the coherence isomorphisms of $\cat{C}$ are compatible with the structure isomorphisms of the group elements acting as monoidal functors, see \cite[Definition~4.2]{maiernikolausschweigerteq}.
\end{xenumerate}
We call a {$G$-equivariant monoidal category} \emph{complex finitely semisimple} if $\cat{C}$ is a $\mathbb{C}$-linear, Abelian, finitely semisimple category such that the monoidal product is $\mathbb{C}$-bilinear (or equivalently a 2-linear map $\cat{C}\boxtimes \cat{C} \to \cat{C}$ defined on the Deligne product of $\cat{C}$ with itself). 
\end{definition}

%By evaluation on the pair of pants decorated with bundles the category $\cat{C}^Z$ becomes a $G$-equivariant monoidal category.
The pair of pants with bundles $g,h \in G$ on the ingoing circles has the bundle $gh$ on the outgoing circle; pictorially:
\begin{center}
\vspace*{0.5cm}
\includegraphics[width=0.25\textwidth]{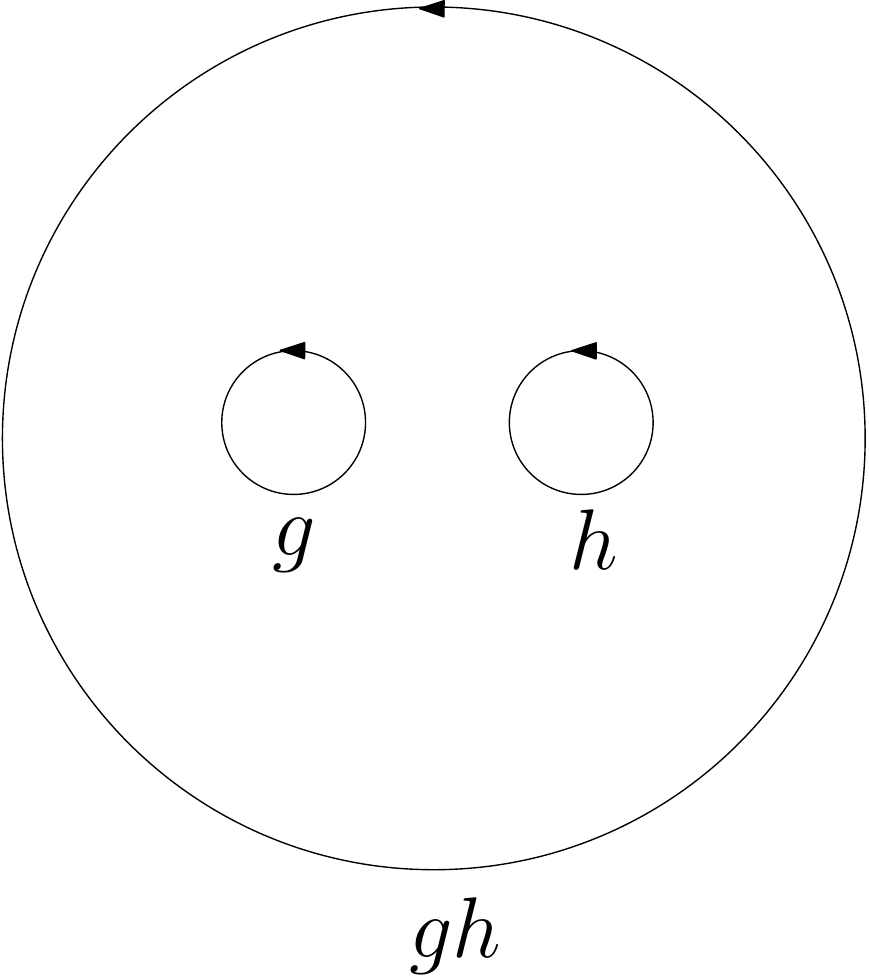}
\vspace*{0.5cm}
\end{center}\label{figtrinion}

Or, in other words, the evaluation of the stack $\Pi(?,BG)$ of $G$-bundles on the pair of pants yields the span
\begin{align}\label{eqnpairofpantsspan}
G//G \times G//G \stackrel{B}{\longleftarrow} (G\times G)//G \stackrel{M}{\to} G//G \ ,
\end{align} where $B$ is the obvious functor and $M$ the multiplication.
Hence, evaluation of $Z$ on the pair of pants decorated with ingoing bundles $g$ and $h$ yields a 2-linear functor $\otimes_{g,h} : \cat{C}_g^Z \boxtimes \cat{C}_h^Z \to \cat{C}_{gh}^Z$. These functors assemble to give the monoidal product $\otimes : \cat{C}^Z\boxtimes \cat{C}^Z \to \cat{C}^Z$. As required for an equivariant monoidal product, it carries $X\in \cat{C}_g^Z$ and $Y \in \cat{C}_h^Z$ to $X\otimes Y \in \cat{C}_{gh}^Z$. Evaluation on the disk seen as bordism $\emptyset \to \sphere^1$ decorated with the trivial $G$-bundle yields a 2-linear functor $\eta: \FinVect \to \cat{C}_1^Z$, which is determined by the object $I:= \eta(\mathbb{C})$ in $\cat{C}_1^Z$. This object can easily be seen to be the monoidal unit. Let us formally state these findings:

\begin{proposition}[]\label{satzequivariantmonoidalstructure}
	For any extended $G$-equivariant topological field theory $Z : G\text{-}\Cob(3,2,1) \to \TwoVect$ the evaluation on the pair of pants endows
	$\cat{C}^Z $ with the structure of a complex finitely semisimple $G$-equivariant monoidal category.
\end{proposition}

\begin{proof} We have already given the monoidal product. The necessary associators and unitors can be found just as in the non-equivariant case treated in \cite{BDSPV153D}. Hence, the only thing left to prove is the fact that the action of $G$ on $\cat{C}^Z$ by \eqref{eqnactionofGbyequiv} is by monoidal functors: Since the disk is contractible, for $g\in G$ there is a natural isomorphism
	\begin{center}
	\begin{tikzpicture}[scale=2, implies/.style={double,double equal sign distance,-implies},
	dot/.style={shape=circle,fill=black,minimum size=2pt,
		inner sep=0pt,outer sep=2pt},]
	\node (A1) at (0,1) {$\FinVect$};
	\node (A2) at (2,1) {$\cat{C}_{1}^Z$};
	\node (B1) at (1,0.5) {};
	\node (B2) at (2,0) {$\cat{C}_{1}^Z$};
	\path[->,font=\scriptsize]
	(A1) edge node[above]{$\eta$} (A2)
%	(A1) edge node[left]{$g\otimes g $} (B1)
	(A2) edge node[right]{$\phi_g$} (B2)
	(A1) edge node[below]{$\eta$} (B2);
	\draw (A2) edge[implies] node[above] {\scriptsize$\cong$} (B1);
	\end{tikzpicture}.
	\end{center}
	Next we have to exhibit natural isomorphisms
	\begin{center}
	\begin{tikzpicture}[scale=2, implies/.style={double,double equal sign distance,-implies},
	dot/.style={shape=circle,fill=black,minimum size=2pt,
		inner sep=0pt,outer sep=2pt},]
	\node (A1) at (0,1) {$\cat{C}_a^Z \boxtimes \cat{C}_b^Z$};
	\node (A2) at (2,1) {$\cat{C}_{ab}^Z$};
	\node (B1) at (0,0) {$\cat{C}_{gag^{-1}}^Z \boxtimes \cat{C}_{gbg^{-1}}^Z$};
	\node (B2) at (2,0) {$\cat{C}_{gabg^{-1}}^Z$};
	\path[->,font=\scriptsize]
	(A1) edge node[above]{$\otimes$} (A2)
	(A1) edge node[left]{$\phi_g\boxtimes \phi_g $} (B1)
	(A2) edge node[right]{$\phi_g$} (B2)
	(B1) edge node[below]{$\otimes$} (B2);
	\draw (A2) edge[implies] node[above] {\scriptsize$\kappa_g\ $} (B1);
	\end{tikzpicture}
	\end{center} for $a,b,g\in G$.
To obtain the isomorphism $\kappa_g$ note that the clockwise composition is naturally isomorphic to the evaluation of $Z$ on
	\begin{center}
\vspace*{0.5cm}
\includegraphics[width=0.8\textwidth]{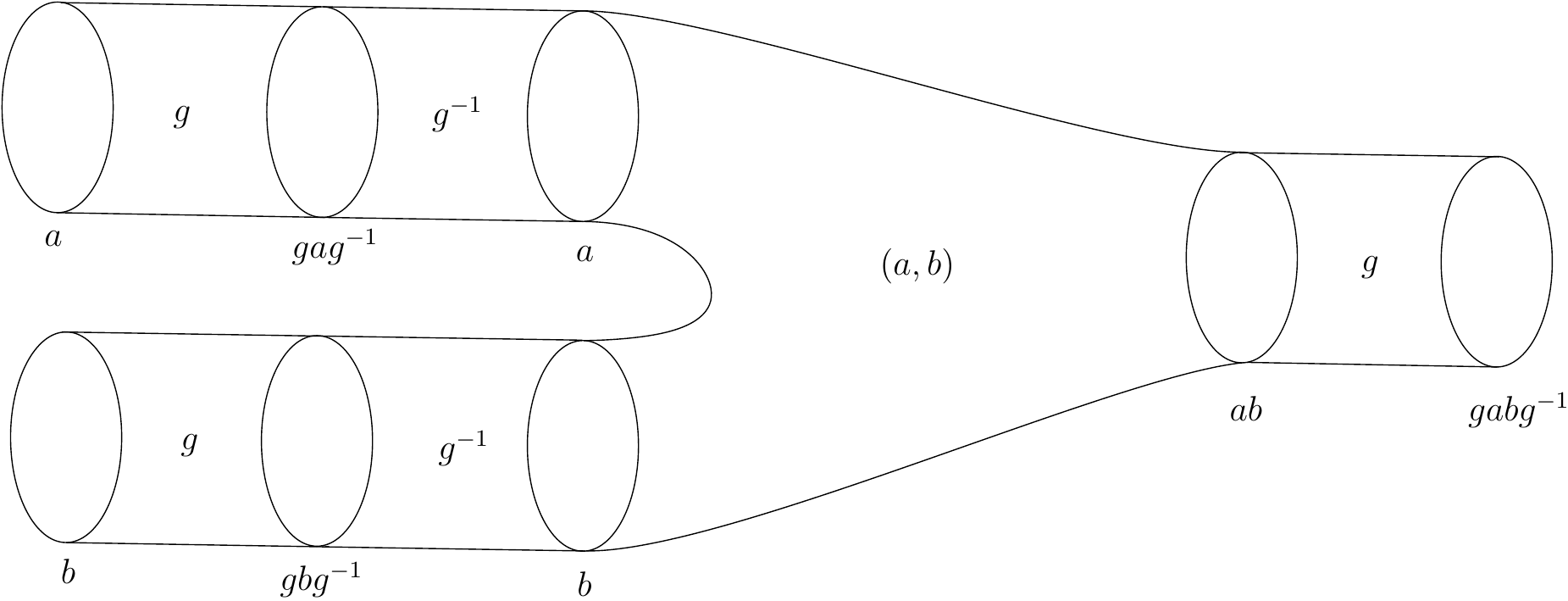}.
\vspace*{0.5cm}
\end{center} 
We have just added on the left cancellable homotopies representing $g$. But this is naturally isomorphic to the counterclockwise composition because this bordism can be seen as a cylinder with the homotopy $g$ on it followed by a pair of pants with ingoing bundles $gag^{-1}$ and $gbg^{-1}$ (the bundle decoration of pair of pants is determined by the ingoing bundles). From this way of constructing $\kappa_g$ it follows that it satisfies the necessary coherence conditions.       
\end{proof}

The following observations allow us to compute a 3-2-1-dimensional $G$-equivariant topological field theory $Z$ on surfaces decorated with $G$-bundles
just by means of the monoidal structure on $\cat{C}^Z$: The monoidal product in $\cat{C}^Z$ is built from the 2-linear maps
\begin{align}
\otimes_{g,h} \ :\ \cat{C}_g^Z \boxtimes \cat{C}_h^Z \to \cat{C}_{gh}^Z \end{align} obtained by evaluation on the pair of pants decorated with bundles as explained above. Evaluation $Z$ on the same manifold read backwards yields 2-linear maps
\begin{align} \Delta_{g,h}\ :\ \cat{C}_{gh}^Z \to \cat{C}_g^Z \boxtimes \cat{C}_h^Z \ .\end{align} The direct generalization of the adjunction relation in \cite{BDSPV153D} gives us the adjunction
\begin{align} \otimes_{g,h} \dashv \Delta_{g,h} \label{eqnrelprodcoprod}\end{align} in $\TwoVect$. The same arguments apply to the monoidal unit
\begin{align} \eta : \FinVect \to \cat{C}_1^Z \end{align}
and the evaluation of $Z$ on the manifold read backwards, namely
\begin{align}
\varepsilon : \cat{C}_1^Z \to \FinVect\ ,\end{align} i.e.\ we obtain the adjunction
\begin{align} \eta \dashv \varepsilon\label{eqnunitcounitadj2vect}\end{align} in $\TwoVect$

In order to use these adjunctions, we recall from \cite[Section~2.2]{BDSPV153D}\label{bimodsppage}  some needed facts on the symmetric monoidal bicategory $\Bimod$ of 2-vector spaces, bimodules (here a bimodule from $\cat{V}$ to $\cat{W}$ between $\mathbb{C}$-linear categories $\cat{V}$ and $\cat{W}$ is a functor $P: \cat{V}^\text{\normalfont opp} \boxtimes \cat{W} \to \FinVect$) and natural transformations. 
The composition of bimodules $P: \cat{U}^\text{\normalfont opp} \boxtimes \cat{V} \to \FinVect$ and $Q: \cat{V}^\text{\normalfont opp} \boxtimes \cat{W}\to\FinVect$ is the bimodule $Q\circ P : \cat{U}^\text{\normalfont opp} \boxtimes \cat{W} \to \FinVect$ given by the coend
\begin{align}
(Q\circ P)(U, W) := \int^{V \in\cat{V}} Q(V,W)\otimes P(U,V) \myforall U \in \cat{U}, \quad W \in \cat{W},
\end{align} see \cite{maclanecat} for an introduction to coends.
Any
2-linear map $F: \cat{V} \to \cat{W}$ gives rise to a bimodule $F_* : \cat{W}^\text{\normalfont opp} \boxtimes \cat{V} \to \FinVect$ by
\begin{align} F_*(W,V):= \Hom_\cat{W}(W,FV) \myforall V\in \cat{V}, \quad W \in \cat{W}.\end{align}This assignment extends to a 2-functor
\begin{align} ?_* : \TwoVect \to \Bimod.\end{align} The functoriality of $?_*$ entails that for 2-linear maps $F: \cat{U} \to \cat{V}$ and $G: \cat{V} \to \cat{W}$
\begin{align}
\Hom_\cat{W} (W,GFU) \cong \int^{V \in \cat{V}} \Hom_\cat{W} (W,GV) \otimes \Hom_\cat{V}(V,FU) \myforall U \in \cat{U},\quad W \in \cat{W}\label{gluinglawbimodules}
\end{align} by a canonical isomorphism of vector spaces. 
Note that $F:\cat{V} \to \cat{W}$ also gives rise to a bimodule $F^* : \cat{V}^\text{\normalfont opp} \boxtimes \cat{W} \to \Vect$ by
\begin{align} F^*(V,W) := \Hom_\cat{W}(FV,W) \myforall V\in \cat{V}, \quad W \in \cat{W},\end{align} which is related to $F_*$ by the adjunction
\begin{align}
F_* \dashv F^*\label{eqnadjunctionrelbimod}
\end{align} in $\Bimod$.

Now from \eqref{eqnrelprodcoprod} we first deduce
\begin{align} (\otimes_{g,h})_* \dashv (\Delta_{g,h})_*, \end{align} but by \eqref{eqnadjunctionrelbimod} also
\begin{align}
(\otimes_{g,h})_* \dashv \otimes_{g,h}^*.
\end{align} Uniqueness of adjoints yields a canonical isomorphism
\begin{align}
(\Delta_{g,h})_* \cong \otimes_{g,h}^*.\end{align} If we apply this also to \eqref{eqnunitcounitadj2vect}, we have proven the following:

\begin{proposition}[]%\label{lemmacomputeZonsurfaces}
Let $G$ be a finite group and $g,h \in G$.
For a 3-2-1-dimensional $G$-equivariant topological field theory $Z$ we obtain the following adjunction relations for the structure functors obtained from surfaces with boundary:
\begin{myenumerate}
\item $(\Delta_{g,h})_* \cong \otimes_{g,h}^*$

\item and $\varepsilon_* \cong \eta^*$

\end{myenumerate} 
\end{proposition}

 If we use the notation $\eta_*(X):= \eta_*(X,\mathbb{C})$ and the dual convention for $\varepsilon$, we arrive at:

\begin{corollary}[]\label{korbimodG}
For $g,h \in G$ we have
\begin{myenumerate}
\item ${\otimes_{g,h}}_*(W,X\boxtimes Y) = \Hom_{\cat{C}^Z_{gh}}(W,X\otimes Y)$ for all $X \in \cat{C}_g^Z, Y\in \cat{C}_h^Z$ and $W\in \cat{C}_{gh}^Z$,

\item ${\Delta_{g,h}}_*(Y\boxtimes W,X) = \Hom_{\cat{C}^Z_{gh}}(Y\otimes W,X)$ for all $X\in \cat{C}_{gh}^Z$, $Y\in \cat{C}_g^Z$ and $W\in \cat{C}_{h}^Z$, 

\item $\eta_*(X) = \Hom_{\cat{C}_1^Z}(X,I)$ for all $X\in \cat{C}_1^Z$
 
\item and $\varepsilon_*(X) = \Hom_{\cat{C}_1^Z}(I,X)$ for all $X\in \cat{C}_1^Z$.

\end{myenumerate}
\end{corollary}

 Corollary~\ref{korbimodG} allows us to compute the evaluation of an extended $G$-equivariant topological field theory on any surface decorated with bundles in terms of the monoidal structure.

\begin{example}\label{exbentcylinder}
As an illustration,
 let us compute the evaluation 
\begin{align} Z(B_g) : \cat{C}_g^Z \boxtimes \cat{C}_{g^{-1}}^Z \to \FinVect \end{align} of
a 3-2-1-dimensional extended $G$-equivariant topological field theory $Z$ on the bent cylinder $B_g$ decorated with bundles as on page~\pageref{bentcylinderfig}. By cutting $B_g$ into a pair of pants and a cup we find via functoriality of $Z$, \eqref{gluinglawbimodules} and Corollary~\ref{korbimodG}
\begin{align}
Z(B_g)(X,Y) \cong \int^{W\in \cat{C}_1^Z} \Hom_{\cat{C}_1^Z} (I,W) \otimes \Hom_{\cat{C}_1^Z}(W,X \otimes Y) \myforall X \in \cat{C}_g^Z, \quad Y \in \cat{C}_{g^{-1}}^Z.
\end{align} 
By the co-Yoneda Lemma, see e.g. \cite[Example~1.4.6]{riehlhomotopy}, this implies
\begin{align}
Z(B_g)(X,Y) \cong \Hom_{\cat{C}_1^Z}(I,X\otimes Y), 
\end{align} i.e.\ $Z(B_g)(X,Y)$ is given by the invariants in the monoidal product $X\otimes Y$. 
\end{example}

\subsection{Duality\label{secduality}}
In the next step we prove that $\cat{C}^Z$ is also rigid:

\begin{proposition}[]\label{satzczduals}
For any extended $G$-equivariant topological field theory $Z : G\text{-}\Cob(3,2,1) \to \TwoVect$ the monoidal category $\cat{C}^Z$ has duals.  
\end{proposition}

The duals considered here are \emph{left} duals in the terminology of \cite{egno}.
An argument analogous to the one given in the proof below shows the existence of right duals.

\begin{proof}
The proof uses the appropriate equivariant versions of the arguments given in \cite{BDSPV153D} in the non-equivariant case: For a group element $g\in G$ we denote the 1-morphism
\begin{center}
\vspace*{0.5cm}
\includegraphics[width=0.65\textwidth]{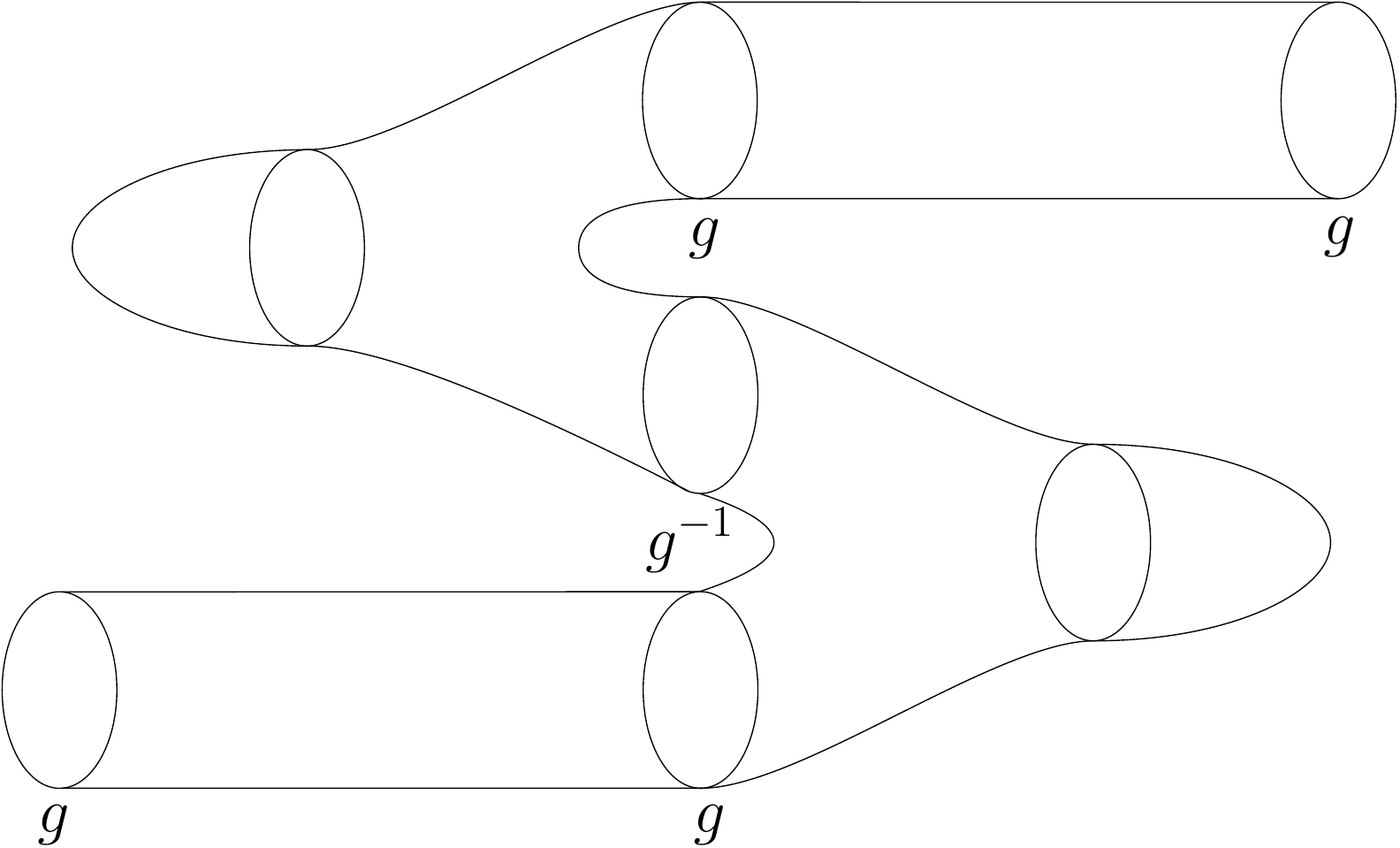},
\vspace*{0.5cm}
\end{center} in $G\text{-}\Cob(3,2,1)$ by $N_g$. This is the 1-morphism appearing in the proof of \cite[Proposition~4.2]{BDSPV153D} appropriately decorated with bundles. It is diffeomorphic to the cylinder with $g$ on the ingoing and outgoing circle and the identity homotopy on it. This gives us a natural isomorphism $\id_{\cat{C}_g^Z} \cong Z(N_g)$ of 2-linear maps $\cat{C}_g^Z \to \cat{C}_g^Z$. By slicing up $N_g$ as indicated in the above picture and using the functoriality and monoidality of $Z$ we find yet another 2-linear map $\cat{C}_g^Z \to \cat{C}_g^Z$, which is also naturally isomorphic to the identity functor. By looking at the resulting isomorphism for the corresponding bimodules $(\cat{C}_g^Z)^\text{\normalfont opp} \boxtimes \cat{C}_g^Z \to \Vect$ one deduces as in \cite[Propositions~4.2 and 4.8]{BDSPV153D} that for any $X\in \cat{C}_g^Z$ there is an object $X' \in \cat{C}_{g^{-1}}^Z$ together with morphisms $\alpha : I \to X' \otimes X$ and $\beta : X \otimes X' \to I$ such that
\begin{center}
\vspace*{0.5cm}
\includegraphics[width=0.3\textwidth]{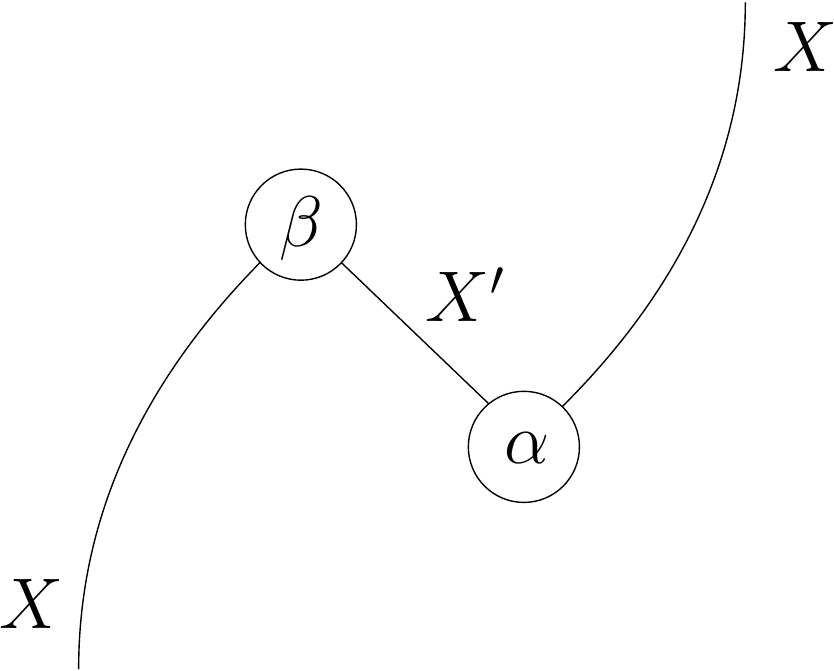}
\vspace*{0.5cm}
\end{center} is the identity of $X$. Again, as in the proof of \cite[Propositions~4.8]{BDSPV153D}, this implies that the endomorphism
\begin{center}
\vspace*{0.5cm}
\includegraphics[width=0.3\textwidth]{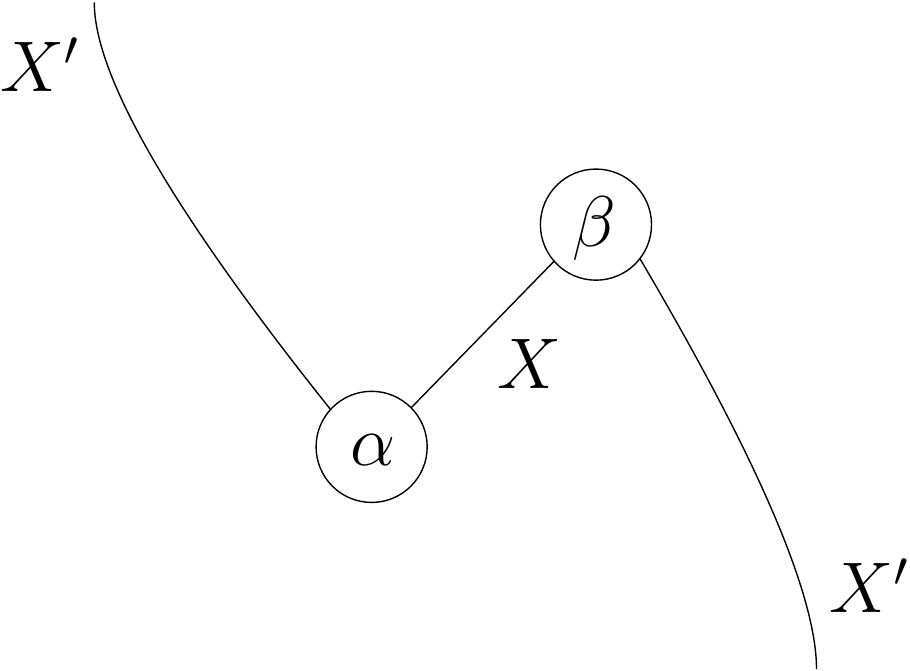}
\vspace*{0.5cm}
\end{center} of $X'$ is an idempotent, and by finite (co)completeness of $\cat{C}_{g^{-1}}^Z$ it splits into morphisms $\gamma : X' \to X^*$ and $\delta : X^* \to X'$ in $\cat{C}_{g^{-1}}^Z$ such that $\delta \circ \gamma = \id_{X^*}$. A direct computation in the graphical calculus shows that $X^* \in \cat{C}_{g^{-1}} ^Z$ is dual to $X$ with evaluation $\operatorname{eva}_X := \beta \circ (\delta \otimes \id_X)$ and coevaluation $\operatorname{coeva}_X := (\id_X \otimes \gamma )\circ \alpha$.       
\end{proof}

\begin{corollary}[]\label{korczduality}
	For any extended $G$-equivariant topological field theory $Z : G\text{-}\Cob(3,2,1) \to \TwoVect$ the duality in the category $\cat{C}^Z$ has the following properties:
	\begin{myenumerate}
		
		\item The dual $X^*$ of an object $X\in \cat{C}_g^Z$ lives in the sector $\cat{C}_{g^{-1}}^Z$.\label{korczdualitya}
		
		\item For $g,h \in G$ and $X \in \cat{C}_g^Z$ the object $h.X^*$ is dual to $h.X$, i.e.\ $(h.X)^* \cong h.X^*$. \label{korczdualityb}
		
		\end{myenumerate}
	\end{corollary}

\begin{proof}
	Assertion~\ref{korczdualitya} is clear from the proof of Proposition~\ref{satzczduals} and also a necessity since $X\otimes X^*$ needs to be in the neutral sector. Assertion~\ref{korczdualityb} follows directly from the fact that $G$ acts by monoidal functors (Proposition~\ref{satzequivariantmonoidalstructure}).     
	\end{proof}

\subsection{Equivariant braiding\label{secbraiding}}
 The next piece of structure on the category $\cat{C}^Z$ obtained from a 3-2-1-dimensional $G$-equivariant topological field theory $Z$ is a $G$-braiding. First let us recall the relevant notion:

\begin{definition}[Equivariant braided category, after $\text{\cite[VI.2.2]{turaevhqft}}$]
	Let $G$ be a finite group $G$. A $G$-\emph{braiding} on a $G$-equivariant monoidal category $\cat{C}$ is a family of isomorphisms
	\begin{align}
		c_{X,Y} : X \otimes Y \to g. Y \otimes X \myforall X \in \cat{C}_g, \quad Y \in \cat{C}_h, \quad g,h\in G
		\end{align} which are natural in $X$ and $Y$ and satisfy an analogue of the hexagon axiom, see e.g.\ \cite[Definition~4.5]{maiernikolausschweigerteq}.
	\end{definition} 

To construct the $G$-braiding, recall  the 1-morphism in $G\text{-}\Cob(3,2,1)$ giving us the monoidal product 
can be written as the pair of pants
\begin{center}
	\vspace*{0.5cm}
	\includegraphics[width=0.25\textwidth]{trinion}
	\vspace*{0.5cm}.
\end{center}The holonomies around the ingoing circles are $g$ and $h$, respectively, and consequently the outgoing circle carries holonomy $gh$. Rotating the inner circles counterclockwise around each other while keeping the outgoing circle fixed yields a diffeomorphism of the pair of pants relative boundary. In the sense of Remark~\ref{bmkextcob}, \ref{bmkextcobdiffeo} this diffeomorphism gives rise to an invertible 2-morphism $G\text{-}\Cob(3,2,1)$, also described in detail in \cite[Lemma~3.25]{maiernikolausschweigerteq}, on which we can evaluate $Z$. As a result, we get natural isomorphisms
\begin{align} c_{X,Y} : X \otimes Y \to g.Y \otimes X \myforall X \in \cat{C}_g^Z, \quad Y \in \cat{C}_h^Z\ .\label{eqnbraidingisomorphisms}\end{align}

\begin{proposition}[]\label{satzequivbraiding}
	For any extended $G$-equivariant topological field theory $Z : G\text{-}\Cob(3,2,1) \to \TwoVect$ the $G$-equivariant monoidal category $\cat{C}^Z$ is $G$-braided by \eqref{eqnbraidingisomorphisms}.
\end{proposition}

\begin{proof}
	The above description of the relevant 2-morphism as coming from a rotation allows us to verify directly the hexagon axiom. % of \cite[Appendix~5,~Definition~1.8]{turaevhqft}.
	 It remains to check that the braiding is compatible with the $G$-action in the sense of \cite[Definition~4.5]{maiernikolausschweigerteq}. But this follows from the fact that the rotation giving rise to the braiding commutes with gluing in cylinders with cancellable homotopies on them, which gave rise to the structure maps for the elements of $G$ as monoidal functors
	(see the proof of Proposition~\ref{satzequivariantmonoidalstructure}).      
\end{proof}

\subsection{Twist\label{sectwist}}
 Finally, we also get an equivariant twist.

\begin{definition}[Equivariant ribbon category, after $\text{\cite[VI.2.3]{turaevhqft}}$]
	A \emph{$G$-twist} on a $G$-braided monoidal category $\cat{C}$ with dualities is a family of natural isomorphisms $\theta_X : X \to g.X$ for all $g\in G$ and $X\in\cat{C}_g$ compatible with duality and the action of $G$, see \cite[Definition~4.8]{maiernikolausschweigerteq}.
	A \emph{$G$-equivariant ribbon category} is a $G$-braided monoidal category with dualities and a $G$-twist. 
	\end{definition}

For the construction of the $G$-twist we compare the identity of $\cat{C}_g^Z$ to the equivalence $\phi_g : \cat{C}_g^Z \to \cat{C}_g^Z$. Both are obtained by evaluation of $Z$ on a cylinder with ingoing and outgoing circle labeled by $g$. But the 1-morphism which yields the identity carries the constant homotopy while the 1-morphism giving us $\phi_g : \cat{C}_g^Z \to \cat{C}_g^Z$ carries $g$ seen as a homotopy. More precisely, if $g$ is represented by the loop $\gamma : \sphere^1 \to BG$, then $\phi_g : \cat{C}_g^Z \to \cat{C}_g^Z$ is the evaluation of $Z$ on the cylinder together with the map $\widetilde \gamma : \sphere^1 \times [0,1] \to BG$ with $\widetilde \gamma(z,t) = \gamma(z\operatorname{e}^{2\pi \im t})$ for all $(z,t) \in \sphere^1 \times I$.

Consider now the Dehn twist of the cylinder, i.e.\ the diffeomorphism
\begin{align}
D : \sphere^1 \times I \to \sphere^1 \times I, \quad (z,t) \mapsto \left(z\operatorname{e}^{2\pi \im t}, t\right)
\end{align} keeping the boundary circles fixed, and observe that the pullback of the constant homotopy from $\gamma$ to $\gamma$ along $D$ is $\widetilde \gamma$. Now by Remark~\ref{bmkextcob}, \ref{bmkextcobdiffeo} we obtain a natural isomorphism from the identity of $\cat{C}_g^Z$ to $g : \cat{C}_g^Z \to \cat{C}_g^Z$, i.e.\ natural isomorphisms \begin{align}\label{eqntheeqtwists} \theta_X : X \to g.X \myforall X\in \cat{C}_g^Z\ .\end{align} 

Let us make two important observations:

\begin{xenumerate}
	
	\item The twist in the sector of some $g\in G$ represented by the loop $\gamma$ (as above) can also be obtained by evaluation of $Z$ on the 2-cell coming from the homotopy
	\begin{align}
	\sphere^1 \times I \times I \ni (z,t,s) \mapsto \gamma(z\operatorname{e}^{2\pi\im st})
	\end{align} of maps on the cylinder $\sphere^1 \times I$ going from the constant homotopy of $\gamma$ to $\widetilde \gamma$. \label{twistandactionobservation1}
	
	\item In geometric terms, the 2-cell
	\begin{center}
		\vspace*{0.5cm}
		\includegraphics[width=0.3\textwidth]{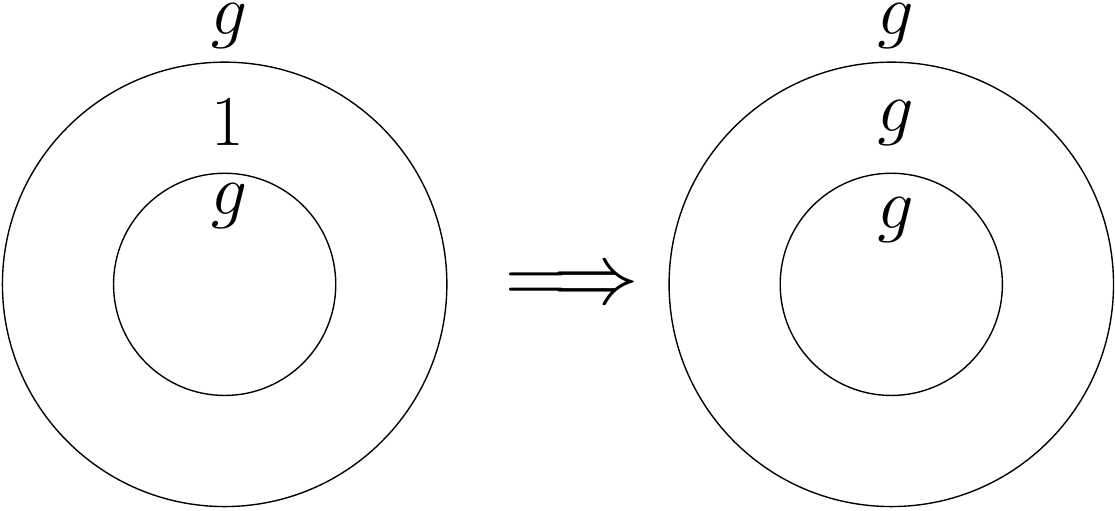}
		\vspace*{0.5cm}
	\end{center} underlying the twist can be seen as coming from a counterclockwise rotation of the ingoing circle against the outgoing one. \label{twistandactionobservation2}
	
\end{xenumerate}

\begin{proposition}[]\label{satzgribbonkat}
	For any extended $G$-equivariant topological field theory $Z : G\text{-}\Cob(3,2,1) \to \TwoVect$ the braided $G$-equivariant monoidal category $\cat{C}^Z$ is equipped with a $G$-twist by \eqref{eqntheeqtwists}, i.e.\ $\cat{C}^Z$ is a complex finitely semisimple $G$-ribbon category.
\end{proposition}

\begin{proof}

For the proof that $\theta$ is actually a $G$-twist we need to show that it is compatible with the already existing structure:

\begin{pnum}
	
	\item Compatibility with the $G$-action: For $X\in \cat{C}_g^Z$ and $h\in G$ we need to prove that the square
	\begin{align}\begin{array}{c}
		\begin{tikzpicture}[scale=2, implies/.style={double,double equal sign distance,-implies},
		dot/.style={shape=circle,fill=black,minimum size=2pt,
			inner sep=0pt,outer sep=2pt},]
		\node (A1) at (0,1) {$h.X$};
		\node (A2) at (2,1) {$(hgh^{-1}).h.X$};
		\node (B1) at (0,0) {$h.g.X$};
		\node (B2) at (2,0) {$(hg).X$};
		\path[->,font=\scriptsize]
		(A1) edge node[above]{$\theta_{h.X}$} (A2)
			(A1) edge node[left]{$h.\theta_X$} (B1)
		(A2) edge node[right]{$\cong$} (B2)
		(B1) edge node[below]{$\cong$} (B2);
		\end{tikzpicture}\end{array}\label{eqntwistandaction}
		\end{align}
	is commutative, where $\cong$ stands for the coherence isomorphisms of $\cat{C}^Z$.
	But this follows from the commutativity of the square 
		\begin{center}
		\vspace*{0.5cm}
		\includegraphics[width=0.7\textwidth]{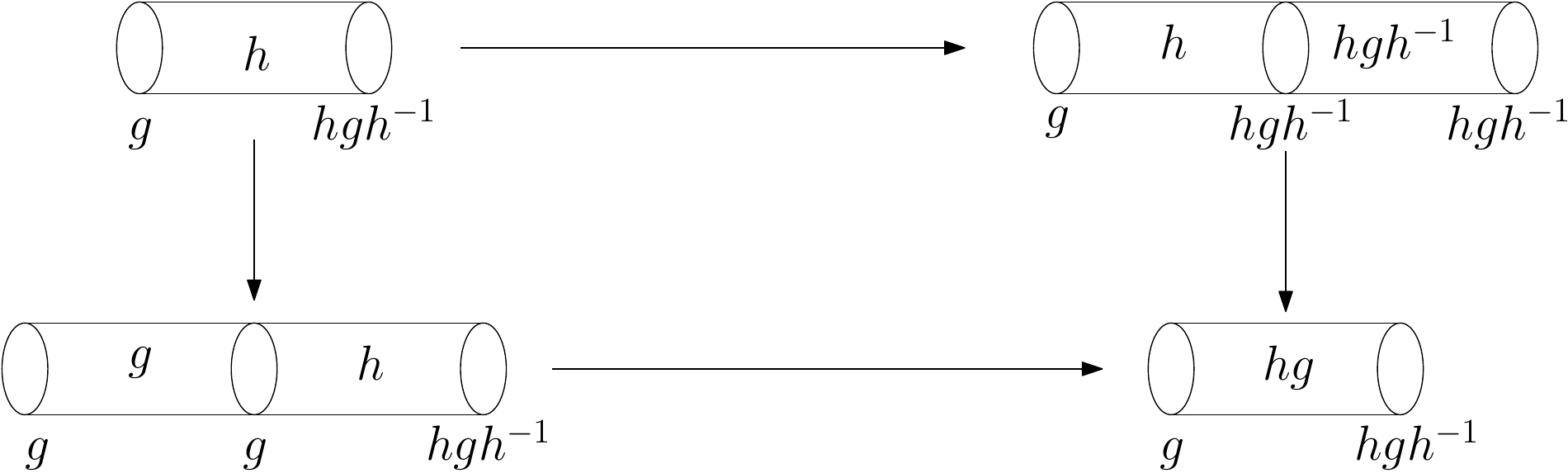}
		\vspace*{0.5cm},
	\end{center} in which the vertices are 1-cells and the edges are 2-cells in $G\text{-}\Cob(3,2,1)$, respectively. In the clockwise path, we first use the twist to add the homotopy $hgh^{-1}:hgh^{-1} \to hgh^{-1}$ and then compose and cancel homotopies. In the counterclockwise path, we use the twist to add the homotopy $g: g \to g$ and compose again. Both paths are represented by homotopies of maps on the cylinder relative boundary (see observation~\ref{twistandactionobservation1} above) and are equivalent up to higher homotopy, which proves commutativity of \eqref{eqntwistandaction}. \label{prooftwistdualityaction}
	 
	\item Compatibility with the braiding: For $X\in \cat{C}_g^Z$ and $Y\in \cat{C}_h^Z$ the isomorphism
	\begin{align} \theta_{X\otimes Y} : X \otimes Y \to (gh).(X\otimes Y) \end{align} is obtained by evaluation of $Z$ on the 2-cell
	\begin{center}
		\vspace*{0.5cm}
		\includegraphics[width=0.7\textwidth]{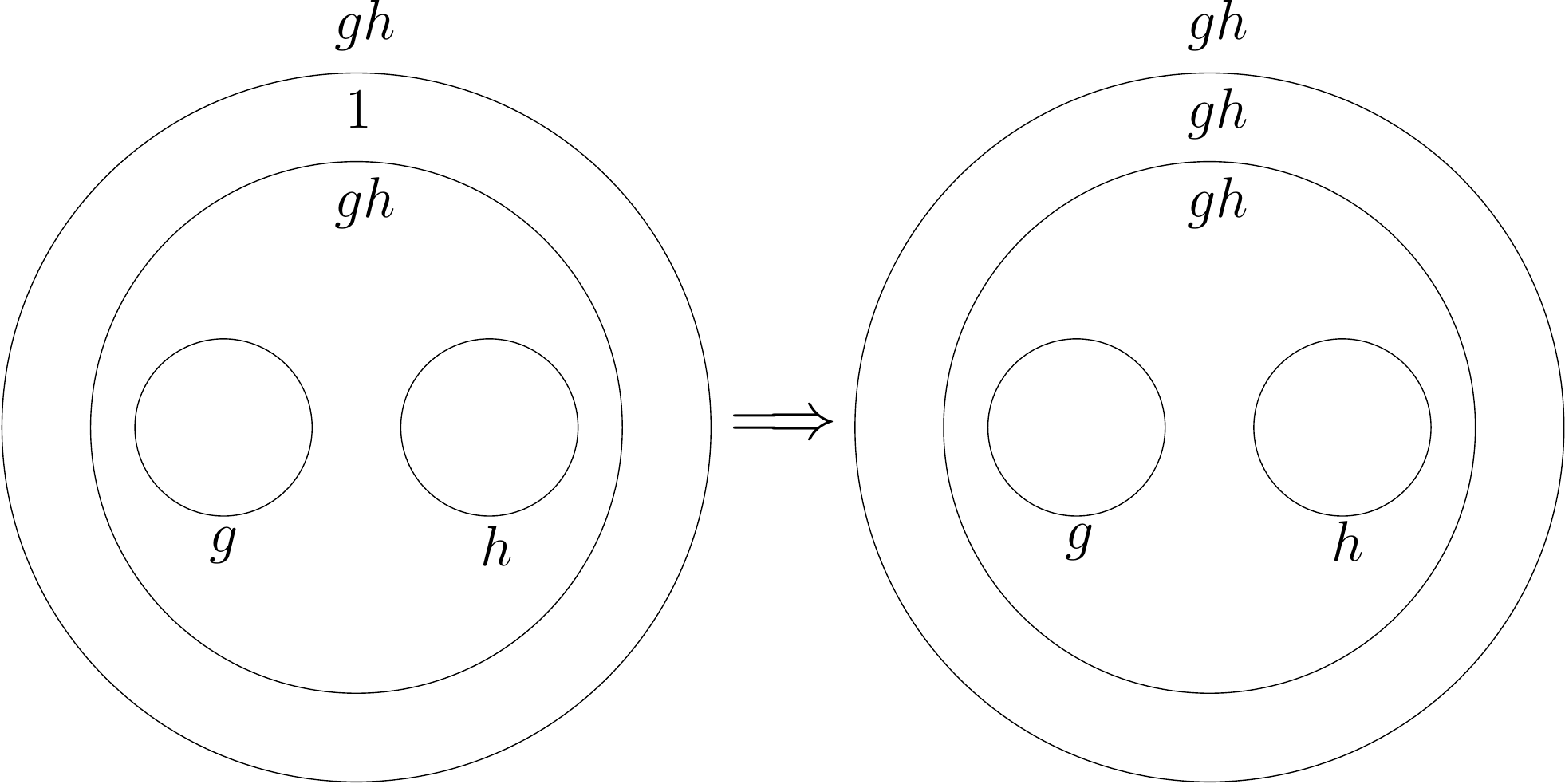}
		\vspace*{0.5cm}
	\end{center} using the twist to add the homotopy $gh: gh \to gh$ in the outer ring. By observation~\ref{twistandactionobservation2}  on page~\pageref{twistandactionobservation2} this is accomplished by a full counterclockwise rotation of the middle circle containing the two smaller circles against the outer circle. The same result is obtained by first rotating the smaller circles within the middle circle (this gives us the twists on the two tensor factors) and the rotating the middle circle against the outer one (this gives us a double braiding by the proof of Proposition~\ref{satzequivbraiding}). In the graphical calculus this means
	\begin{center}
	\vspace*{0.5cm}
	\includegraphics[width=0.25\textwidth]{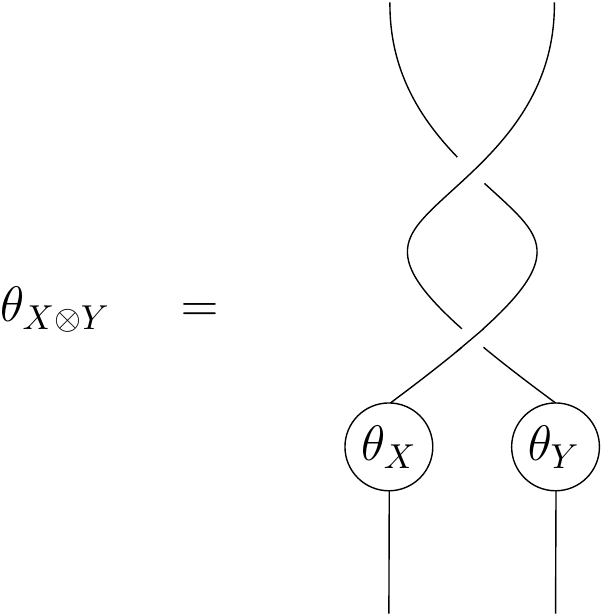}
	\vspace*{0.5cm} ,
\end{center} which is exactly the compatibility of twist and braiding. 
	
	\item Compatibility with duality: Since $X^*\in \cat{C}_{g^{-1}}^*$ for $X \in \cat{C}_g^Z$ by Proposition~\ref{satzczduals}, the twist evaluated on $g.X^*$ together with the coherence isomorphisms yields an isomorphism
	\begin{align} \theta_{g.X^*} : g.X^* \to g^{-1}.g.X^* \cong X^*. \label{eqndualitytwisteqn1}\end{align} Here we also used the coherence isomorphisms, but by abuse of notation refrain from giving a new name to the composite. To prove the compatibility of twist and duality, we need to show that this map is equal to the dual
	\begin{align} \theta_X^* : g.X^*\cong(g.X)^* \to X^*\label{eqndualitytwisteqn2} \end{align} of $\theta_X : X \to g.X$ (recall that $g.X^*\cong(g.X)^*$ by Proposition~\ref{satzczduals}). 
	To this end, we evaluate the commutative triangle 
	\begin{center}
		\vspace*{0.5cm}
		\includegraphics[width=0.65\textwidth]{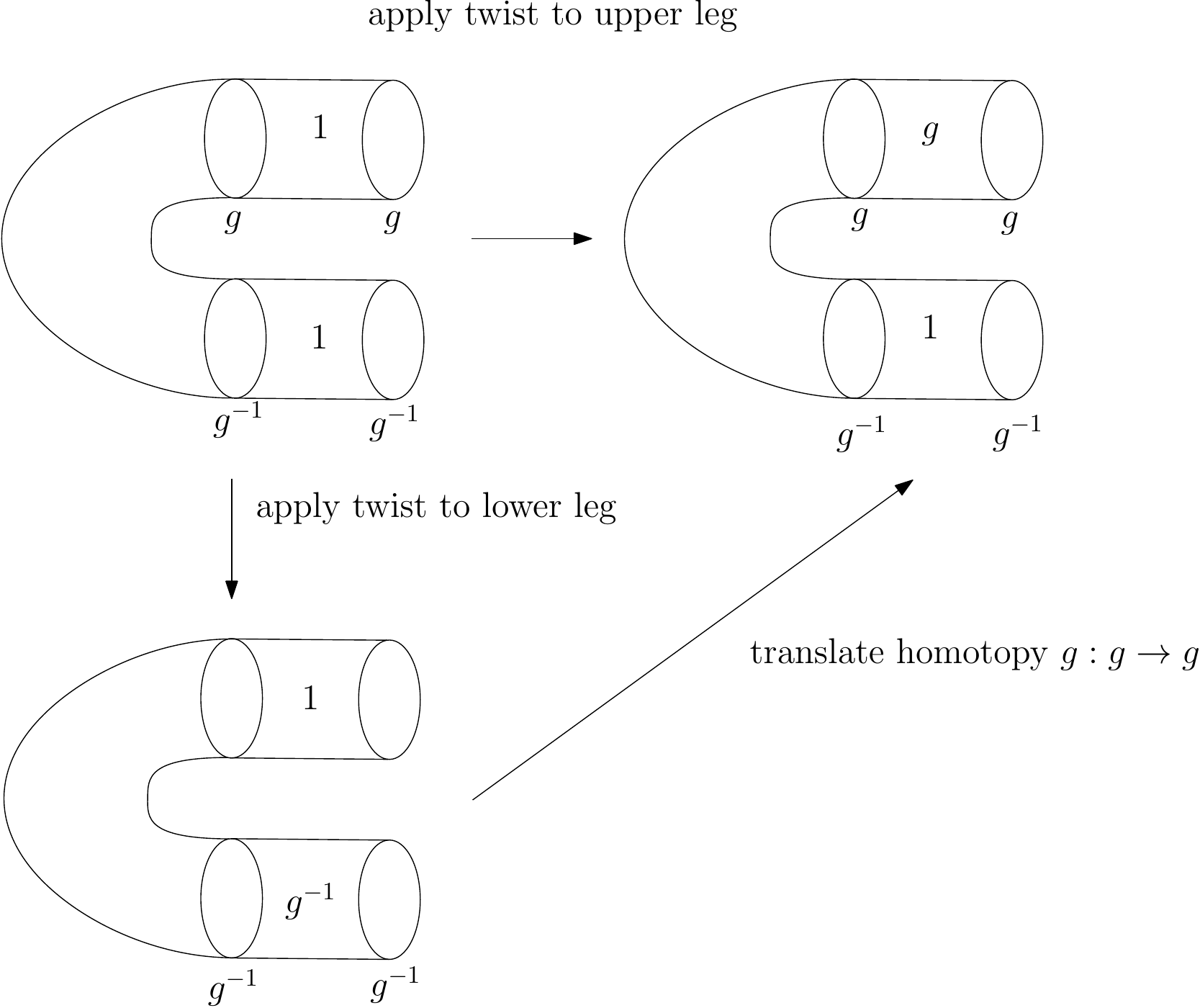}
		\vspace*{0.5cm} 
	\end{center} on the level of bimodules, see page~\pageref{bimodsppage}. By Corollary~\ref{korbimodG} we translate it to the commutative triangle 
\begin{align}\label{eqntwistdualitytriangle}\begin{array}{c}
	\begin{tikzpicture}[scale=2, implies/.style={double,double equal sign distance,-implies},
	dot/.style={shape=circle,fill=black,minimum size=2pt,
		inner sep=0pt,outer sep=2pt},]
	\node (A1) at (0,1) {$\Hom_{\cat{C}_1^Z}(X\otimes X^*,I)$};
	\node (A2) at (4,1) {$\Hom_{\cat{C}_1^Z}(g.X\otimes X^*,I)$};
	\node (B1) at (1,0.5) {};
	\node (B2) at (4,0) {$\Hom_{\cat{C}_1^Z}(X\otimes g^{-1}.X^*,I)$\ .};
	\path[->,font=\scriptsize]
	(A1) edge node[above]{$f \mapsto f \circ (\theta_X^{-1} \otimes \id_{X^*})$} (A2)
	(B2) edge node[right]{$g$} (A2)
	(A1) edge node[left]{$f \mapsto f \circ ( \id_{X} \otimes \theta_{X^*}^{-1}) \qquad $} (B2);
	\end{tikzpicture}\end{array}\end{align}
	Here by abuse of notation we denote by $g$ the map induced by the functor $g: \cat{C}_g^Z \to \cat{C}_g^Z$ on morphism spaces and coherence isomorphisms, i.e.\ the map
	\begin{align}
	\Hom_{\cat{C}_1^Z}(X\otimes g^{-1}.X^*,I) \stackrel{g}{\to} \Hom_{\cat{C}_1^Z}(g.(X\otimes g^{-1}.X^*),g.I)\cong \Hom_{\cat{C}_1^Z}(g.X\otimes X^*,I)\ .
	\end{align} Since the evaluation $\operatorname{eva}_X: X \otimes X^* \to I$ is an element of $\Hom_{\cat{C}_1^Z}(X\otimes X^*,I)$, we obtain from \eqref{eqntwistdualitytriangle}
	\begin{align}
	g. \left(   \operatorname{eva}_X \circ (\id_X \otimes \theta_{X^*}^{-1})  \right) = \operatorname{eva}_X \circ (\theta_X^{-1} \otimes \id_{X^*})\ .
	\end{align} Using that $g: \cat{C}_g^Z \to \cat{C}_g^Z$ is a monoidal functor (Proposition~\ref{satzequivariantmonoidalstructure}) and the compatibility of twist and $G$-action in \ref{prooftwistdualityaction}, this implies
	\begin{align}
	\operatorname{eva}_{g.X} \circ (\theta_X \otimes \id_{g.X^*}) = \operatorname{eva}_X \circ (  \id_X \otimes \theta_{g.X^*}). \end{align}Now a straightforward computation in the graphical calculus using the snake identities for the duality morphisms shows that \eqref{eqndualitytwisteqn1} is indeed equal to \eqref{eqndualitytwisteqn2}.     
	\end{pnum}

\end{proof}

\spaceplease
\begin{remark}\label{homrelaxbmk}
By \cite[Theorem~3.1]{turaevhqft} \emph{non-extended} two-dimensional $G$-equivariant topological field theories are classified by \emph{crossed Frobenius $G$-algebras}, see \cite[II.3.2]{turaevhqft} for a definition of the latter. The structure and properties of such crossed Frobenius $G$-algebras arise from the evaluation of a two-dimensional $G$-equivariant topological field theory on surfaces equipped with bundles, just like a $G$-ribbon category arises from the evaluation of an extended three-dimensional $G$-equivariant topological field theory on surfaces. Hence, we should be able to trace back the occurrence of certain structures and properties to common geometric origins, where of course the $G$-ribbon category lies one categorical level higher than the crossed Frobenius $G$-algebras. For the former equalities, hold up to coherent isomorphism; they are homotopically relaxed. Such a comparison is given in the following table:
\spaceplease
\begin{longtable}[c]{|p{0.2\textwidth}|p{0.3\textwidth}|p{0.3\textwidth}|}
\hline \bfseries\textbf{Geometric origin}  & \textbf{$G$-crossed Frobenius algebra $\mathfrak{A}=\bigoplus_{g\in G}\mathfrak{A}_g$ } & \textbf{$G$-equivariant ribbon category $\cat{C}=\bigoplus_{g\in G} \cat{C}_g$}\normalfont \\ \hline 
homotopies on the cylinder & $G$-action, shifting sectors by conjugation & $G$-action up to coherent isomorphism, shifting sectors by conjugation \\ \hline
pair of pants decorated with bundles & associative and unital product taking $\mathfrak{A}_g \otimes \mathfrak{A}_h$ to $\mathfrak{A}_{gh}$, $G$-action by algebra automorphisms & monoidal product taking $\cat{C}_g \boxtimes \cat{C}_h$ to $\cat{C}_{gh}$, $G$-action by monoidal functors \\ \hline
rotation of the pair of pants & crossed commutativity: $xy=(g.y)x$ for $x \in \mathfrak{A}_g$, $y\in  \mathfrak{A}_h$ & $G$-braiding $X \otimes Y \cong g.Y \otimes X$ for $X \in \cat{C}_g$ and $Y\in \cat{C}_h$\\ \hline
Dehn twist & self-invariance of twisted sectors: $g.x = x$ for $x\in \mathfrak{A}_g$ & $G$-twist $X \cong g.X$ for $X\in \cat{C}_g$\\ \hline
\end{longtable}
\end{remark}

\subsection{Geometric versus algebraic orbifoldization}
Given an extended $G$-equivariant topological field theory $Z: G\text{-}\Cob(3,2,1) \to \TwoVect$ we can evaluate the orbifold theory $Z/G : \Cob(3,2,1) \to \TwoVect$ from Definition~\ref{defofkext} on the circle and obtain a 2-vector space $Z/G(\sphere^1)$. By \cite{BDSPV153D} the topological field theory $Z/G$ can be used to endow $Z/G(\sphere^1)$ with the structure of a complex finitely semisimple ribbon category.  
Using the explicit description of the orbifold theory $Z/G$ in Proposition~\ref{satzorbifoldconcrete} we will now characterize $Z/G(\sphere^1)$ in terms of $\cat{C}^Z = Z(\sphere^1,?)$. This will allow us in Theorem~\ref{thmorbifoldtheorymodular} to relate the geometric orbifold construction of this article to the concept of an orbifold category appearing e.g.\ in \cite{kirrilovg04} or \cite{centerofgradedfusioncategories}.

The following observation can be verified by a direct computation:

\begin{lemma}[]\label{lemmamultplicationfunctorhtpfiber}
For the multiplication functor $M: (G\times G)//G \to G//G$ the homotopy fiber $M^{-1}[g]$ over any $g\in G$ is equivalent to the discrete groupoid with object set
$\{(a,b) \in G\times G \,|\, ab = g\}$. 
\end{lemma}

 Recall from Proposition~\ref{satzorbifoldconcrete} that the orbifold theory $Z/G$ assigns to the circle the 2-vector space of parallel sections of $\cat{C}^Z$. The data of a parallel section of $\cat{C}^Z$ is an object $s(g) \in Z(\sphere^1,g)$ for each $g\in G$ together with coherent isomorphisms $h.s(g)\cong s(hgh^{-1})$ for each $h\in G$. These isomorphisms describe the parallelity up to isomorphism. 

\begin{proposition}[]\label{satzdescriptorbtensorstructure}
Let $G$ be a finite group and $Z : G\text{-}\Cob(3,2,1) \to \TwoVect$ an extended $G$-equivariant topological field theory. The value $Z/G (\sphere^1)$ of the orbifold theory $Z/G$ on the circle naturally carries in the sense of \cite{BDSPV153D} the structure of a complex finitely semisimple ribbon category.
This structure arises in the following way from the structure of $\cat{C}^Z$:
\begin{myenumerate}

\item For $s,s' \in Z/G(\sphere^1)$, up to natural isomorphism, the monoidal product is given by
\begin{align}
(s\otimes s')(g) = \coprod_{ab=g} s(a) \otimes s'(b) \myforall g\in G\ .\end{align} The unit of this monoidal product is the unit of $\cat{C}^Z$ seen as a parallel section in the obvious way. If $\cat{C}^Z$ has a simple unit, then so has $Z/G(\sphere^1)$. \label{satzdescriptorbtensorstructurea}

\item For $s,s' \in Z/G(\sphere^1)$ the braiding isomorphism $s \otimes s' \cong s' \otimes s$ is given by the isomorphisms
\begin{align}
(s \otimes s')(g ) = \coprod_{ab=g} s(a) \otimes s'(b) \to \coprod_{uv=g} s'(u) \otimes s(v) = (s' \otimes s)(g ) \myforall g\in G
\end{align} which map the summand $(a,b)$ to the summand $(aba^{-1},a)$ by
\begin{align}
s(a) \otimes s'(b) \xrightarrow{c_{  s(a) , s'(b)   }} a.s'(b) \otimes s(a) \xrightarrow{\text{parallelity}} s'(aba^{-1}) \otimes s(a).
\end{align}\label{satzdescriptorbtensorstructureb}

\item For $s\in Z/G(\sphere^1)$ the twist is given by
\begin{align}
s(g) \xrightarrow{\theta_{s(g)}} g.s(g) \xrightarrow{\text{parallelity}} s(ggg^{-1}) = s(g) \myforall g\in G\ .
\end{align}

\end{myenumerate}
 
\end{proposition}

\begin{proof}
$\phantom{X}$
\begin{myenumerate}

\item The monoidal product is obtained from the pair of pants. Hence, using the span \eqref{eqnpairofpantsspan} and the concrete description of the orbifold construction in Proposition~\ref{satzorbifoldconcrete}, \ref{satzorbifoldconcreteb} we find
\begin{align} (s\otimes s')(g) = \lim_{(a,b,h) \in M^{-1}[g]} h. (s(a) \otimes s(b))\ .\end{align}
Since $G$ acts by monoidal functors (Proposition~\ref{satzequivariantmonoidalstructure}) and $s$ and $s'$ are parallel, this reduces to
\begin{align} (s\otimes s')(g) \cong  \lim_{(a,b,h) \in M^{-1}[g]} s(hah^{-1}) \otimes s(hbh^{-1})\ .\end{align}
Now Lemma~\ref{lemmamultplicationfunctorhtpfiber} yields the assertion if we take into account that finite coproducts and finite products in a 2-vector space coincide. 
The monoidal unit can also be obtained by Proposition~\ref{satzorbifoldconcrete}, \ref{satzorbifoldconcreteb}. Alternatively, we can just use that the given object is a unit for the monoidal product and hence the unique one up to isomorphism. 

We need to prove the additional statement on the simplicity of units: The unit of $Z/G(\sphere^1)$ is $I$ with the canonical isomorphisms $\phi_g : g.I\cong I$ coming from the fact that $G$ acts by monoidal functors. Hence, an endomorphism of the unit of $Z/G(\sphere^1)$ is a morphism $\psi: I \to I$ such that $\phi_g \circ (g.\psi) = \psi \circ \phi_g$ for all $g\in G$. If $I$ is simple, then $\psi = \lambda \id_I$ for some $\lambda \in \mathbb{C}$ and the requirement $\phi_g \circ (g.\psi) = \psi \circ \phi_g$ is automatically true since $g$ acts as a $\mathbb{C}$-linear functor. This proves that an endomorphism of the unit of $\cat{C}^Z$ is the same as an endomorphism of the unit of $Z/G(\sphere^1)$. Therefore, the unit of $Z/G(\sphere^1)$ is simple as well.

\item The evaluation of the stack $\Pi(?,BG)$ on the 2-cell in $\Cob(3,2,1)$ that we used to produce the braiding yields the span of spans
	\begin{center}
\begin{tikzpicture}[scale=2,     implies/.style={double,double equal sign distance,-implies},
dot/.style={shape=circle,fill=black,minimum size=2pt,
	inner sep=0pt,outer sep=2pt},]
\node (A1) at (0,0) {$G//G \times G//G$};
\node (A2) at (2,1) {$(G\times G)//G$};
\node (A3) at (4,0) {$G//G$};
\node (B2) at (2,-1) {$(G\times G)//G$};
\node (C) at (2,0) {$(G\times G)//G$};
\node (B1) at (1,-0.5) {$$};
\node (B3) at (2,-1) {$$};
%\node (B2) at (1,0) {$B\Omega$};
\path[->,font=\scriptsize]
(A2) edge node[above]{$B$} (A1)
(A2) edge node[above]{$M$} (A3)
(B2) edge node[below]{$B$} (A1)
(B2) edge node[below]{$M$} (A3)
%(A4) edge node[above]{$i$} (A5)
(C) edge node[right]{$R$} (B2)
(C) edge node[right]{$=$} (A2);
\draw (C) edge[implies] node[above] {\scriptsize$\alpha$} (A1);
\end{tikzpicture},
\end{center} where $R : (G \times G)//G \to (G \times G)//G $ is the functor $(g,h) \mapsto (ghg^{-1},g)$ and $\alpha$ is the obvious natural transformation. By Proposition~\ref{satzorbifoldconcrete}, \ref{satzorbifoldconcretec} the braiding isomorphism $(s \otimes s')(g) \cong (s'\otimes s)(g)$ is given as follows: We start with 

\begin{align}
(s \otimes s')(g ) = \lim_{(a,b,h) \in M^{-1}[g]} s(hah^{-1}) \otimes s'(hbh^{-1})\ , \end{align} apply vertex-wise the equivariant braiding, i.e.\ the isomorphisms \begin{align} s(hah^{-1}) \otimes s'(hbh^{-1}) \cong (hah^{-1}).s'(hbh^{-1}) \otimes s(hah^{-1})\ ,  
\end{align} use parallelity
\begin{align}  (hah^{-1}).s'(hbh^{-1}) \otimes s(hah^{-1})\cong s'(haba^{-1}h^{-1}) \otimes s(hah^{-1})   
\end{align} and push the resulting limit
\begin{align}
\lim_{(a,b,h) \in M^{-1}[g]} (hah^{-1}).s'(hbh^{-1}) \otimes s(hah^{-1})\cong s'(haba^{-1}h^{-1}) \otimes s(hah^{-1})  
\end{align} along the equivalence $M^{-1}[g] \cong M^{-1}[g]$ induced by $R$. Using the identifications made in \ref{satzdescriptorbtensorstructurea} based on Lemma~\ref{lemmamultplicationfunctorhtpfiber} the assertion follows.

\item The proof of this assertion follows also from Proposition~\ref{satzorbifoldconcrete}.

\end{myenumerate}

\end{proof}

 In order to compare Proposition~\ref{satzdescriptorbtensorstructure} to the concept of an orbifold category, let us recall the latter from \cite[Lemma~2.3 and Theorem~3.9]{kirrilovg04}:

\spaceplease 
\begin{proposition}[Algebraic orbifoldization of an equivariant ribbon category from \cite{kirrilovg04}]\label{satzalgofk}
	Let $G$ be a finite group and $\cat{C}$ a complex finitely semisimple $G$-ribbon category, then the orbifold category $\cat{C} / G$ (the category of homotopy fixed points), i.e.\ the category of objects $X$ in $\cat{C}$ together with a family of coherent isomorphisms $(\chi_g : g.X \to X)_{g\in G}$ inherits the following structure from $\cat{C}$:  
	\begin{myenumerate}
	\item By
	\begin{align} (X,(\chi_g)_{g\in G}) \otimes (Y,(\lambda_g)_{g\in G}) := (X \otimes Y,(\chi_g \otimes \lambda_g)_{g\in G}) \end{align} for all   $(X,(\chi_g)_{g\in G}) , (Y,(\lambda_g)_{g\in G}) \in \cat{C} / G$  it is made into a monoidal category with the monoidal unit in $\cat{C}$ (seen as a homotopy fixed point) as the monoidal unit. The monoidal category $\cat{C} / G$ has duals.
	\item The monoidal category $\cat{C} /G$ is braided and the underlying isomorphism $X \otimes Y \to Y \otimes X$ for objects \begin{align} \left(X=\bigoplus_{g\in G} X_g ,(\chi_g)_{g\in G}\right) , \left(Y=\bigoplus_{g \in G} Y_g ,(\lambda_g)_{g\in G}\right) \in \cat{C} / G\end{align} is given by
	\begin{align}
	X_g \otimes Y_h \xrightarrow{c_{X_g,Y_h}} g.Y_h \otimes X_g \xrightarrow{\lambda_g \otimes \id_{X_g}} Y_h \otimes X_g \myforall g,h \in G.
	\end{align} 
	
	\item The braided monoidal category $\cat{C} / G$ comes with a twist which on the object \begin{align}\left(X=\bigoplus_{g\in G} X_g ,(\chi_g)_{g\in G}\right)\end{align} arises from the equivariant twist by
	\begin{align} X_g \xrightarrow{\theta_{X_g}} g.X_g \xrightarrow{\chi_g} X_g.\end{align} 
	
	\end{myenumerate}
	\end{proposition}

 We can now state our comparison result: 

\begin{theorem}[]\label{thmorbifoldtheorymodular} For any extended $G$-equivariant topological field theory $Z : G\text{-}\Cob(3,2,1) \to \TwoVect$ the evaluation of the orbifold theory $Z/G: \Cob(3,2,1) \to \TwoVect$ on $\sphere^1$ yields an equivalence
	\begin{align} \frac{Z}{G}(\sphere^1) \cong \frac{\cat{C}^Z}{G}\label{eqmodcatcompare}\end{align} as 2-vector spaces.
	Both categories carry the structure of a complex finitely semisimple ribbon category:
	\begin{itemize}
		\item $Z/G(\sphere^1)$ by being the value of an extended topological field theory on the circle in the sense of Proposition~\ref{satzdescriptorbtensorstructure}.
		
		\item $\cat{C}^Z /G$ by Proposition~\ref{satzalgofk}.
		
	\end{itemize}
	Both structures agree, i.e.\ \eqref{eqmodcatcompare} is true on the level of complex finitely semisimple ribbon categories. 
	
\end{theorem}

\begin{proof}
	The equivalence $Z/G(\sphere^1) \cong \cat{C}^Z/G$ of 2-vector spaces holds by definition of the orbifold construction and the definition of the orbifold category in \cite{kirrilovg04}. By Proposition~\ref{satzalgofk} the category $\cat{C}^Z/G$ naturally inherits from $\cat{C}^Z$ the structure of an complex finitely semi-simple ribbon category, and by Proposition~\ref{satzdescriptorbtensorstructure} the category $Z/G(\sphere^1)$ has the same type of structure. Comparing the description of these structures as given in Proposition~\ref{satzalgofk} and Proposition~\ref{satzdescriptorbtensorstructure} shows that agree.     
\end{proof}

 Diagrammatically, the above Theorem means that the square 
	\begin{center}
\begin{tikzpicture}[scale=2, implies/.style={double,double equal sign distance,-implies},
dot/.style={shape=circle,fill=black,minimum size=2pt,
	inner sep=0pt,outer sep=2pt},]
\node (A1) at (0,1) {$\substack{\text{3-2-1-dimensional } G\text{-equivariant} \\ \text{topological field theories} } $};
\node (A2) at (4,1) {$\substack{\text{complex finitely semisimple} \\ G\text{-ribbon categories} } $};
\node (B1) at (0,0) {$\substack{\text{3-2-1-dimensional} \\ \text{topological field theories} } $};
\node (B2) at (4,0) {$\substack{\text{complex finitely semisimple} \\ \text{ribbon categories} } $};
\path[->,font=\scriptsize]
(A1) edge node[above]{evaluation on the circle} (A2)
(A1) edge node[left]{orbifoldization $?/G$} (B1)
(A2) edge node[right]{orbifold category} (B2)
(B1) edge node[below]{evaluation on the circle} (B2);
\end{tikzpicture}
\end{center} commutes up to natural isomorphism.

\begin{corollary}[]\label{kororbifoldtheorymodular}
	For any extended $G$-equivariant topological field theory $Z : G\text{-}\Cob(3,2,1) \to \TwoVect$ the orbifold theory $Z/G: \Cob(3,2,1) \to \TwoVect$ is determined up to equivalence by the orbifold category $\cat{C}^Z/G$.
	\end{corollary}

\begin{proof}
	This follows from Theorem~\ref{thmorbifoldtheorymodular} if we take into account that by \cite{BDSPV153D} any 3-2-1-dimensional topological field theory is determined up to equivalence by the complex finitely semisimple ribbon category it yields on the circle.     
	\end{proof}

 As an application we can give a generalization of \cite[Example~4.7]{schweigertwoikeofk} concerned with the orbifoldization of equivariant Dijkgraaf-Witten theories: 

\begin{proposition}[]\label{satzequivdwmodel}
	Let $Z_\lambda : J\text{-}\Cob(3,2,1) \to \TwoVect$ be the extended $J$-equivariant Dijkgraaf-Witten theory constructed in  \cite{maiernikolausschweigerteq} from a short exact sequence $0\to G \to H \stackrel{\lambda}{\to} J \to 0$ of finite groups. The orbifold theory $Z_\lambda /J$ is equivalent to the extended Dijkgraaf-Witten theory $Z_H$ for the group $H$, i.e.\
	\begin{align} \frac{Z_\lambda}{J}\cong Z_H\ .\end{align}
\end{proposition}

\begin{proof}
	In \cite[Proposition~35]{maiernikolausschweigerteq} the orbifold category $\cat{C}^{Z_\lambda}/J$ of $\cat{C}^{Z_\lambda}$ is computed to be category $D(H)\text{-}\Mod$ of finite-dimensional modules over the Drinfeld double $D(H)$ of the group $H$. By Theorem~\ref{thmorbifoldtheorymodular} this is the category that $Z_\lambda /J$ assigns to the circle. Since this category is also the value of $Z_H$ on the circle we can use Corollary~\ref{kororbifoldtheorymodular} to deduce the desired assertion.     
\end{proof}

 One should appreciate that this statement, although more general, admits a significantly simpler and more conceptual proof than the corresponding statement in \cite{schweigertwoikeofk} because it can be completely played back to the categories obtained on the circle.

In another application we will use topological field theory as a counting device: For this let us first recall the following well-known fact which in a different language appears for instance in \cite[Corollary~IV.12.1.2]{turaev1}:

\begin{lemma}\label{lemmanumberofsimples}
Let $Z:\Cob(n,n-1,n-2) \to \TwoVect$ be an extended topological field theory, then
\begin{align}
Z(\torus^n) = \# \ \text{simple objects in $Z(\torus^{n-2})$} \ . 
\end{align}
\end{lemma}

\begin{proof}
Set $\cat{C}:=Z(\torus^{n-2})$, then $\cat{C}$ is dualizable in (the homotopy category of) $\TwoVect$
and the vector space assigned to $\torus^{n-1}=\torus^{n-2}\times\mathbb{S}^1$ is the concatenation of the coevaluation and evaluation of $\cat{C}$, which is given by $\bigoplus_{j=1}^n \Hom_\cat{C}(X_j,X_j)$ where the sum runs over the simple objects. The dimension of this vector space is the number of simple objects. By \cite[Theorem~III.2.1.3]{turaev1} this number is also the invariant that $Z$ assigns to the top-dimensional manifold $\mathbb{T}^n$. 
\end{proof}

In order to combine this fact with the orbifold construction, we recall that the groupoid of $G$-bundles over $\torus^n$ for $n\ge 1$ is equivalent to the action groupoid $\Com (G^n)//G$ of the action of $G$ on $n$-tuples of mutually commuting group elements by conjugation. Hence, a $G$-bundle over $\torus^n$ can be described by $n$ group elements $g_1,\dots,g_n \in G$ such that $g_ig_j = g_jg_i$ for all $1\le i,j\le n$. 

\begin{theorem}[]\label{thmnumberofsimpleobjectsorbifold}
	Let $G$ be a finite group and $Z: G\text{-}\Cob(n,n-1,n-2) \to \TwoVect$ an extended $G$-equivariant topological field theory. Then
	\begin{align}
  	\# \ \text{simple objects in $\frac{Z}{G}(\torus^{n-2})$} =   \frac{1}{|G|} \sum_{(g_1,\dots,g_n) \in \Com (G^n)}   Z(\torus^n,g_1,\dots,g_n) \ . \label{dimensionformula1}
	\end{align} For $n=3$ we also find the formula
	\begin{align}
	\# \ \text{simple objects in $\frac{\cat{C}^Z}{G}$} =\frac{1}{|G|} \sum_{(g_1,g_2,g_3) \in \Com (G^3)}   Z(\torus^3,g_1,g_2,g_3) \label{dimensionformula2}
	\end{align} using the orbifold category $\cat{C}^Z/G$ of the $G$-ribbon category $\cat{C}^Z$ that $Z$ gives rise to.
	\end{theorem}

\spaceplease
\begin{proof}
	Once we prove \eqref{dimensionformula1}, formula \eqref{dimensionformula2} will follow from Theorem~\ref{thmorbifoldtheorymodular}.
	Hence, we only have to prove \eqref{dimensionformula1}:
	By Lemma~\ref{lemmanumberofsimples} we find
	\begin{align}
	\# \ \text{simple objects in $\frac{Z}{G}(\torus^{n-2})$} = \frac{Z}{G}(\torus^n)\ .
	\end{align} The number $Z/G (\torus^n)$ can be computed using the non-extended orbifold construction. Knowing the groupoid of $G$-bundles over $\torus^n$ we can use \cite[Corollary~4.4 (c)]{schweigertwoikeofk} to express $Z/G(\torus^n)$ as the integral
	\begin{align}
	\frac{Z}{G}(\torus^n) = \int_{(g_1,\dots,g_n) \in \Com (G^n)//G} Z(\torus^n,g_1,\dots,g_n) = \sum_{[g_1,\dots,g_n] \in \pi_0 (   \Com (G^n)//G  )} \frac{Z(\torus^n,g_1,\dots,g_n)}{|\hspace{-.16667em} \Aut(g_1,\dots,g_n)|}
	\end{align} with respect to groupoid cardinality. By the orbit stabilizer Theorem we obtain \begin{align} |   \hspace{-.16667em}   \Aut(g_1,\dots,g_n)| = \frac{|G|}{ |\mathcal{O}(g_1,\dots,g_n)|}\ , \end{align} where $\mathcal{O}(g_1,\dots,g_n)$ is the orbit of $(g_1,\dots,g_n)$ in $\Com (G^n)//G$. This implies
		\begin{align}
	\frac{Z}{G}(\torus^n) = \frac{1}{|G|} \sum_{(g_1,\dots,g_n) \in \Com (G^n)}   Z(\torus^n,g_1,\dots,g_n) 
	\end{align} and hence the result.     
	\end{proof}

Even in the non-extended case we can read off from the above proof that
\begin{align}
\frac{1}{|G|} \sum_{(g_1,\dots,g_n) \in \Com (G^n)}   Z(\torus^n,g_1,\dots,g_n)  = \frac{Z}{G}(\torus^n) = \dim \frac{Z}{G}(\torus^{n-1})\label{countingeqngtft}
\end{align} is a non-negative integer. This provides constraints for manifold invariants which arise from a (not necessarily extended) equivariant topological field theory:

\begin{corollary}
	Consider an invariant of closed oriented $n$-dimensional manifolds decorated with $G$-bundles for a finite group $G$ which yields on the torus $\torus^n$ decorated with the bundle specified by commuting group elements $(g_1,\dots,g_n) \in G^n$ the number $z_{g_1,\dots,g_n} \in \mathbb{C}$. If the invariant arises from an $G$-equivariant topological field theory, then
$\sum_{(g_1,\dots,g_n) \in \Com (G^n)} z_{g_1,\dots,g_n}$ is a non-negative integer multiple of $|G|$. 
	\end{corollary}

\begin{example}[Permutation orbifolds]\label{expermutationorbifolds}
	Let $\cat{C}$ be a modular category and $Z: \Cob(3,2,1) \to \TwoVect$ the extended topological field theory giving us $\cat{C}$ upon evaluation on the circle ($Z$ is unique up to equivalence). Consider now a finite group, which for illustration purposes we take to be the permutation group $\operatorname{S}_n$ on $n$ letters (this is not really a restriction because any finite group embeds into a permutation group). The pullback $\operatorname{Cov}^* Z$ of $Z$ along the cover functor $\operatorname{Cov} : \operatorname{S}_n\text{-}\Cob(3,2,1) \to \Cob(3,2,1)$ from Example~\ref{extotalspacefunctor} is a $\operatorname{S}_n$-equivariant topological field theory. Using Theorem~\ref{thmorbifoldtheorymodular} we see that the evaluation of the orbifold theory $(\operatorname{Cov}^* Z) / \operatorname{S}_n$ on the circle is what is commonly referred to as the \emph{permutation orbifold} of $\cat{C}$ and which is denoted by $\cat{C} \wr \operatorname{S}_n$ in \cite{bantay98,bantay02}. Since a permutation orbifold is a special case of an orbifold theory, we can use Theorem~\ref{thmnumberofsimpleobjectsorbifold} to compute the number of simple objects of $\cat{C} \wr \operatorname{S}_n$. 
	
	To this end, note that for any finite group $G$ and mutually commuting groups elements $g_1,g_2,g_3 \in G$ we can define the quotient $P_{g_1,g_2,g_3}$ of $\mathbb{R}^3 \times G$ by
	\begin{align} (x_1+1,x_2,x_3,h) &\sim (x_1,x_2,x_3,hg_1)\ ,\\   (x_1,x_2+1,x_3,h) &\sim (x_1,x_2,x_3,hg_2)\ ,  \\ (x_1,x_2,x_3+1,h) &\sim (x_1,x_2,x_3,hg_3) \end{align} for all $x_1,x_2,x_3 \in \mathbb{R}$ and $h\in G$. The projection $\mathbb{R}^3 \times G \to \mathbb{R}^3$ induces a map $P_{g_1,g_2,g_3} \to \torus^3$, which is a $G$-bundle with holonomy values $g_1,g_2$ and $g_3$ along the generators of the fundamental group of $\torus^3$. The subgroup $\langle g_1,g_2,g_3\rangle \subset G$ generated by $g_1,g_2$ and $g_3$ acts from the right on $G$. It is easy to see that
	\begin{align}
	P_{g_1,g_2,g_3} \cong \coprod_{| G / \langle g_1,g_2,g_3\rangle| } \torus^3
	\end{align} as manifolds. 
	
	Going back to $G=\operatorname{S}_n$ we find by Theorem~\ref{thmnumberofsimpleobjectsorbifold}
	\begin{align}
	\# \ \text{simple objects in $\cat{C} \wr \operatorname{S}_n$} = \frac{1}{n!} \sum_{\substack{ \text{mutually commuting} \\ \text{permutations} \\ \sigma_1,\sigma_3,\sigma_3 \\ \text{on $n$ letters}        }}   \left(  \# \ \text{simple objects in $\cat{C}$} \right) ^{  |  \operatorname{S}_n / \langle  \sigma_1 , \sigma_2 , \sigma_3  \rangle       |      }.
	\end{align} Hence, Theorem~\ref{thmnumberofsimpleobjectsorbifold} specializes to the formula given in \cite[Equation~(3)]{bantay02}. In fact, our orbifold construction allows for a uniform treatment of the entire theory of permutation orbifolds. 
	
	In \cite{muellerwoike} we also explain how Theorem~\ref{thmnumberofsimpleobjectsorbifold} yields the formulae for the number of simple twisted representations of finite groups and the number of simple representations of twisted Drinfeld doubles of finite groups found in \cite{willterongerbesgrpds}.

	\end{example}

\subsection{Equivariant Verlinde algebra and modularity\label{secmod}}
The evaluation of a 3-2-1-dimensional topological field theory on the circle yields a modular tensor category by \cite{BDSPV153D} (possibly with non-simple unit, see however \cite[Lemma~5.3]{BDSPV153D}). 
In this section we give the equivariant version of this result. To make contact to an equivariant modularity we use the equivariant Verlinde algebra from \cite{kirrilovg04} whose definition can be understood by evaluation of the modular functor corresponding to the equivariant theory on the 2-torus $\torus^2$, see \cite[Section~8]{kirrilovg04}, which is inspired by \cite[Section~8.6]{turaevhqft}. We begin by working out these ideas in the language of coends and based on a strong geometric motivation.

Let $Z: G\text{-}\Cob(n,n-1,n-2) \to \TwoVect$ be an extended $G$-equivariant topological field theory. Any $(n-1)$-dimensional closed oriented manifold $\Sigma$ together with a map $\varphi : \Sigma \to BG$ gives rise to a 2-linear map $Z(\Sigma,\varphi) : \FinVect \to \FinVect$ and hence to a vector space, 
which by abuse of notation we will also denote by $Z(\Sigma,\varphi)$. The dependence on $\varphi$ is functorial, so we get a functor
\begin{align}
Z(\Sigma,?) : \Pi(\Sigma,BG) \to \FinVect, \quad \varphi \mapsto Z(\Sigma,\varphi),
\end{align} i.e.\ a representation of (or in more geometric terms: a vector bundle over) the groupoid of $G$-bundles over $\Sigma$. Clearly, this is the representation we obtain be seeing $Z$ as a non-extended theory and applying \cite[Proposition~2.8]{schweigertwoikeofk}. 

These vector bundles enjoy the following gluing properties which follow directly from the functoriality of $Z$ and \eqref{gluinglawbimodules}:

\begin{lemma}[]\label{lemmagluingreps}
Let $G$ be a finite group,
 $Z: G\text{-}\Cob(n,n-1,n-2) \to \TwoVect$ an extended $G$-equivariant topological field theory and $\Sigma$ a closed oriented $(n-1)$-dimensional manifold obtained by gluing the oriented $(n-1)$-dimensional manifolds $\Sigma'$ and $\Sigma''$ along the $(n-2)$-dimensional closed oriented manifold $S$. 
 Then for two maps $\varphi' : \Sigma' \to BG$ and $\varphi'' : \Sigma'' \to BG$ with $\varphi'|_S = \varphi''|_S =: \xi$ we have
\begin{align}
Z(\Sigma, \varphi' \cup_S \varphi'') \cong \int^{X\in Z(S,\xi)} Z(\Sigma'',\varphi'') X \otimes \Hom_{Z(S,\xi)} (X, Z(\Sigma',\varphi')\mathbb{C})
\end{align} by a canonical isomorphism of vector spaces. 
\end{lemma}

 A particularly important special case arises if $\Sigma$ is the 2-torus $\torus^2$. By the holonomy classification of flat bundles the groupoid of $G$-bundles over the torus is equivalent to the full subgroupoid of $\Com (G^n)//G\subset (G\times G)//G$ consisting of pairs of commuting group elements.

\begin{proposition}[]\label{satzmodfunctorontorus}
Let $Z: G\text{-}\Cob(3,2,1) \to \TwoVect$ be an extended $G$-equivariant topological field theory, then for all $g,h\in G$ with $gh=hg$
\begin{align} Z(\torus^2)(g,h) \cong \int^{X \in \cat{C}_h^Z} \Hom_{\cat{C}_h^Z}(g.X ,X) \end{align} by a canonical isomorphism of vector spaces. 
\end{proposition}

\begin{proof}
We can cut the torus with bundle decoration $(g,h)$, i.e.\ with a $G$-bundle with holonomies $g$ and $h$, respectively, along the generators of the fundamental group, as indicated in the following picture:
\begin{center}
		\vspace*{0.5cm}
		\includegraphics[width=0.3\textwidth]{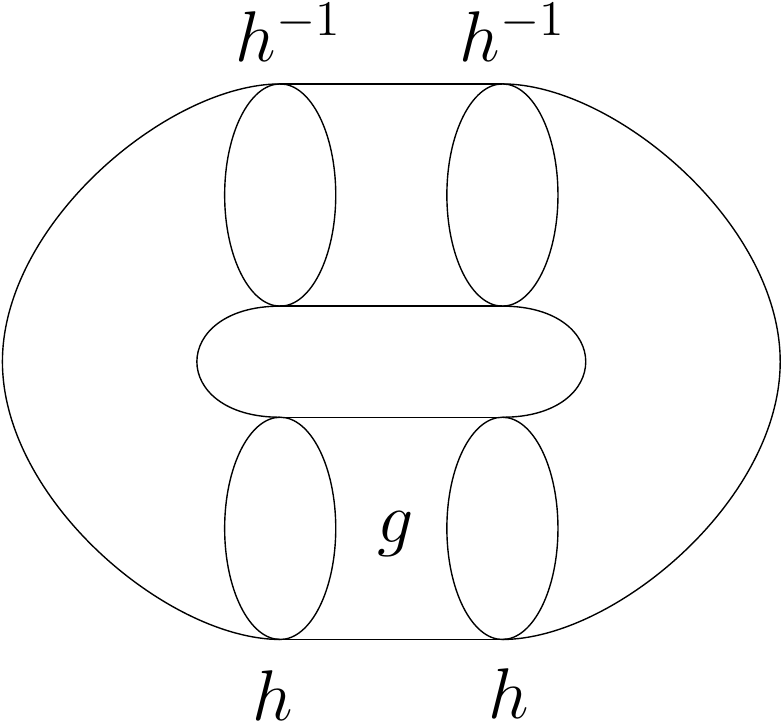}
		\vspace*{0.5cm}
	\end{center}
We want to apply Lemma~\ref{lemmagluingreps} with

\begin{itemize}

\item $(\Sigma'',\varphi'')$ given by the bent cylinder $B_h$ as described in Example~\ref{exbentcylinder} (that is the right third of the above picture),

\item $(\Sigma',\varphi')$ given by the same bent cylinder read backwards with two cylinders glued to it such that the lower leg is equipped with $g$ (that is the left and the middle third of the above picture glued together).

\end{itemize}
Hence, $(S,\xi)$ is given by two copies of the circle with $h$ and $h^{-1}$ on it.
By Example~\ref{exbentcylinder} we find for $X \in \cat{C}_h^Z$ and $Y \in \cat{C}_{h^{-1}}^Z$
\begin{align} Z(\Sigma'',\varphi'') (X \boxtimes Y) \cong \Hom_{\cat{C}_1^Z}(I,X \otimes Y) \end{align} and similarly (i.e.\ by means of Corollary~\ref{korbimodG}) 
\begin{align}
\Hom_{Z(S,\xi)} (X\boxtimes Y, Z(\Sigma',\varphi')\mathbb{C}) \cong \Hom_{\cat{C}_1^Z}(g.X \otimes Y,I)\ .
\end{align}
Now by applying Lemma~\ref{lemmagluingreps} we obtain
\begin{align}
Z(\torus^2,(g,h)) \cong \int^{X \boxtimes Y \in \cat{C}_h^Z \boxtimes \cat{C}_{h^{-1}}^Z} \Hom_{\cat{C}_1^Z}(I,X \otimes Y) \otimes \Hom_{\cat{C}_1^Z}(g.X \otimes Y,I)\ .
\end{align} By Fubini's Theorem for coends and $\Hom_{\cat{C}_1^Z}(I,X \otimes Y)\cong \Hom_{\cat{C}_h^Z}(Y^*,X)$ we find
\begin{align} Z(\torus^2,(g,h)) &\cong \int^{X \in \cat{C}_h^Z} \int^{Y \in \cat{C}_{h^{-1}}^Z} \Hom_{\cat{C}_h^Z}(Y^*,X) \otimes \Hom_{\cat{C}_1^Z}(g.X \otimes Y,I) \\ &= \int^{X \in \cat{C}_h^Z} \int^{Y \in \cat{C}_{h}^Z} \Hom_{\cat{C}_h^Z}(Y,X) \otimes \Hom_{\cat{C}_h^Z}(g.X ,Y) \ ,
\end{align} where in the last step we used the substitution $Y \mapsto Y^*$ and \begin{align}
\Hom_{\cat{C}_1^Z}(g.X \otimes Y,I) \cong \Hom_{\cat{C}_h^Z}(g.X ,Y^*)\ . \end{align} By the co-Yoneda Lemma (compare to Example~\ref{exbentcylinder}) we arrive at
\begin{align}
 Z(\torus^2,(g,h)) &\cong  \int^{X \in \cat{C}_h^Z} \Hom_{\cat{C}_h^Z}(g.X ,X)\ .     
\end{align}
\end{proof}

\begin{remark}
A map from the surface $\Sigma_g$ of genus $g$ to $BG$ can equivalently be described by a morphism $\varphi : \pi_1(\Sigma_g) \to G$ from the fundamental group of $\Sigma_g$ to $G$. We denote by $a_1,\dots,a_g,b_1,\dots,b_g$ usual generators of $\pi_1(\Sigma_g)$ subject to the relation $\prod_{j=1}^g [a_j,b_j] =1$. 
With similar methods, duality and the fact $\Hom_{\cat{C}_1^Z}(I,?)$ is exact and hence preserves finite colimits we find
\begin{align}
Z (\Sigma_g, \varphi ) \cong \Hom_{\cat{C}_1^Z}(I , L_\varphi)\ ,
\end{align} where $L_\varphi$ is the coend
\begin{align}
L_\varphi := \bigotimes_{j=1}^g L_\varphi^j, \quad L_\varphi^j := \int^{X_j \in \cat{C}_{\varphi(a_j)}^Z }  X_j \otimes \varphi(b_j).X_j^*\ .
\end{align}
These formulae can be found in \cite[VII.3.3]{turaevhqft}, where they are used as a definition to build a $G$-modular functor from an appropriate type of $G$-category. 
Above we have followed the converse logic and started with a given extended $G$-equivariant topological field theory, extracted this category and the corresponding modular functor and derived these formulae. 
\end{remark}

If we denote by $P$ the pair of pants, then evaluation of $Z$ on the bordism $\sphere^1\times P : \torus^2 \coprod \torus^2 \to \torus^2$ appropriately decorated with $G$-bundles yields linear maps
\begin{align}
Z(\torus^2)(g,h) \otimes Z(\torus^2)(g,h') \to Z(\torus^2)(g,hh') \myforall g,h,h' \in G\ ,
\end{align} which extends by zero to an associative multiplication on the total space \begin{align} \bigoplus_{\substack{g,h\in G\\ gh=hg}} Z(\torus^2)(g,h) \cong \bigoplus_{\substack{g,h\in G\\ gh=hg}} \int^{X \in \cat{C}_h^Z} \Hom_{\cat{C}_h^Z}(g.X ,X) \label{eqeqverlindealgg}\ .\end{align} The vector space \eqref{eqeqverlindealgg} together with this multiplication is called the \emph{equivariant Verlinde algebra of $Z$}. It helps to prove the following:

\spaceplease 

\begin{proposition}[]\label{satztwistedsectorsnontrivial}
Let $G$ be a finite group and $Z: G\text{-}\Cob(3,2,1) \to \TwoVect$ an extended $G$-equivariant topological field theory such that the monoidal unit of $\cat{C}^Z$ is simple. Then all twisted sectors $\cat{C}_g^Z$ for $g\in G$ are non-trivial, i.e.\ different from the zero 2-vector space. \end{proposition}

\begin{proof}
It is well-known that the mapping class group of the torus has an element $S:\torus^2 \to \torus^2$ such that the bundle $(g,h)$ is pulled back along $S$ to the bundle $(h^{-1},g)$. Hence, the evaluation of $Z$ on the invertible 2-morphism in $G\text{-}\Cob(3,2,1)$ built from $S$ (Remark~\ref{bmkextcob}, \ref{bmkextcobdiffeo}) yields an isomorphism $Z(\torus^2)(g,h) \cong Z(\torus^2)(h^{-1},g)$ for $g,h\in G$; in particular 
\begin{align}
Z(\torus^2)(g,1) \cong Z(\torus^2)(1,g) \myforall g\in G\ . \label{eqncompareveontorus}
\end{align} Suppose now $\cat{C}_g^Z=0$ for some $g\neq 1$. Then $Z(\torus^2)(1,g)=0$ by Proposition~\ref{satzmodfunctorontorus} and hence $Z(\torus^2)(g,1)=0$ by \eqref{eqncompareveontorus}. On the other hand, if we complete the unit $I \in \cat{C}_1^Z$ to a basis $(I,(B_j)_{j\in J})$ of simple objects for $\cat{C}_1^Z$, we find by Proposition~\ref{satzmodfunctorontorus}
\begin{align}
Z(\torus^2)(g,1) \cong \Hom_{\cat{C}_1^Z} (g.I,I) \oplus \bigoplus_{j\in J} \Hom_{\cat{C}_1^Z} (g.B_j,B_j)\ .
\end{align} By Proposition~\ref{satzequivariantmonoidalstructure} the element $g$ acts as a monoidal functor, so $\Hom_{\cat{C}_1^Z} (g.I,I) \cong \Hom_{\cat{C}_1^Z} (I,I)\cong \mathbb{C}$ leading to $Z(\torus^2)(g,1) \neq 0$ and hence to a contradiction.     
\end{proof}

\begin{example}[]\label{exequivcatnonsplit}
	The statement of Proposition~\ref{satztwistedsectorsnontrivial} is false if we do not assume the simplicity of the monoidal unit: Let $Z: \Cob(3,2,1) \to \TwoVect$ be a non-equivariant extended topological field theory such that the unit of $\cat{C}^Z := Z(\sphere^1)$ is simple. Then by \cite{BDSPV153D} the category $\cat{C}^Z$ is modular. If we push $Z$ along the group morphism $\iota : \{1\} \to G$ for some finite group $G$ using the pushforward construction of Section~\ref{secpushofk}, we obtain a $G$-equivariant topological field theory $\iota_* Z$. Evaluation of $\iota_* Z$ on the circle yields the category $\cat{C}^{\iota_* Z}$ with trivial twisted sectors and neutral sector $\cat{C}_1^{\iota_* Z} = \bigoplus_{g\in G} \cat{C}^Z$. The action by $h\in G$ sends the copy for $g$ to the copy for $hg$. If we denote by $I^g$ the unit $I$ of $\cat{C}^Z$ in the copy for $g\in G$, then the unit of $\cat{C}^{\iota_* Z}$ is given by $J = \bigoplus_{g\in G} I^g$, so it is not simple for $|G|\ge 2$. As a semisimple braided monoidal category, $\cat{C}^{\iota_* Z}$ decomposes into semisimple braided monoidal categories with simple unit, see \cite[Lemma~5.3]{BDSPV153D}, but this decomposition is not preserved by the $G$-action. 
	
	The twisted sectors of $\cat{C}^{\iota_* Z}$ are allowed to be trivial because the argument given in the proof of Proposition~\ref{satztwistedsectorsnontrivial} fails. More precisely, in contrast to the proof we find $Z(\torus^2)(g,1)=0$ for $g \neq 1$ because $\cat{C}^{\iota_* Z}$ has no simple objects invariant (up to isomorphism) under $g$. 
	\end{example}

	 We have seen in Proposition~\ref{satztwistedsectorsnontrivial} that it is important to know whether the unit of the equivariant monoidal category coming from an equivariant topological field theory is simple. The situation is under control for those theories arising from our pushforward construction:

	\begin{proposition}[]
		Let $\lambda : G \to H$ be a morphism of finite groups and $Z: G\text{-}\Cob(3,2,1) \to \TwoVect$ an extended $G$-equivariant topological field theory such that the monoidal unit $I\in\cat{C}^Z$ is simple. The monoidal unit in the category $\cat{C}^{\lambda_* Z}$ associated to the pushforward \begin{align}\lambda_* Z: H\text{-}\Cob(3,2,1) \to \TwoVect\end{align} of $Z$ along $\lambda$ in the sense of Definition~\ref{defpushforward} has the endomorphism space $\mathbb{C}^{|H/\img \lambda|}$. In particular, the unit of $\cat{C}^{\lambda_* Z}$ is simple if and only if $\lambda$ is surjective. 
	\end{proposition}
	
	\begin{proof}
		The group morphism $\lambda$ induces a functor $\lambda_* : G//G \to H//H$ for the groupoids of $G$-bundles and $H$-bundles over the circle, respectively. 
		An easy computation shows that the homotopy fiber over $1 \in H$ is given by $(\ker \lambda \times H) //G$, where $G$ acts on $\ker \lambda \times H$ by
		\begin{align} a.(g,h) = \left(aga^{-1},h\lambda\left(a^{-1}\right)\right) \myforall a \in G, \quad g \in \ker \lambda, \quad \quad h \in H.\end{align}
		By the definition of the pushforward, $\cat{C}^{\lambda_* Z}_1$ is the 2-vector space of parallel sections of the 2-vector bundle obtained by pullback of $\cat{C}^Z : G//G \to \TwoVect$ along the projection $(\ker \lambda \times H) //G \to G//G$. The evaluation of $\lambda _* Z$ on the disk decorated with the trivial $H$-bundle yields a map $\FinVect \to \cat{C}^{\lambda_* Z}_1$ whose image on $\mathbb{C}$ is the monoidal unit  $J$ of $\cat{C}^{\lambda_* Z}$. 
		Again, by the definition of the pushforward, this map $\FinVect \to \cat{C}^{\lambda_* Z}_1$ and its image on $\mathbb{C}$ are computed as follows: The morphism $\lambda$ induces the functor $\star //G \to \star //H$ for the $G$-bundles and $H$-bundles over the disk, respectively. Its homotopy fiber over $\star$ is given by $(\{1\}\times H) //G$. By restriction to the boundary, this groupoid embeds into the homotopy fiber $(\ker \lambda \times H) //G$ that we computed for the circle. Denote by $\iota : (\{1\}\times H) //G \to (\ker \lambda \times H) //G$ the embedding. Now the monoidal unit $J \in \cat{C}^{\lambda_* Z}_1$ is the parallel section given on $(g,h) \in \ker \lambda \times H$ by
		\begin{align} J(g,h) = \lim_{\iota^{-1}[g,h]} I.\end{align}
	This parallel section  is supported on $\{1\} \times H$, where it has constant value $I$. Since $H$ acts on $\cat{C}^{\lambda_* Z}$ by linear functors, we see that the endomorphism space of $J$ of is given by $\mathbb{C}^{|    \pi_0 ((\{1\} \times H) //G )|}=\mathbb{C}^{|H/\img \lambda|}$.      
	\end{proof}

The right hand side of \eqref{eqeqverlindealgg} makes sense for any $G$-ribbon category (regardless of whether it comes from an equivariant topological field theory) and inspires the following definition:

\begin{definition}[Equivariant modularity, after \cite{kirrilovg04}]\label{defmultimodular}
Let $G$ be a finite group and $\cat{C}$ a complex finitely semisimple $G$-equivariant ribbon category. We define as in \cite[Section~8]{kirrilovg04} 
		\begin{align}
		\widetilde{\cat{V}}  (\cat{C}) _{g,h} := \int^{X \in \cat{C}_h} \Hom_{\cat{C}_h}(g.X ,X)
		\end{align} and (using this)
		 the
 \emph{equivariant Verlinde algebra} 
	\begin{align}
	\widetilde{\cat{V}}  (\cat{C}) := \bigoplus_{  \substack{ g,h\in G \\ gh=hg  }  } 	\widetilde{\cat{V}}  (\cat{C}) _{g,h}\ .
	\end{align} 
	For $g,h \in G$ with $gh=hg$, $X\in \cat{C}_h$, $Y \in \cat{C}_g$ and a morphism $\varphi : g.X \to X$ we define the morphism $\widetilde s (\varphi) : Y \to h.Y$ as
		\begin{center}
		\vspace*{0.5cm}
		\includegraphics[width=0.25\textwidth]{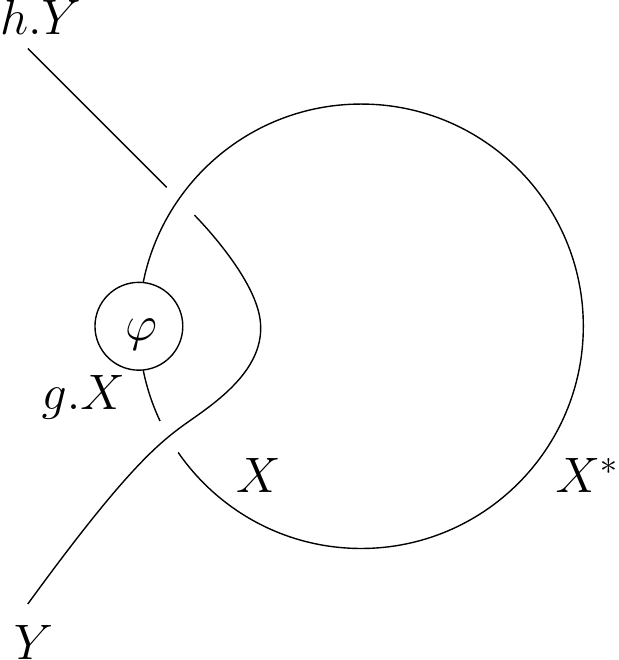}
		\vspace*{0.5cm} .
	\end{center}
This assignment induces a linear map
\begin{align}
\widetilde s : \widetilde{\cat{V}}  (\cat{C}) _{g,h} = \int^{X \in \cat{C}_h} \Hom_{\cat{C}_h}(g.X ,X) \to \int^{Y \in \cat{C}_g} \Hom_{\cat{C}_g}(Y ,h.Y) \cong \widetilde{\cat{V}}  (\cat{C}) _{h^{-1},g}\ .
\end{align} 
We denote the induced map $\widetilde{\cat{V}}  (\cat{C}) \to \widetilde{\cat{V}}  (\cat{C})$ also by $\widetilde s$.
We call the complex finitely semisimple $G$-equivariant ribbon category $\cat{C}$ a \emph{$G$-multimodular category} if the map $\widetilde s : \widetilde{\cat{V}}  (\cat{C}) \to \widetilde{\cat{V}}  (\cat{C})$ is invertible.
A \emph{$G$-modular category} is a $G$-multimodular category with simple monoidal unit.
\end{definition}

\begin{remarks}\label{bmkdefequivmod}
\item The name equivariant Verlinde \emph{algebra} is also justified in the purely algebraic case because $\widetilde{\cat{V}}  (\cat{C})$ comes with a multiplication, see \cite[Section~8]{kirrilovg04}, which is in accordance with the multiplication provided by Proposition~\ref{satzmodfunctorontorus} in the case where our category comes from a topological field theory.

\item A $\{1\}$-multimodular category is just a modular category without the requirement that the unit is simple. However by \cite[Lemma~5.3]{BDSPV153D}, such a category decomposes into a sum of modular categories. For $G \neq \{1\}$ such a decomposition need not be possible, see Example~\ref{exequivcatnonsplit}, so the simplicity of the unit is an important requirement for equivariant categories. 

\item In \cite{hrt} a $G$-modular category is defined to be a  complex finitely semisimple $G$-equivariant ribbon category with simple unit such that the twisted sectors are non-trivial and the neutral sector is modular. This notion of $G$-modularity turns out to be equivalent to the one defined above as follows from a result by Müger in \cite[Appendix~5, Theorem~4.1 (ii)]{turaevhqft}, see also \cite{mueger}, and the characterization of $G$-modularity as defined above in terms of the orbifold theory given in \cite{kirrilovg04} and recalled as Theorem~\ref{thmalgebraicorbifoldization} below.\label{bmkdefequivmod3}

\end{remarks}

 Now we can prove the main results of this subsection, namely the equivariant modularity of the category $\cat{C}^Z$ that a 3-2-1-dimensional $G$-equivariant topological field theory $Z$ yields on the circle. We will have two versions of the result depending on whether the unit in $\cat{C}^Z$ is simple. The proofs will be totally independent.

If the unit of $\cat{C}^Z$ is simple, then we will prove that $\cat{C}^Z$ is $G$-modular. The method of proof demonstrates that the geometric orbifold construction provides a link between the purely algebraic understanding of equivariant modular categories in \cite{kirrilovg04} to the topological results of \cite{BDSPV153D}. To this end, we use that the notion of equivariant modularity is completely governed by the following strong algebraic result from \cite{kirrilovg04} that we slightly rephrase:

\begin{theorem}[\text{\cite[Theorem~10.5]{kirrilovg04}}]\label{thmalgebraicorbifoldization}
	Let $G$ be a finite group. For any complex finitely semisimple $G$-equivariant ribbon category $\cat{C}$ the orbifold category $\cat{C}/G$ naturally inherits by Proposition~\ref{satzalgofk} the structure of a complex finitely semisimple ribbon category and
	\begin{align}
		\cat{C} \ \text{is $G$-modular} \quad \Longleftrightarrow \quad \cat{C}/G \ \text{is modular}.
	\end{align}
\end{theorem}

\begin{theorem}[]\label{thmgmodcatoncircle}
	Let $G$ be a finite group.	For any extended $G$-equivariant topological field theory $Z $ the category $\cat{C}^Z$ obtained by evaluation on the circle  is 
	\begin{myenumerate}
	\item 
	$G$-modular if its monoidal unit is simple,\label{thmgmodcatoncirclea}
	
	\item
	and in the general case still $G$-multimodular.\label{thmgmodcatoncircleb}
	\end{myenumerate}
\end{theorem}

\begin{proof}
	If the unit of $\cat{C}^Z$ is simple, the monoidal unit of $Z/G(\sphere^1)$ is simple as well by Proposition~\ref{satzdescriptorbtensorstructure}. Now Theorem~\ref{thmorbifoldtheorymodular} yields an equivalence  
	\begin{align} \frac{Z}{G}(\sphere^1) \cong \frac{\cat{C}^Z}{G}\end{align} of complex finitely semisimple ribbon categories. But by \cite{BDSPV153D} the category $Z/G(\sphere^1)$ is even modular, hence so is $\cat{C}^Z /G$. Now Theorem~\ref{thmalgebraicorbifoldization} implies that $\cat{C}^Z$ is $G$-modular.  This proves \ref{thmgmodcatoncirclea}.

	For the proof of \ref{thmgmodcatoncircleb},
		by Proposition~\ref{satzgribbonkat} we only have to show that the operator $\widetilde s : \widetilde{\cat{V}}  (\cat{C}^Z) \to \widetilde{\cat{V}}  (\cat{C}^Z) $ is invertible. In the non-equivariant case this follows from the fact that $\widetilde s$ is obtained by evaluation on $Z$ on an invertible 2-morphism $\torus^2 \Longrightarrow \torus^2$ in the bordism bicategory, see \cite[Section~5.3]{BDSPV153D} for a detailed discussion. A straightforward generalization to the equivariant case yields an invertible 2-morphism from $\torus^2$ with bundle decoration $(g,h)$ for $g,h \in G$ with $gh=hg$ to $\torus^2$ with bundle decoration $(h^{-1},g)$, compare to the proof of Proposition~\ref{satztwistedsectorsnontrivial}. By evaluation of $Z$ on this 2-morphism we see that the map $\widetilde s : \widetilde{\cat{V}}  (\cat{C}) _{g,h} \to \widetilde{\cat{V}}  (\cat{C}) _{h^{-1},g}$ is invertible. But then $\widetilde s : \widetilde{\cat{V}}  (\cat{C}^Z) \to \widetilde{\cat{V}}  (\cat{C}^Z) $ is also invertible.
\end{proof}

\spaceplease

\begin{remark}
	We can give another proof of Theorem~\ref{thmgmodcatoncircle} \ref{thmgmodcatoncirclea}: By Remark~\ref{bmkdefequivmod}, \ref{bmkdefequivmod3} it suffices to show the following two things:
	\begin{itemize}
		\item  The neutral sector of $\cat{C}^Z$ is modular: This follows from the fact that we can pull $Z$ back along the symmetric monoidal functor $\Cob(3,2,1) \to G\text{-}\Cob(3,2,1)$ equipping all manifolds with the trivial $G$-bundle. This yields an ordinary extended topological field theory whose value on the circle is $\cat{C}_1^Z$, which is a modular category by \cite{BDSPV153D}.
		
		\item The twisted sectors of $\cat{C}^Z$ are non-trivial: This was proven directly in Proposition~\ref{satztwistedsectorsnontrivial} based on modular invariance. 
		
		\end{itemize}
	\end{remark}

 Note that \ref{thmgmodcatoncircleb} generalizes \ref{thmgmodcatoncirclea} if we take the statement in Proposition~\ref{satzdescriptorbtensorstructure} on the simplicity of the units into account.

\begin{remark}\label{remclass}
For a finite group $G$ there are two main constructions for three-dimensional $G$-equivariant topological field theories due to Turaev and Virelizier:

\begin{itemize}

\item The \emph{homotopy Turaev-Viro construction} \cite{htv} takes as an input a spherical $G$-fusion category $\cat{S}$ and yields the $G$-equivariant Turaev-Viro theory $\operatorname{TV}^G_\cat{S}$,

\item The \emph{homotopy Reshetikhin-Turaev construction} \cite{hrt} takes as an input an (anomaly-free) $G$-modular category $\cat{C}$ and yields the $G$-equivariant Reshetikhin-Turaev theory $\operatorname{RT}^G_\cat{C}$.

\end{itemize}
Both constructions are equivariant generalizations of the famous non-equivariant constructions, but so far only cover the non-extended case. However, it is likely that both types of theories admit extensions to 3-2-1-theories. 
For the following considerations in this remark we will assume that
\begin{itemize}
	\item the homotopy Reshetikhin-Turaev construction can be generalized to give extended homotopy quantum field theories in the sense of this article such that the value of $\operatorname{RT}^G_\cat{C}$ on the circle is $\cat{C}$,
	
	\item the homotopy Turaev-Viro construction can also be generalized to give extended homotopy quantum field theories and the evaluation of $\operatorname{TV}^G_\cat{S}$ on the circle will be given by the $G$-center $Z_G(\cat{S})$ of $\cat{S}$ according to the conjecture
	\begin{align}\operatorname{TV}^G_\cat{S} \cong \operatorname{RT}^G_{Z_G(\cat{S})}\label{tvconjecture} \end{align} made in
	\cite{htv} (as a generalization of the non-equivariant case) that we would also have to be able to interpret on the level of extended field theories.
	
	\end{itemize}
Under these assumptions,
 we can compute the orbifold theories of $\operatorname{RT}^G_\cat{C}$ and $\operatorname{TV}^G_\cat{S}$ for a $G$-modular category $\cat{C}$ and a $G$-fusion category $\cat{S}$:
By Theorem~\ref{thmorbifoldtheorymodular} the orbifold theory $\operatorname{RT}^G_\cat{C}/G: \Cob(3,2,1) \to \TwoVect$ of $\operatorname{RT}^G_\cat{C}$ is the Reshetikhin-Turaev theory for the orbifold category $\cat{C}/G$, i.e.\
\begin{align}
\frac{\operatorname{RT}^G_\cat{C}}{G} \cong \operatorname{RT}_{\cat{C}/G}. \label{rtorbifold}
\end{align}
For $\operatorname{TV}^G_\cat{S}$ we find
\begin{align}
\frac{\operatorname{TV}^G_\cat{S}}{G}(\sphere^1) \overeqn{tvconjecture}{\cong}  \frac{\operatorname{RT}^G_{Z_G(\cat{S})}}{G} (\sphere^1) \overeqn{rtorbifold}{\cong} \frac{Z_G(\cat{S})}{G} \cong Z(\cat{S})
\end{align} 
as modular categories, where in the last step we used \cite[Theorem~3.5]{centerofgradedfusioncategories}. Hence, the orbifold theory $\operatorname{TV}^G_\cat{S}/G$ is just the non-equivariant Turaev-Viro theory for $\cat{S}$ seen as spherical fusion category (recall that a $G$-fusion category is fusion if and only if $G$ is finite, see \cite[Section~4.2]{htv}). Hence, on the level of spherical fusion categories, orbifoldization amounts to forgetting equivariance.

Furthermore, we remark that a generalization of $\operatorname{RT}^G_?$ taking $G$-multimodular categories as input should provide a weak inverse to the functor from $G$-equivariant 3-2-1-dimensional topological field theories to $G$-multimodular categories by evaluation on the circle, see Theorem~\ref{thmgmodcatoncircle} (when restricting to the anomaly-free case). Hence, $G$-equivariant 3-2-1-dimensional topological field theories should be classified by (anomaly-free) $G$-multimodular categories.

\end{remark}

\end{document}